\documentclass[10pt,a4paper]{amsart} 

\usepackage{amsmath,amsthm,amssymb,amsfonts}
\usepackage[initials,msc-links]{amsrefs}
\usepackage{url}
\usepackage[dvipsnames]{xcolor}
\usepackage[english=american]{csquotes}
\usepackage{enumerate}
\usepackage{mathrsfs}
\usepackage{mathtools}
\usepackage{amsaddr}
\usepackage{bbm}
\usepackage{bm}
\usepackage{caption} 
\usepackage[multiple]{footmisc}
\usepackage{hyperref}
\hypersetup{
	bookmarks=true,         
	unicode=false,          
	pdftoolbar=true,        
	pdfmenubar=true,        
	pdffitwindow=false,     
	pdfstartview={FitH},    
	pdftitle={My title},    
	pdfauthor={Author},     
	pdfsubject={Subject},   
	pdfcreator={Creator},   
	pdfproducer={Producer}, 
	pdfkeywords={keyword1, key2, key3}, 
	pdfnewwindow=true,      
	colorlinks=true,       
	linkcolor=blue ,          
	citecolor=blue ,        
	filecolor=magenta,      
	urlcolor=Aquamarine           
}
\usepackage{caption}
\usepackage{tikz}

\usetikzlibrary{patterns}

\newlength{\widehatchspread}
\newlength{\widehatchthickness}
\newlength{\widehatchshift}
\newcommand{\widehatchcolor}{}
\tikzset{hatchspread/.code={\setlength{\widehatchspread}{#1}},
	hatchthickness/.code={\setlength{\widehatchthickness}{#1}},
	hatchshift/.code={\setlength{\widehatchshift}{#1}},
	hatchcolor/.code={\renewcommand{\widehatchcolor}{#1}}}
\tikzset{hatchspread=3pt,
	hatchthickness=0.4pt,
	hatchshift=0pt,
	hatchcolor=black}
\pgfdeclarepatternformonly[\widehatchspread,\widehatchthickness,\widehatchshift,\widehatchcolor]
{custom north west lines}
{\pgfqpoint{\dimexpr-2\widehatchthickness}{\dimexpr-2\widehatchthickness}}
{\pgfqpoint{\dimexpr\widehatchspread+2\widehatchthickness}{\dimexpr\widehatchspread+2\widehatchthickness}}
{\pgfqpoint{\dimexpr\widehatchspread}{\dimexpr\widehatchspread}}
{
	\pgfsetlinewidth{\widehatchthickness}
	\pgfpathmoveto{\pgfqpoint{0pt}{\dimexpr\widehatchspread+\widehatchshift}}
	\pgfpathlineto{\pgfqpoint{\dimexpr\widehatchspread+0.15pt+\widehatchshift}{-0.15pt}}
	\ifdim \widehatchshift > 0pt
	\pgfpathmoveto{\pgfqpoint{0pt}{\widehatchshift}}
	\pgfpathlineto{\pgfqpoint{\dimexpr0.15pt+\widehatchshift}{-0.15pt}}
	\fi
	\pgfsetstrokecolor{\widehatchcolor}
	\pgfusepath{stroke}
}

	\definecolor{navyblue}{rgb}{0.25, 0.4, 0.96}
	\definecolor{teal}{rgb}{0.04, 0.73, 0.71}

\makeatletter
\DeclareFontFamily{OMX}{MnSymbolE}{}
\DeclareSymbolFont{MnLargeSymbols}{OMX}{MnSymbolE}{m}{n}
\SetSymbolFont{MnLargeSymbols}{bold}{OMX}{MnSymbolE}{b}{n}
\DeclareFontShape{OMX}{MnSymbolE}{m}{n}{
	<-6>  MnSymbolE5
	<6-7>  MnSymbolE6
	<7-8>  MnSymbolE7
	<8-9>  MnSymbolE8
	<9-10> MnSymbolE9
	<10-12> MnSymbolE10
	<12->   MnSymbolE12
}{}
\DeclareFontShape{OMX}{MnSymbolE}{b}{n}{
	<-6>  MnSymbolE-Bold5
	<6-7>  MnSymbolE-Bold6
	<7-8>  MnSymbolE-Bold7
	<8-9>  MnSymbolE-Bold8
	<9-10> MnSymbolE-Bold9
	<10-12> MnSymbolE-Bold10
	<12->   MnSymbolE-Bold12
}{}
\let\llangle\@undefined
\let\rrangle\@undefined
\DeclareMathDelimiter{\llangle}{\mathopen}%
{MnLargeSymbols}{'164}{MnLargeSymbols}{'164}
\DeclareMathDelimiter{\rrangle}{\mathclose}%
{MnLargeSymbols}{'171}{MnLargeSymbols}{'171}
\makeatother

\newtheorem{theorem}{Theorem}[section]
\newtheorem{lemma}{Lemma}[section]
\newtheorem{proposition}{Proposition}[section]
\newtheorem{corollary}{Corollary}[section]

\theoremstyle{definition}
\newtheorem{definition}{Definition}[section]
\newtheorem{remark}{Remark}[section]
\newtheorem{example}{Example}[section]

\newcommand{\Crm}{\mathrm{C}}

\newcommand{\Grm}{\mathrm{G}}

\newcommand{\Irm}{\mathrm{I}}

\newcommand{\Lrm}{\mathrm{L}}

\newcommand{\Trm}{\mathrm{T}}

\newcommand{\Wrm}{\mathrm{W}}

\newcommand{\Acal}{\mathcal{A}}
\newcommand{\Bcal}{\mathcal{B}}

\newcommand{\Dcal}{\mathcal{D}}
\newcommand{\Ecal}{\mathcal{E}}
\newcommand{\Fcal}{\mathcal{F}}
\newcommand{\Gcal}{\mathcal{G}}
\newcommand{\Hcal}{\mathcal{H}}

\newcommand{\Lcal}{\mathcal{L}}
\newcommand{\Mcal}{\mathcal{M}}

\newcommand{\Ocal}{\mathcal{O}}
\newcommand{\Pcal}{\mathcal{P}}
\newcommand{\Qcal}{\mathcal{Q}}
\newcommand{\Rcal}{\mathcal{R}}
\newcommand{\Scal}{\mathcal{S}}

\newcommand{\Wcal}{\mathcal{W}}
\newcommand{\Xcal}{\mathcal{X}}

\newcommand{\Ffrak}{\mathfrak{F}}

\newcommand{\Ebf}{\mathbf{E}}

\newcommand{\Ibf}{\mathbf{I}}

\newcommand{\Nbf}{\mathbf{N}}

\newcommand{\Ybf}{\mathbf{Y}}

\newcommand{\Abb}{\mathbb{A}}
\renewcommand{\Bbb}{\mathbb{B}}

\newcommand{\Nbb}{\mathbb{N}}

\newcommand{\Sbb}{\mathbb{S}}
\newcommand{\Tbb}{\mathbb{T}}

\newcommand{\Xbb}{\mathbb{X}}

\newcommand{\Zbb}{\mathbb{Z}}

\DeclareMathOperator{\id}{id}

\DeclareMathOperator{\im}{Im}

\DeclareMathOperator*{\wslim}{w*-lim}

\DeclareMathOperator{\dist}{dist}
\DeclareMathOperator{\rank}{rank}

\DeclareMathOperator{\spn}{span}

\DeclareMathOperator{\supp}{supp}

\newcommand{\ac}[1]{{#1}^\mathrm{ac}}

\newcommand{\area}[1]{\mathrm{Area}(#1)}

\newcommand{\set}[2]{\left\{\, #1 \  \textup{\textbf{:}}\  #2 \,\right\}}
\newcommand{\setn}[2]{\{\, #1 \  \textup{\textbf{:}}\  #2 \,\}}
\newcommand{\setb}[2]{\bigl\{\, #1 \  \textup{\textbf{:}}\  #2 \,\bigr\}}
\newcommand{\setB}[2]{\Bigl\{\, #1 \  \textup{\textbf{:}}\  #2 \,\Bigr\}}
\newcommand{\setBB}[2]{\biggl\{\, #1 \  \textup{\textbf{:}}\  #2 \,\biggr\}}

\newcommand{\dpr}[1]{\langle #1 \rangle}

\newcommand{\dprb}[1]{\bigl\langle #1 \bigr\rangle}

\newcommand{\ddprb}[1]{\llangle#1\rrangle}
\newcommand{\ddprB}[1]{\Bigl\llangle #1 \Bigr\rrangle}

\newcommand{\cl}[1]{\overline{#1}}
\newcommand{\di}{\mathrm{d}}
\newcommand{\dd}{\;\mathrm{d}}

\newcommand{\N}{\mathbb{N}}
\newcommand{\R}{\mathbb{R}}

\newcommand{\Q}{\mathbb{Q}}

\newcommand{\loc}{\mathrm{loc}}
\newcommand{\sym}{\mathrm{sym}}

\newcommand{\per}{\mathrm{per}}
\newcommand{\reg}{\mathrm{reg}}
\newcommand{\sing}{\mathrm{sing}}

\newcommand{\toweak}{\rightharpoonup}
\newcommand{\toweakstar}{\overset{*}\rightharpoonup}

\newcommand{\toup}{\uparrow}
\newcommand{\todown}{\downarrow}

\newcommand{\embed}{\hookrightarrow}
\newcommand{\cembed}{\overset{c}{\embed}}

\newcommand{\BigO}{\mathrm{\textup{O}}}

\newcommand{\sbullet}{\begin{picture}(1,1)(-0.5,-2)\circle*{2}\end{picture}}
\newcommand{\frarg}{\,\sbullet\,}
\newcommand{\BV}{\mathrm{BV}}

\newcommand{\toY}{\overset{\Ybf}{\to}}

\newcommand{\eps}{\epsilon}

\DeclareMathOperator{\Tan}{Tan}

\newcommand{\proofstep}[1]{\textit{#1}}

\newcommand{\Leb}{\mathscr L}

\renewcommand{\eps}{\varepsilon}
\renewcommand{\phi}{\varphi}

\newcommand{\M}{\mathcal M}

\makeatletter
\renewcommand*\env@matrix[1][*\c@MaxMatrixCols c]{%
	\hskip -\arraycolsep
	\let\@ifnextchar\new@ifnextchar
	\array{#1}}
\makeatother

\DeclareMathOperator{\Y}{\bf Y}

\DeclareMathOperator{\E}{\bf E}
\DeclareMathOperator{\e}{\bf E}
\DeclareMathOperator{\A}{\mathcal A}
\DeclareMathOperator{\B}{\mathcal B}

\DeclareMathOperator{\tr}{Tr}

\DeclareMathOperator{\Div}{div}

\newcommand{\mres}{\mathbin{\vrule height 1.6ex depth 0pt width
		0.13ex\vrule height 0.13ex depth 0pt width 1.3ex}}

\newcommand{\aveint}[2]{\mathchoice%
	{\mathop{\kern 0.2em\vrule width 0.6em height 0.69678ex depth -0.58065ex
			\kern -0.8em \intop}\nolimits_{\kern -0.45em#1}^{#2}}%
	{\mathop{\kern 0.1em\vrule width 0.5em height 0.69678ex depth -0.60387ex
			\kern -0.6em \intop}\nolimits_{#1}^{#2}}%
	{\mathop{\kern 0.1em\vrule width 0.5em height 0.69678ex depth -0.60387ex
			\kern -0.6em \intop}\nolimits_{#1}^{#2}}%
	{\mathop{\kern 0.1em\vrule width 0.5em height 0.69678ex depth -0.60387ex
			\kern -0.6em \intop}\nolimits_{#1}^{#2}}}

\title[Characterization of $\A$-free Young measures]{Characterization of generalized \\ Young measures generated by \\ $\Acal$-free measures}

\date{\today}

\author[A. Arroyo-Rabasa]{Adolfo Arroyo-Rabasa}
\address{Mathematics Institute, University of Warwick, Coventry CV4 7AL, UK}
\email{\href{mailto:adolforabasa@gmai.com}{adolforabasa@gmail.com},\href{mailto:Adolfo.Arroyo-Rabasa@warwick.ac.uk}{Adolfo.Arroyo-Rabasa@warwick.ac.uk}}


\makeindex

\begin{document}
	\maketitle

\begin{abstract} 
	We give two characterizations, one for the class of  generalized Young measures generated by $\A$-free measures  and one for the class  generated by $\Bcal$-gradient measures $\Bcal u$.
%
	Here, $\A$ and $\Bcal$ are linear homogeneous operators of arbitrary order, which we assume satisfy the   constant rank property.
The characterization places the class of generalized $\Acal$-free Young measures  in duality with the class of $\A$-quasiconvex integrands by means of a well-known  Hahn--Banach separation property. A similar statement holds for  generalized $\Bcal$-gradient Young measures. Concerning applications, we discuss several examples that showcase the rigidity or the failure of $\Lrm^1$-compensated compactness when concentration of mass is allowed. These include the failure of $\Lrm^1$-estimates for elliptic systems and the failure of $\Lrm^1$-rigidity for the two-state problem. As a byproduct of our techniques we also show that, for any bounded open set $\Omega$, the inclusions 
\begin{align*}
	\Lrm^1(\Omega) \cap \ker \Acal  & \hookrightarrow \M(\Omega) \cap \ker \A\,,\\
		\{\Bcal u\in \Crm^\infty(\Omega)\} & \hookrightarrow \{\Bcal u\in \Mcal(\Omega)\},
\end{align*}
	are dense with respect to the area-functional convergence of measures.  	
\end{abstract}

{
	\tableofcontents
}

\section{Introduction}

The last decades have witnessed an extensive  development of 
the study of {non-convex} variational energies related to equilibrium configurations of materials  in a wide range of physical models (such as the study of crystalline solids and thermoelastic materials, linear elasticity, perfect plasticity, micro-magnetics, and ferro-magnetics, among others~\cite{BallJames1987,chipot1988,desinone1993,james1992}). Often, these models consist in a minimization principle  for integrals of the form
\begin{equation}\label{eq:I}
w \mapsto I_f(w) \coloneqq \int_\Omega f(x,w(x)) \dd x,  
\end{equation}
where $\Omega \subset \R^d$ is an open and bounded set, $f : \Omega \times \R^N \to \R$ satisfies a uniform $p$-growth condition $|f(z)| \lesssim 1 + |z|^p$, and the  configurations $w : \Omega  \to \R^N$ obey a set of physical laws determined by a system of linear PDEs, where, depending on the particular model, either 
\begin{align}\label{eq:gov}
\Acal w & = 0  \!\!\!\!\!\!\!\!\!\!\!\!&&\text{in the sense of distributions on $\Omega$, or}\\
      w &= \Bcal u  \!\!\!\!\!\!\!\!\!\!\!\!&&\text{for some potential $u : \Omega \to \R^M$}.\label{eq:gov2}
\end{align}
We shall refer to the first scenario as the \emph{$\Acal$-free} framework and to the latter as the \emph{potential} or \emph{$\Bcal$-gradient} framework. In order to keep the exposition as simple as possible, we shall henceforth adopt the $\Acal$-free perspective.

In these circumstances, designs with near to minimal energy exhibit compatible equilibrium behavior at microscopical scales, while, at larger scales, configurations adapt by \emph{gluing} together the low energy patterns allowed by the governing equations in~\eqref{eq:gov}/\eqref{eq:gov2}. This interplay conveys the formation of finer and finer oscillations,  often resulting in some form of $\Lrm^p$-weak convergence $w_j \toweak w$ when $p > 1$, or {weak-$*$} convergence (in the sense of measures) 
when $p =1$ \cite{ambrosio1992-On_the_relaxation,baia2013lower-semiconti,fonseca2004mathcal-a-quasi,fonseca1993relaxation-of-q,kristensen2010relaxation-of-s,kristensen2010characterizatio,barroso2000a-relaxation-th,arroyo-rabasa2017lower-semiconti,rindler2011lower-semiconti,rindler2012lower-semiconti}.  In general, such weak forms of convergence are incompatible with the lower semicontinuity of the energy, which is usually the starting point for minimization principles. 
 Additionally, the case $p = 1$ may be ill-posed in the sense that, independently of the PDE-constraint,  a solution to the minimization problem may fail to exist. The reason is that $\Lrm^1$ is not reflexive and it naturally lacks of compactness properties that guarantee the existence of minimizers. To solve this, one relaxes the variational setting~\eqref{eq:I}-\eqref{eq:gov}/\eqref{eq:I}-\eqref{eq:gov2} to the minimization of  the extended energy functional  
\begin{equation}\label{eq:cI}
	\mu \mapsto \cl{I_f}(\mu) \coloneqq \int_{\Omega} f(x,\ac \mu(x)) \dd x + \int_{\cl \Omega} f^\infty \bigg(\frac{\mathrm d \mu}{\mathrm d|\mu|}(x)\bigg)\dd |\mu^s|,
\end{equation}
defined for measure-valued configurations $\mu \in \M(\Omega;\R^N) \cap \ker \A$ (or $\mu = \Bcal u$ for some potential $u \in \Mcal(\Omega;\R^M)$). Here, $f^\infty$ is a regularization at infinity called the \emph{strong recession function} of $f$, which is defined (provided that it exists) as
\begin{equation}\label{eq:res}
f^\infty(x,z) \coloneqq \lim_{\substack{x' \to x\\z' \to z\\t \to \infty}} \frac{f(x',tz')}{t} \qquad \text{for all $x \in \cl \Omega$ and $z  \in \R^N$}.
\end{equation}

In this paper we focus on the case $p = 1$, which requires a careful study of \emph{oscillations} and \emph{concentrations} occurring along weak-$*$ convergent sequences of measures satisfying~\eqref{eq:gov}/\eqref{eq:gov2}. 
In this regard, an equivalent approach towards the understanding of~\eqref{eq:I}-\eqref{eq:gov} consists of characterizing all {generalized} Young measures (see~\cite{diperna1987oscillations-an}) generated by sequences $\{\mu_j\} \subset \M(\Omega;\R^N) \cap \ker \A$. 
Let us recall that, formally, a generalized Young measure associated to a sequence $\{\mu_j\} \subset \Mcal(\Omega;\R^N)$ is a triple $\bm \nu = (\nu_x,\lambda,\nu^\infty_x)_{x \in \cl \Omega}$ conformed by a non-negative measure $\lambda \in \Mcal^+(\Omega)$ and two families $\{\nu_x\}$, $\{\nu_x^\infty\}$ of probability measures over the target space $\R^N$, satisfying the fundamental property that
%
%
\begin{equation*}
\begin{split}
\cl{ I_f}(\mu_j) \quad \longrightarrow  \quad & \ddprB{f,\bm \nu}\coloneqq \int_{\Omega} \underbrace{\bigg(\int_{\R^N} f(x,z) \dd \nu_x(z) \bigg)}_{\text{oscillation part}} \dd x \\
& \quad\qquad \qquad +  \int_{\cl\Omega}\underbrace{\bigg( \int_{\R^{N}} f^\infty(x,z) \dd \nu_x^\infty(z) \bigg)}_{\text{concentration part}} \dd \lambda(x),
\end{split}
\end{equation*}
for all sufficiently regular integrands $f : \Omega \times \R^N \to \R$ with linear growth at infinity. 

The main result of this paper is contained in~Theorem~\ref{thm:char} and states that a generalized Young measure $\bm \nu$, with zero boundary-values $\lambda(\partial \Omega) = 0$, is generated by a sequence $\Acal$-free measures if and only if (see Definition~\ref{def:qc})
\[
	f\left(\ddprB{\id_{\R^N},\bm \nu}\right) \le \ddprB{f,\bm \nu} \quad \text{for all  $\Acal$-quasiconvex integrands $f:\Omega \times \R^N \to \R$}\,.
\]
This separation result implies that the class of generalized $\Acal$-free Young measures is a convex set \emph{characterized by duality} in terms of all $\Acal$-quasiconvex integrands.  In addition to this duality characterization, we give a characterization in terms of the \emph{blow-up properties} of generalized Young measures (see Theorem~\ref{thm:local}). More precisely, we prove that $\bm \nu$ as above is generated by a sequence of $\Acal$-free measures if and only if its tangent cones $\Tan(\bm \nu,x)$ {almost always} contain a generalized Young measure that is generated by $\Acal$-free measures.
Lastly, in Theorem~\ref{lem:app}, we establish the following approximation result: if $\mu \in \Mcal(\Omega;\R^N)$ is a bounded $\Acal$-free measure, then there exists a sequence of $\Acal$-free functions $\{w_j\} \subset \Lrm^1(\Omega;\R^N)$ that converges to $\mu$ in the sense of the {generalized area functional}, i.e., 
\[
w_j\Leb^d \toweakstar \mu \; \text{as measures on $\Omega$, and}
\]
\[
\int_\Omega {\textstyle\sqrt{1 + |w_j|^2}} \dd x \to \int_\Omega \sqrt{1 + |\ac\mu|^2} \dd x + |\mu^s|(\Omega).
\]
We also prove analogous results in the $\Bcal$-potential setting~\eqref{eq:I}-\eqref{eq:gov2}, for generalized measures generated by sequences of the form $\{\Bcal u_j\} \subset \Mcal(\Omega;\R^N)$. These are contained in Theorem~\ref{thm:char2}, Theorem~\ref{thm:local2} and Theorem~\ref{lem:app2}.

\subsection{State of the art} 
The work of \textsc{Young} \cite{young1937generalized-cur,young1942generalized-sur,young1942generalized-surII} and the use of (classical) Young measures plays a fundamental role in representing solutions of optimal control problems. The study of Young measures,  from the point of view of partial differential equations,  started with the work of \textsc{Tartar \& Murat}, who, 
motivated by problems in continuum mechanics and electromagnetism, introduced the theory of {compactness by compensation}~\cite{arroyo-rabasa2017lower-semiconti,murat1978compacite-par-c,tartar1979compensated-com,tartar1983the-compensated,murat1985optimality-cond}. The first characterization of  Young measures in the PDE-constrained context is due to \textsc{Kinderlehrer \& Pedregal}~\cite{kinderlehrer1991characterizatio,kinder1994} for the potential configuration $w = \nabla u$, of a Sobolev function $u \in \Wrm^{1,p}(\Omega;\R^m)$ with $p > 1$. 
This characterization of \emph{$\Lrm^p$-gradient Young measures} accounts for the validity of Jensen's inequality between gradient Young measures and (curl-)quasiconvex integrands.\footnote{The importance of quasiconvexity in the calculus of variations was first observed by Morrey~\cite{Morrey1,Morrey2}, who showed quasiconvexity is a sufficient and necessary condition for the lower semicontinuity of~\eqref{eq:I}-\eqref{eq:gov2}, when $\Bcal = D$ is the gradient operator.} 
More precisely, the authors showed that 
a (purely oscillatory) family of probability distributions $\{\nu_x\}_{x \in \Omega}$ on the space of $m \times d$ matrices $M^{m\times d}$ 
is a  Young measure generated by a $p$-equi-integrable sequence of gradients $\nabla u_j \toweak \nabla u$  if and only if 
\begin{equation}\label{eq:kp}
{	f(\nabla u(x)) \le \int_{M^{m \times d}} f(z) \dd \nu_x(z) \quad \text{at $\Leb^d$-a.e. $x \in \Omega$,}}
\end{equation}
for all quasiconvex integrands $f :M^{m \times d} \to \R$ with $p$-growth at infinity. The characterization also covers the case $p =1$, but only when the generating sequences are assumed to be equi-integrable. The extension of this result to generalized Young measures generated by gradients, which is instead associated to the space $\BV(\Omega;\R^m)$ of functions of bounded variation, is due to \textsc{Kristensen \& Rindler}~\cite{kristensen2010characterizatio}. There, the authors show that a generalized Young measure $\bm \nu = (\nu_x,\lambda,\nu^\infty_x)_{x \in \cl{\Omega}}$ is generated by a sequence of gradient measures if and only if a version of~\eqref{eq:kp} holds for the absolutely continuous part of $\bm \nu$, that is,
\[
	f(\ac D u(x)) \le \int_{M^{m \times d}} f(z) \dd \nu_x(z) + \ac\lambda(x)\int_{M^{m \times d}} f^\infty(z) \dd \nu_x^\infty(z),
\] 
for all quasiconvex integrands $f :M^{m \times d} \to \R$ with linear-growth at infinity,
where $Du = \ac D u \Leb^d + D^s u$.
Somewhat surprisingly, this conveys that the nonlinear moments of the \emph{purely concentration} part $(\lambda^s,\nu_x^\infty)$ of $(\nu,\lambda,\nu^\infty)_{x \in \cl{\Omega}}$ are fully unconstrained. (This is a consequence of Alberti's rank one theorem~\cite{Alberti} and a recent rigidity result for positively homogeneous rank-one convex functions established by \textsc{Kirchheim \& Kristensen}~\cite{kirchheim2016on-rank-one-con}.)

The efforts to establish an $\A$-free variational theory for Young measures initiated with the work of \textsc{Dacorogna}~\cite{dacorogna1982weak-continuity}, who studied $\A$-free functions $w$ that are represented by potentials $w = \Bcal u$ where $\Bcal$ is a suitable first-order operator. However, it was 
 the seminal work of \textsc{Fonseca \& M\"uller} that laid the foundations for an {$\A$-free setting} under the \emph{more general} assumption of $\A$ satisfying the \emph{constant rank property}; see~\eqref{eq:cr} below.
 \footnote{A recent result of \textsc{Raita}~\cite{raitua2019potentials}, crucial to this work, establishes that \textsc{Dacorogna's} assumption and the constant rank assumption are locally equivalent, up to considering $\Bcal$ of higher-order (see Lemma~\ref{lem:raita}).}   
The authors generalized \textsc{Morrey's} notion of quasiconvexity to the $\A$-free setting and showed that the necessary and sufficient condition for the lower semicontinuity of~\eqref{eq:I}-\eqref{eq:gov}, under $p$-growth and $p$-equi-integrability assumptions, was precisely the $\A$-quasiconvexity of the integrand. 
\textsc{Fonseca \& M\"uller} also extended {Kinderlehrer--Pedregal}'s characterization theorem to the $\A$-free setting by showing that a family of probability distributions $\{\nu_x\}_{x \in \Omega} \subset \mathrm{Prob}(\R^N)$ is a Young measure generated by a $p$-equi-integrable sequence of $\A$-free maps $\{w_j\}$   if and only if the following two conditions hold:
\begin{enumerate}[(i)]
	\item there exists $w \in \Lrm^p(\Omega;\R^N)$ such that $\Acal w = 0$ and 
	\[
	w(x)  \equiv \int_{\R^N} z \dd \nu_x(z) \quad \text{as functions in $\Lrm^p(\Omega;\R^N)$,}
	\]
	\item   at $\Leb^d$-almost every $x \in \Omega$, Jensen's inequality  
	\[
	h(w(x)) \le \int_{\R^N}  h(z) \dd \nu_x(z) \dd x, 
	\]
	holds  and all $\A$-quasiconvex integrands $h :  \R^N \to \R$ with $p$-growth at infinity. 
\end{enumerate}
The generalization of this result to generalized Young measures without the $p$-equi-integrability assumption in the range $1 < p < \infty$ was later established by \textsc{Fonseca \& Kru\v{z}\'{\i}k}~\cite{kruzik}. 
In the generalized Young measure framework for $p=1$, the only characterization results are restricted to two well-known potential structures, gradients $\Bcal = D$~(\cite{kristensen2010characterizatio}) and symmetrized gradients  $\Bcal = E$ (\cite{de-philippis2017characterizatio}).\footnote{During the peer revision process of this work, \textsc{Kristensen \& Raita}~\cite[Thm~1.1]{kristensen2019oscillation} have proposed an interesting alternative proof of the characterization of $\Acal$-free generalized  Young measures (a version of Theorem~\ref{thm:char}), which seems to appeal to different methods to the ones presented here.}  The well-established proofs for the case when $\Bcal = D,E$ cited above rely on the strong \emph{rigidity properties} that gradients and symmetric gradients possess. However, such properties are not known to hold for general higher-order operators. Up to now,  the only $\Acal$-free result in the generalized setting was a partial characterization due to \textsc{Ba\'ia, Matias \& Santos}~\cite{baia2013}. There, the authors characterize all generalized Young measures generated by $\A$-free measures under the following somewhat restrictive assumptions: 
	(a) The operator $\A$ is assumed to be of first-order. This implies that its associated principal symbol map $\xi \mapsto \Abb(\xi)$ is a linear map. In turn, this allows for \emph{rigidity} and homogenization-type arguments which unfortunately fail for higher order operators.
	(b) The characterization is restricted to Young measures generated by sequences $\mu_j \toweakstar \mu$, where the limiting measure $\mu$ satisfies the following Morrey-type bound
	\begin{equation*}\label{baia}
	\sup_{r > 0} \frac{|\mu|(B_r(x))}{r^{1 + \alpha}} < \infty \qquad  \text{for some } \alpha > 0.
	\end{equation*}
	This upper-density bound on $\mu$ is in general too restrictive for applications as it rules out $1$-rectifiable measures. For instance, every non-degenerate closed smooth curve $\gamma : [0,1] \to \Gamma \subset \R^d$ defines a  divergence-free measure by setting $\mu = (\dot\gamma/|\dot\gamma|) \, \Hcal^1 \mres \Gamma$.

    The purpose of this work is to give a full characterization of all generalized Young measures generated by $\A$-free measures, as well as a characterization of generalized Young measures generated by $\Bcal$-gradients, for operators $\Acal$ and $\Bcal$ satisfying the constant rank property. Therefore, we extend the aforementioned results into a unified general setting that allows for the appearance of mass  concentrations  in the case $p =1$.
    Our strategy departs from previous ones (even in the case of gradients) in the sense that we do not work with \emph{averaged Young measure} approximations nor rely on the \emph{rigidity} of PDE-constrained measures. Instead, we work with  Lebesgue-point continuity properties and the gluing of local generating sequences at the level of \emph{potentials}; it is for this last point that the \emph{constant rank} property is fundamental because it guarantees Sobolev-type regularity estimates when the kernel of $\Acal$ (or $\Bcal$) is removed.  It is worth to mention that the characterization for $\Acal$-free measures presented here does not deal with characterizations up to the boundary. The main assumption being that a generating sequence $w_j : \Omega \to W$ does not concentrate mass on the boundary $\partial \Omega$. 
    In this regard, the work of \textsc{Ba\'ia, Kr\"{o}mer \& Kru\v{z}\'{\i}k}~\cite{BKK18} addresses the characterization of generalized gradient Young measures up to the boundary; such results for general operators $\Acal$ are  yet to be explored.

\subsection{Comments on the constant rank assumption} It is worthwhile to briefly discuss the role that the constant rank assumption plays for \emph{both} the $\Acal$-free setting and the $\Bcal$-potential setting. On the one hand, potentials allow for localizations of the form $\Bcal u\mapsto \Bcal (\phi u)$. In the case of gradients, these localizations are stable thanks to  
Poincar\'e's inequality $\|u\|_{\Lrm^p} \lesssim \|Du\|_{\Lrm^p}$.
For general $\Bcal$, this type of Poincar\'e estimates only holds after \emph{removing the kernel} of $\Bcal$, i.e., $\|u - \pi_\Bcal u\|_{\Lrm^p} \lesssim \|\Bcal u\|_{\Lrm^p}$, where $\pi_\Bcal$ is the ($\Lrm^2$-)projection onto $\ker \Bcal$. At the time the Fonseca-M\"uller characterization was given, one of the challenges \emph{was} the lack of a potential structure for $\A$-free fields. In this regard, {Fonseca and M\"uller's} strategy  in the $\A$-free setting  departs from the one of {Kinderlehrer and Pedregal}, because 
localizations had to be carried out at the level of the $\A$-free field $w \mapsto \phi w$. In order to handle this localization, the authors relied on the use of the \emph{equivalent} estimate  $\|w - \pi_\Acal w\|_{\Lrm^p} \lesssim \|\Acal w\|_{\Wrm^{-k,p}}$, often referred as the \emph{Fonseca-M\"uller projection} in the calculus of variations (it is, in fact, a Calder\'on--Zygmund-type bound).
%
%
In either case, the constant rank property is a sufficient and {necessary} condition  for the boundedness of both Poincar\'e's and Fonseca--M\"uller's estimates (see~\cite{guerra}). 

While most of the physically relevant applications are modeled by operators that \emph{do} satisfy the constant rank property (see Section~\ref{sec:ex}), the variational theory in the setting~\eqref{eq:I}-\eqref{eq:gov}/\eqref{eq:I}-\eqref{eq:gov2} is not restricted to operators satisfying the constant rank property.
Notably, \textsc{M\"uller}~\cite{muller_diagonal} characterized all (classical) Young measures generated by diagonal gradients. This setting is associated to the operator $\Acal (w_1,w_2) = (\partial_2 w_1,\partial_1 w_2)$, which is one of the simplest examples of an operator that does not satisfy the constant rank property (see~\cite{tartar1979compensated-com}). Other related work about the understanding of PDE-constraints where the constant rank property is not a main assumption include~\cite{de-philippis2017characterizatio,arroyo-rabasa2018dimensional-est,tartar}, and more recently~\cite{new}.
 
\subsection{Set-up and main results} We assume throughout the paper that $U \subset \R^d$ is an open set, and that $\Omega \subset \R^d$ is an open and bounded set satisfying $\Leb^d(\partial \Omega) = 0$. Here, $\Leb^d$ denotes the $d$-dimensional Lebesgue measure.

 We work with a homogeneous partial differential operator  
$\A$ (or $\Bcal$), from $W$ to $X$ (or $V$ to $W$), of the form
\begin{equation}\label{eq:A}
\Acal = \sum_{|\alpha|=k} A_\alpha \partial^\alpha, \quad A_ \alpha \in \mathrm{Lin}(W,X),
\end{equation}
where $W,X$ (and $V$) are finite dimensional inner product euclidean spaces. 	Here $\alpha \in \Nbb_0^d$ is a multi-index with modulus $|\alpha| = \alpha_1 + \dots + \alpha_d$  and  $\partial^\alpha$ represents the distributional derivative $\partial_1^{\alpha_1} \cdots \partial_d^{\alpha_d}$. 
Our main assumption on 
$\A$ (and $\Bcal$) is that it satisfies the following {constant rank property}: there exists a positive integer $r$ such that 
\begin{equation}\label{eq:cr} 
\rank \Abb(\xi) = r  \quad \text{for all $\xi$ in $\R^d \setminus  \{0\}$,}
\end{equation} 
where the tensor-valued $k$-homogeneous polynomial
\[
\Abb(\xi) \coloneqq (2\pi \mathrm i)^k\sum_{|\alpha|=k} A_\alpha\xi^\alpha \, \in \, \mathrm{Lin}(W;X), \qquad \xi \in \R^d,
\]
is the principal symbol associated to the operator $\Acal$. Here, $\xi^\alpha \coloneqq \xi_1^{\alpha_1}\cdots \xi_d^{\alpha_d}$.  We also recall the notion of \emph{wave cone} associated to $\Acal$, which plays
a fundamental role for the study of $\Acal$-free fields, as discussed in the work of \textsc{Murat \& Tartar}~\cite{murat1978compacite-par-c,tartar1979compensated-com,tartar1983the-compensated,murat1985optimality-cond}: 
\[
	\Lambda_\Acal = \bigcup_{\xi \in \R^d \setminus \{0\}} \ker \Abb(\xi) \subset W.
\]
The wave cone contains those Fourier amplitudes along which it is possible to construct highly oscillating
$\A$-free fields. More precisely, $P \in \Lambda_\Acal$ if and only if there
exists $\xi \in \R^d \setminus \{0\}$ such that $\A(P\, \phi(x \cdot \xi)) = 0$  for all $\phi \in \Crm^k(\R)$. 
	
	
	Let us begin our exposition by introducing a few concepts about the theory of generalized Young measures (as introduced in~\cite{diperna1987oscillations-an}, and later extended in~\cite{alibert1997non-uniform-int}):
	
	\begin{definition}[Generalized Young measure]
		A triple $\bm \nu = (\nu,\lambda,\nu^\infty)$ is called a locally {bounded generalized Young measure} on $U$, with values in $W$, provided that  
		\begin{enumerate}
			\item[(i)] $\nu : \Omega \to \mathrm{Prob}(W) : x \mapsto \nu_x$ is a weak-$*$ measurable map,
			\item[(ii)] $\lambda \in \M^+(\cl U)$ is a non-negative  Radon measure on $\cl U$, 
			\item[(iii)] $\nu^\infty : U \to \mathrm{Prob}(S_{W}):x \mapsto \nu_x^\infty$ is a weak-$*$ $\lambda$-measurable map, where $S_{W}$ is the unit sphere in $W$, and
			\item[(iv)] the map $x \mapsto \dpr{|\frarg|,\nu_x}$ belongs to $\Lrm^1_\loc(U)$.
		\end{enumerate}
	If moreover,
	\begin{enumerate}
		\item[(iv)] the map $x \mapsto \dpr{|\frarg|,\nu_x}$ belongs to $\Lrm^1(U)$, and
		\item[(v)] $\lambda$ is a finite measure,
	\end{enumerate}
then we say that $\bm \nu$ is a generalized Young measure.
	We write $\Y_\loc(U;W)$ to denote the set of locally bounded generalized  Young measures, and $\Y(U;W)$ to denote the set of generalized Young measures. 	\end{definition}

\noindent\textbf{Notation.} In the following and when no confusion arises, we will often refer to generalized Young measures simply as Young measures. We will also write 
\[
\Ybf_0(U;W) \coloneqq \set{(\nu,\lambda,\nu^\infty) \in \Ybf(U;W)}{\lambda(\partial \Omega) = 0}.
\]

\begin{definition}\label{def:gen}
	We say that a sequence of measures $\{\mu_j\} \subset \Mcal(U;W)$ generates the Young measure $\bm \nu = (\nu,\lambda,\nu^\infty)\in \Y(U;W)$ if and only if
	\begin{align*}
		\cl{I_f}(\mu_j,U) & \to \int_{U} \dprb{f,\nu} \dd \Leb^d + \int_{\cl U} \dprb{f^\infty,\nu^\infty} \dd \lambda\\
		& \coloneqq \int_{ U} \bigg(\int_{W} f(x,z) \dd \nu_x(z) \bigg) \dd x \\
		&  \qquad +  \int_{\cl U}\bigg( \int_{S_W} f^\infty(x,z) \dd \nu_x^\infty(z) \bigg) \dd \lambda(x), \\
	\end{align*}
for all integrands $f \in \E(U,W)$; see Section~\ref{sec: Young} for the precise definition of $\E(U;W)$. In this case we write
\[
	\mu_j \toY \bm \nu \; \text{on $U$.}
\]
\end{definition}

Next, we incorporate the PDE constraint into the concept of generalized Young measure. Let us recall that $\Wrm^{-k,p}(U) = (\Wrm^{k,p'}_0(U))^*$ for all $1 \le p < \infty$. Here, $p'$ is the dual exponent of $p$ and $\Wrm^{k,p'}_0(U)$ is the closure of $\Crm_c^\infty(U)$ with respect to the $\Wrm^{k,p'}$-norm.

\begin{definition}[Generalized $\Acal$-free Young measure]
	A Young measure $\bm \nu \in \Y(U;W)$ is called a {generalized $\A$-free Young measure} if there exists a   sequence $\{\mu_j\} \subset \M(U;W)$ such that
	\[
	\|\Acal \mu_j\|_{\Wrm^{-k,q}(U)} \, \to \, 0 \quad \text{for some \; $1 < q < \frac{1}{d-1}$,}
	\]
	and   
	\[
		\mu_j \toY\bm \nu \; \text{on $U$.}
	\]
		We write $\Y_\Acal(U)$ to denote the set of such Young measures.
%
\end{definition}

	\subsubsection{Characterization of $\Acal$-free Young measures}
	Let us begin by recalling the definition of $\Acal$-quasiconvexity, which will be a necessary concept to state the main characterization theorem.
	\begin{definition}\label{def:qc}
		A locally bounded Borel integrand  $h:W \to \R$ is called $\Acal$-quasiconvex if
		\[
			h(z) \le \int_{[0,1]^d} h(z + w(y)) \dd y \quad \text{for all $z \in W$},
		\]
		for all periodic fields $w \in \Crm^\infty_\per([0,1]^d;W)$ satisfying
		\[
			\Acal w = 0 \quad \text{and} \quad \int_{[0,1]^d} w \dd y= 0.
		\]
	\end{definition}
	We will also require to define a weaker notion of recession function. For a Borel integrand $h: W\to \R$ with linear growth at infinity, we define its \emph{upper recession function} as
	\begin{align*}
		h^{\#}(z) &:= \limsup_{\substack {z' \to z \\ t \to \infty}} \;\frac{h(tz')}{t}, \quad \text{for all  $z \in W$}.
	\end{align*}
	Differently from $h^\infty$, the upper recession function always exists and defines an upper semicontinuous and positively $1$-homogeneous function on $W$. 

	We are now in position to state our main characterization results. The first result extends the Hahn--Banach-type characterization~\cite[Theorem~4.1]{fonseca1999mathcal-a-quasi} to sequences of weak-$*$ convergent measures:

\begin{theorem}\label{thm:char} 
	Let
	$\bm \nu = (\nu,\lambda,\nu^\infty) \in \Y_0(\Omega;W)$. Then, $\bm \nu$ is a generalized $\Acal$-free Young measure if and only if~\vskip8pt
%
\begin{itemize}\setlength{\itemsep}{8pt}
 	\item[(i)] there exists $\mu \in \M(\Omega;W)$ satisfying
 	\[
 		\Acal \mu = 0 \quad \text{in the sense of distributions on $\Omega$},
 	\] 
 	and 
 	\[
 		\mu = \dprb{\id_W,\nu} \Leb^d + \dprb{\id_W,\nu^\infty}\lambda\,,
 	\]
 	\item[(ii)] at $\Leb^d$-almost every $x \in \Omega$, the Jensen-type inequality 
 	\begin{align*}
 		h(\ac \mu(x)) \le \dprb{h,\nu_x}   + \dprb{h^\#,\nu^\infty_x} \ac \lambda(x)
 	\end{align*}
 	holds for all  $\A$-quasiconvex upper-semicontinuous  integrands $h : W\to \R$ with linear growth at
 	 infinity, and
 	\item[(iii)] at $\lambda^s$-almost every $x \in \Omega$, 
 	\[
 		\supp(\nu_x^\infty) \subset   W_\Acal \coloneqq \spn\{\Lambda_{\Acal}\}.
 	\]
 	\end{itemize}
	
%
\end{theorem}

	\begin{remark}\label{rem:ya} If $\Acal$ is defined in its essential domain, i.e.,   
		\[
			W= W_\Acal = \spn\{\Lambda_{\Acal}\},
		\]
	then the purely singular part $(\lambda^s,\nu^\infty)$ of $\bm \nu$ is unconstrained since then (iii) is equivalent to the trivial set inclusion
	\[
	\supp (\nu_x^\infty) \subset W \qquad \text{$\lambda^s$-a.e.} 
	\]
	In Section~\ref{sec:ex}, we shall revise a few examples of operators that satisfy this property.
\end{remark}

\begin{remark}\label{rem:kk}  The condition at regular points, embodied by property (ii), conveys a similar constraint for the supports of $\nu_x$ and $\nu_x^\infty$ on a set of full $\Leb^d$-measure. The results contained in Corollary~\ref{cor:last} imply that $\nu_x$ is the $\ac\mu(x)$ translation of a probability measure supported on $W_\Acal$, i.e., 
\[
\text{$\supp(\delta_{-\ac\mu(x)} \star \nu_x) \subset W_\Acal$ \qquad for $\Leb^d$-a.e.   $x \in \Omega$.}
\] 
The same corollary also conveys that property (iii) holds $\lambda$-a.e., that is, 
\[		
\text{$\supp(\nu_x^\infty) \subset W_\Acal$ \qquad for $(\ac\lambda \Leb^d)$-a.e.  $x \in \Omega$.}
\]
On the other hand, the property at singular points (iii) is equivalent to the complementary Jensen's inequality  
	\[
	h^\#\bigg(\frac{\mathrm d \mu}{\mathrm d|\mu|}(x)\bigg) \le \int_{S_{W}} h^\#(z) \dd \nu_x^\infty(z) \quad \text{for $\lambda^s$-a.e. $x\in \Omega$}. \tag{iii'}
	\]
	This follows directly from the structure theorem for $\Acal$-free measures~\cite[Theorem~1.1]{de-philippis2016on-the-structur} and the rigidity results established in~\cite{kirchheim2016on-rank-one-con}.
\end{remark}

 Our second result characterizes  generalized $\A$-free Young measures in terms of their tangent cone (in the spirit of~\cite{rindlerlocal}); definitions of tangent Young measures will be postponed to Section~\ref{sec: Young}.

\begin{theorem}\label{thm:local}
	Let $\bm \nu = (\nu,\lambda,\nu^\infty) \in \Y_0(\Omega;W)$. Then, $\bm \nu$ is a generalized $\Acal$-free measure if and only if\vskip8pt
	\begin{itemize}\setlength{\itemsep}{8pt}
		\item[(i)]  there exists $\mu \in \M(\Omega;W)$ satisfying
		\[
		\Acal \mu = 0 \quad \text{in the sense of distributions on $\Omega$},
		\] 
		and 
		\[
		\mu = \dprb{\id_{W},\nu} \, \Leb^d + \dprb{\id_W,\nu^\infty} \lambda\,,
		\]
		\item[(ii)] at $(\Leb^d + \lambda^s)$-almost every $x \in \Omega$, there exists a tangent Young measure 
		\[
		\bm \sigma \in \Tan(\bm \nu,x) \in \Y_\loc(\R^d;W),
		\]
		such that $\bm\sigma \mres U \in \Y_\Acal(U)$ for all open and Lipschitz subsets $U \Subset \R^d$ with $\lambda(\partial U) = 0$.
	\end{itemize}
	\end{theorem}

	We close the characterization of $\A$-free Young measures with an application of the methods developed in this paper, which allows us to re-define $\A$-free measures in terms of a pure $\A$-free constraint: 
	
	\begin{corollary}\label{cor:pure} 
		Let $\bm\nu \in \Y_0(\Omega;W)$. The following are equivalent:
\begin{itemize}
	\item[(i)]  $\bm\nu$ is a generalized $\A$-free Young measure,
	\item[(ii)] $\bm\nu$ is generated by  $\A$-free measures.
\end{itemize}
\end{corollary}

\subsubsection{Area strict density of absolutely continuous $\Acal$-free measures} 
Independently of the characterization of $\Acal$-free Young measures, our methods 
 allow us to show that an  $\A$-free measure $\mu$ defined on an open and bounded domain $U \subset \R^d$ can be approximated in the area-strictly sense of measures, by a sequence of $\A$-free functions. This approximation result is of relevance to certain minimization principles involving the relaxation of functionals of the form
 \[
 u \mapsto \int_{U}  f(x,w(x)), \quad w \in \Lrm^1(U;W), \quad \Acal w = 0,
 \]
 Frequently, it has been \emph{accepted} to impose a geometric assumption on $U$ that guarantees the approximation of $\A$-free measures by $\Lrm^1$-integrable $\A$-free fields in the strict sense of measures (see for instance~\cite{muller1987homogenization-,arroyo-rabasa2017relaxation-and-,arroyo-rabasa2017lower-semiconti}). More precisely, that   $U$ is a {strictly star-shaped domain}, i.e., there exists $x \in U$ such that
 \[
 \cl{(U - x)} \subset  \rho ({U - x}) \quad \text{for all $\rho > 1$.} 
 \]
 The approximation result contained in Theorem~\ref{lem:app} below allows, in particular, to dispense with this assumption on the geometry of $U$. 
 In order to state this result  we need to introduce the following basic concept. The {\it  area functional} of a measure is defined as
 \begin{equation}\label{eq:area}
 \area{\mu,U} \coloneqq \int_{U} \sqrt{1 + |\ac\mu|^2} \dd x + |\mu^s|(U), \qquad \mu \in \M(U;\R^N).
 \end{equation}
 In addition to the well-known weak-$*$ convergence of measures, we say that a sequence $\{\mu_j\}$ converges in \emph{area} to $\mu$ in $\M(\Omega;W)$ if 
 \[\text{$\mu_j \toweakstar \mu$ in $\M(U;W)$} \qquad \text{and} \qquad 
 \area{\mu_j,U} \to \area{\mu,U}.\]
 This notion of convergence turns out to be stronger than the conventional {\it strict convergence} of measures, which requires $|\mu_j|(U) \to |\mu|(U)$. The {usefulness} of this form of convergence rests in the fact that  the functional
 \begin{equation}\label{eq:resh}
 \mu \mapsto \int_U f\bigg(x,\frac{\mathrm d \mu}{\mathrm d \Leb^d}(x)\bigg) \dd x + \int_{U} f^\infty\bigg(x,\frac{\mathrm d  \mu}{\mathrm d  |\mu|}(x)\bigg) \dd |\mu^s|(x)
 \end{equation}
 is area-continuous on $\Mcal(U;W)$ for all integrands $f \in \Crm(U \times W)$ such that the strong recession function $f^\infty$ exists on $\cl U \times W$ (see~\cite[Theorem~5]{kristensen2010characterizatio}).

 We have the following area-convergence approximation result 
  (see Section~\ref{sec: Young} for the definition of elementary Young measures $\bm\delta_\mu$):
 \begin{theorem}\label{lem:app} Let $\Omega \subset \R^d$ be an open and bounded set and let $\mu \in \M(\Omega;W)$ be a bounded $\A$-free measure. Then 
 	there exists a sequence $\{w_j\} \subset \Lrm^1(\Omega;W)$, satisfying 
 	\begin{align*}
 		\Acal w_j = 0 \; \text{in the sense of distributions on $\Omega$,}
 	\end{align*}
 \begin{align*}
 	w_j \, \Leb^d \; &\toweakstar \; \mu &&\text{as measures in $\M(\Omega;W)$},\\
 	w_j \, \Leb^d \; & \toY \; \bm\delta_\mu &&\text{on $\Omega$, }
 \end{align*}
and
\[
	 	\area{w_j \,\Leb^d,\Omega} \; \to \; \area{\mu,\Omega}. 	
\]
 \end{theorem}

\begin{remark}\label{rem:reg} Notice that we do not require $\Omega$ to be Lipschitz nor  $\Leb^d(\partial \Omega) = 0$. The regularity of the recovery sequence $\{w_j\}$ can be lifted to be of class  $\Crm^k(\Omega;W)$ for any $k \in \Nbb$. 
\end{remark}

\subsubsection{Characterization of generalized $\Bcal$-gradient Young measures} In this section we state the characterization results that belong to the potential setting~\eqref{eq:I}-\eqref{eq:gov2}. 

Let $\Bcal$ be a homogeneous linear operator
of arbitrary order, from $V$ to $W$, and assume that $\Bcal$ satisfies the constant rank property~\eqref{eq:cr}. Let us first introduce the notion of $\Bcal$-gradient Young measure:

\begin{definition}
	A  Young measure $\bm \nu \in \Ybf(U;W)$ is called a generalized {$\Bcal$-gradient Young measure} if there exists a sequence $\{u_j\} \subset \Mcal(U;V)$ such that $\{\Bcal u_j\} \subset \Y(U;W)$ and
	\[
	\Bcal u_j \; \toY \; \bm \nu \; \text{in $\Ybf(U;W)$}.
	\]
	We write $\Bcal\Ybf(U)$ to denote the set of these Young measures.
\end{definition}
\begin{remark}
	We \emph{do not} require that the sequence of generating potential measures $\{u_j\}$ is uniformly bounded. However, of course, the sequence $\{\Bcal u_j\}$ must be uniformly bounded since it generates $\bm \nu$.
\end{remark}


Since $\Bcal$ is satisfies the constant rank property, the results in~\cite{raitua2019potentials} (see also Section~\ref{sec:cr}) yield the existence of an annihilator operator for $\Bcal$. More precisely, there exists $\Acal$ from $W$ to $X$ as in~\eqref{eq:A} such that 
\begin{equation}\label{eq:ee}
\im \Bbb(\xi) = \ker \Abb(\xi) \qquad \text{for all $\xi \in \R^d \setminus \{0\}$.}
\end{equation}
A localization argument and an application of the Fourier transform imply that $\Bcal$-gradients are $\Acal$-free fields and therefore, in this case,
\[
\Bcal\!\Y(U) \subset \Y_\Acal(U).
\]
A more interesting question in this context is to understand \emph{how far} is a generalized $\Acal$-free measure from being a generalized $\Bcal$-gradient Young measure. 
The first step to answer this question is to notice that, by a slight modification of the proof of Theorem~\ref{thm:local}, we obtain the following local characterization:
\begin{theorem}\label{thm:local2}
	Let $\bm \nu  = (\nu,\lambda,\nu^\infty)\in \Y_0(\Omega;W)$. Then, $\bm \nu \in \Bcal\!\Y(\Omega)$ if and only if\vskip8pt
	\begin{itemize}\setlength{\itemsep}{8pt}
		\item[(i)]  there exists $u \in \M(\Omega;V)$ such that  
		\[
		\Bcal u = \dprb{\id,\nu}   \, \Leb^d  \, + \, \dprb{\id,\nu^\infty}\, \lambda\,,
		\]
		\item[(ii)] at $(\Leb^d + \lambda^s)$-almost every $x \in \Omega$, there exists a tangent Young measure 
		\[
		\bm \sigma \in \Tan(\bm \nu,x) \in \Bcal\!\Y(\R^d)
		\]
		such that $\bm \sigma \mres U \in \Bcal\!\Y(U)$ for all open Lipschitz sets $U \Subset \R^d$ with $\lambda(\partial U) = 0$.
	\end{itemize}
\end{theorem}
Now, before stating the analog of Theorem~\ref{thm:char} for $\Bcal$-gradients, we will need to adapt some of the preliminary definitions of the $\Acal$-framework into the $\Bcal$-framework. In the case of potentials, the role of the wave cone is replaced by the \emph{image cone}
\[
\Irm_\Bcal \coloneqq \bigcup_{\xi \in \R^d} \im \Bbb(\xi) \subset W,
\]
which contains the set of $\Bcal$-gradients in Fourier space. The exactness property~\eqref{eq:ee} has two direct consequences: Firstly, it implies  (see~\cite[Corollary~1]{raitua2019potentials}) the equivalence 
 between $\Acal$-quasiconvexity and {$\Bcal$-gradient quasiconvexity}:
 \begin{definition}
 	A locally bounded Borel integrand $h : W \to \R$ is called $\Bcal$-gradient quasiconvex if 
 	 \[
 	h(z) \le \int_{(0,1)^d} h(z + \Bcal w(y)) \dd y \qquad \text{for all $w \in \Crm^\infty_c((0,1)^d,V)$.}
 	\]
 	and all $z \in W$.
 \end{definition}   

 Secondly, the wave cone of $\Acal$ coincides with the image cone of $\Bcal$ (i.e., $\Irm_\Bcal = \Lambda_\Acal$). 
These two observations and Theorem~\ref{thm:local} imply that $\Y_\Acal(\Omega)$ and $\Bcal\!\Y(\Omega)$ are structurally equivalent, except at their associated barycenter measures:

\begin{theorem}\label{thm:char2}
	Let $\bm \nu = (\nu,\lambda,\nu^\infty) \in \Y_0(\Omega;W)$.  Then, $\bm \nu \in \Bcal\!\Y(\Omega)$ if and only if\vskip8pt
	\begin{itemize}\setlength{\itemsep}{8pt}
		\item[(i)] there exists $u \in \M(\Omega;V)$ such that 
		\[
		\Bcal u = \dprb{\id,\nu} \, \Leb^d + \dprb{\id,\nu^\infty} \lambda\, ,
		\]
		\item[(ii)] at $\Leb^d$-almost every $x \in \Omega$, 
		\begin{align*}
			h\big(\ac\Bcal u(x)\big) \le \dprb{h,\nu_x}   + \dprb{h^\#,\nu^\infty_x} \ac \lambda(x)
		\end{align*}
		for all   upper-semicontinuous $\B$-gradient quasiconvex  integrands $h : W\to \R$ with linear growth at infinity, and 
		\item[(iii)] at $\lambda^s$-almost every $x \in \Omega$, it holds
		\[
		\supp(\nu_x^\infty) \subset   \spn\{\Irm_\Bcal\}.
		\]
	\end{itemize}
\end{theorem}

We close this section with the analog of Theorem~\ref{lem:app} for $\Bcal$-gradients:
\begin{theorem}\label{lem:app2} Let $\Omega \subset \R^d$ be a bounded open set and let $u \in \M(\Omega;V)$ be such that $\Bcal u \in \Mcal(\Omega;W)$ is a bounded Radon measure. Then, there exists a sequence $\{u_j\} \subset \Crm^\infty(\Omega;V)$ satisfying 
	\begin{align*}
		\Bcal u_j \, \Leb^d \; &\toweakstar \; \Bcal u \; &&\text{as measures in $\M(\Omega;W)$,}\\
		\Bcal u_j \, \Leb^d \; &\toY \; \bm\delta_{\Bcal u} \; &&\text{on $\Omega$,}
	\end{align*}
and  
\[
		\area{\Bcal u_j \,\Leb^d,\Omega} \; \to \; \area{\Bcal u,\Omega}. 	
\]
\end{theorem}


\section{Examples}\label{sec:ex} In this section we review, with concise examples, a few of the most well-known $\Acal$-free and $\Bcal$-gradient structures; most of which ---with the exception of (e)--- satisfy the spanning property
\[
	\spn\{\Lambda_\Acal\} = W \quad \text{or} \quad \spn\{\Irm_\Bcal\}  = W.
\] 
Let us recall that, in this case, the point-wise relation (iii) of the singular part in Theorems~\ref{thm:char} and~\ref{thm:char2} is superfluous (cf. Remark~\ref{rem:ya}). In the following list of examples, the labels \enquote{$\Acal$-free} or \enquote{potential} indicate the setting on which the operator is considered: 
\begin{enumerate}[(a)]\setlength{\itemsep}{8pt}
	\item Gradients (potential). Let $D$ be the gradient operator acting on functions $u : \Omega\to \R^m$. Clearly, $D$ is first-order operator from $\R^m$ to $\R^m \otimes \R^d$. Moreover, the set of gradients in Fourier space is 
	\[
		\Irm_{D} = \set{a \otimes \xi}{a \in \R^m,\xi \in \R^d},
	\]
	which is a generating set of $\R^m \otimes \R^d$. 
	
	\item Higher order gradients (potential). In the same context as the last point, the $k$-gradient operator
	\[
		D^k u = \bigg(\frac{\partial^k u^i}{\partial x_{p_1} \cdots \partial x_{p_k}}\bigg)\qquad p_j \in \{1,\dots,d\}; \quad i= 1,\dots,m,
	\]
	is a $k$-th order operator from $\R^m$ to $\R^m \otimes E_k(\R^d)$, where $E_k(\R^d)$ is the space of $k$-th order symmetric tensors. The set of $k$-th order gradients in Fourier space is the set
	\[
		\Irm_{D^k} = \set{a \otimes^k \xi}{a \in \R^m, \xi \in \R^d}.
	\]
	A standard polarization argument implies that $\Irm_{D^k}$ indeed spans $\R^m \otimes E_k(\R^d)$.
	
	\item Symmetric gradients (potential). The symmetric gradient of a vector field $u : \R^d \to \R^d$ is defined as as $Eu = \sym(Du) = \frac 12 (Du + Du^T)$. Clearly, $E$ defines a first-order operator from $\R^d$ to $E_2(\R^d)$. The space of Fourier symmetric gradients is given by
	\[
		\Irm_{E} = \set{a \otimes \xi + \xi \otimes a}{a,\xi \in \R^d},
	\]
	which, again by a polarization argument, can be seen to generate $E_2(\R^d)$.
	
	\item Deviatoric operator (potential). The operator that considers only the shear part of the symmetric gradient is given by
	\[
		E_Du \coloneqq \sym(Du) - \frac {\Div(u)}d  I_d, \qquad u: \R^d \to \R^d,
	\]
	where $I_d$ is the identity in $\R^d \otimes \R^d$. Therefore, $E_D$ is a first-order operator form $\R^d$ to $\setn{M \in E_2(\R^d)}{\tr(M) = 0}$. The set of shear symmetric gradients in Fourier space is 
	the set
	\[
		\set{\frac 12 (a \otimes \xi + \xi \otimes a) - \frac{(a \cdot \xi)}{d}I_d}{a,\xi \in \R^d}.
	\]
	This set contains all the tensors of the form $e_i \otimes e_i - e_j \otimes e_j$ and $e_i \otimes e_j + e_j \otimes e_i$ for $i \neq j$, which conform a basis of the trace-free 
	symmetric tensors.
	\item[(e)] The Laplacian ($\Acal$-free and potential). An interesting case is the Laplacian operator 
	\[
	\Delta u = \sum_{i=1}^d \partial_{ii} u, \qquad u : \R \to \R.
	\] 
	 The Laplacian is a 2nd order operator from $\R$ to $\R$. The $\Acal$-free perspective of the Laplacian corresponds to the variational properties of the harmonic functions. The first statement of Lemma~\ref{lem:over} gives $\Lambda_\Delta = \{0\}$ and
	\[
		\Ybf_\Delta(\Omega) \cap \Y_0(\Omega;\R) = \set{(\delta_w,0,\frarg)}{\text{$w : \Omega \to \R$ is harmonic}}.
	\]
	This says that there are no concentration nor oscillation effects occurring along sequences of uniformly bounded harmonic maps.	Of course, this is not surprising since harmonic functions satisfy local $(\Wrm^{1,\infty},\Lrm^1)$-estimates. 
	
	The $\Delta$-potential perspective is completely opposite ($\Irm_\Delta = \R$). Indeed, since $\Delta$ is a full-rank elliptic operator, then the second statement in Lemma~\ref{lem:over} implies that
	\[
		\Delta\!\Ybf(\Omega;\R) \cap \Ybf_0(\Omega;\R) = \Ybf_0(\Omega;\R),
	\]
	which says that being generated by the Laplacian of a sequence of functions represents no constraint. Heuristically, this is also not surprising due to existence of  a fundamental solution for the Laplacian.
%
	\item[(f)]  Solenoidal measures ($\Acal$-free). Let us consider the scalar-divergence operator 
	\[
		\Div(w) = \partial_1 w^1 + \dots + \partial_d w^d, \qquad w : \R^d \to \R^d.
	\]
	This defines a first-order operator from $\R^d$ to $\R$. It is straightforward to verify 
	\[
		\Lambda_{\Div} = \set{\xi^\perp}{\xi \in \R^d} = \R^d.
	\]
 	In particular, every div-quasiconvex function is convex (cf. Section~\ref{sec:qc}). This, implies that both the constitutive relations of the absolutely continuous and singular parts of solenoidal Young measures are fully unconstrained:
 	\[
 		\Ybf_{\Div}(\Omega) \cap \Y_0(\Omega;\R) =  \set{\bm \nu \in \Ybf_0(\Omega;\R^d)}{\Div ([\bm \nu]) = 0},
 	\]
 	where $[\bm \nu] = \dprb{\id_{\R^d},\nu} + \lambda\dprb{\id_{\R^d},\nu^\infty}$ is the barycenter of $\bm \nu$.

	\item[(g)] Normal currents ($\Acal$-free and potential).
	The following framework has recently received attention in light of the new ideas proposed to study certain dislocation models, which are related to functionals defined on normal {$1$-currents} without boundary (boundaries of normal $2$-currents). For a thorough understanding of these models, we refer the reader to~\cite{conti2015dislocations,hudson2018existence} and references therein. 
	
	Let $1 \le m \le d$ be an integer. The space of $m$-\emph{dimensional currents} consists of all distributions $T \in \Dcal'(\Omega;\bigwedge_m\R^d)$.  The \emph{boundary} operator $\partial_m$ acts (in the sense of distributions) on $m$-dimensional currents $T$ as 
	\[
	\dprb{\partial_m T,\omega} = \dprb{T,\mathrm d\omega}, \qquad \omega \in \textstyle{\Crm^\infty(\Omega;\bigwedge^{m-1}\R^d)}.
	\]
	Therefore $\partial_m$ is first-order operator from $\bigwedge^{m} \R^d$ to $\bigwedge^{m-1} \R^d$. De Rham's theorem implies that $\partial_m$ is a constant rank operator. Indeed,  $\im \partial_{m}(\xi) = \ker \partial_{m-1}(\xi)$ for all $\xi \in \R^d \setminus \{0\}$. Hence $\rank \partial_m(\xi)$ is continuous on the sphere, and thus also constant. The space $\Nbf_m(\Omega)$ of \emph{$m$-dimensional normal currents} is defined as the space of $m$-currents $T$, such that both $T$ and $\partial_m T$ can be represented by measures: 
	\[
	\textstyle{	\Nbf_m(\Omega) \coloneqq \setn{T \in \M(\Omega;\bigwedge_m\R^d)}{\partial_m T \in \M(\Omega;\bigwedge_{m-1}\R^d)} }.
	\]
	We say that $T \in \Nbf_m(\Omega)$ is a current without boundary provided that $\partial_m T = 0$. In this context, we say that
	\begin{enumerate}[(i)]
		\item $\bm \nu \in \Ybf_{\partial_m}(\Omega;\bigwedge^m \R^d)$ is an $m$-current Young measure without boundary, 
		\item $\bm \nu \in \partial_m\!\Y(\Omega;\bigwedge^{m-1} \R^d)$ is an $m$-boundary Young measure.
	\end{enumerate}
	Notice that the symbol of $\partial_m$ acts on $m$-vectors $w \in \bigwedge^m \R^d$ precisely as the interior multiplication $\partial_m(\xi) w = w \lrcorner \xi$. If $e_1,\dots,e_d$ is a basis of $\R^d$, then $\partial_m(e_i) [ e_j \wedge v] = 0$ for all $i \neq j$ and all $v \in \bigwedge_{m-1} \R^d$. In particular
	\[
	\{e_{i_1} \wedge \dots \wedge e_{i_m}\}_{i_1 < \dots < i_m} \subset \Lambda_{\partial_m} \quad \Longrightarrow \quad \spn \{\Lambda_{\partial_m}\} = \textstyle\bigwedge_m \R^d.
	\]
	If we consider $\partial_m$ as a potential, then De Rham's theorem gives
	\[
	\spn \{\Irm_{\partial_m} \}= \spn \{\Lambda_{\partial_{m-1}} \}= \textstyle \bigwedge_{m-1
	} \R^d.
	\]
\end{enumerate}

\section{Applications} In this section we discuss some applications of the dual characterizations. First, we give an explicit description of the Young measures, both from the $\Acal$-free and potentials perspectives, associated to full-rank elliptic systems. The remaining sections are devoted to discuss, mostly via abstract constructions, the failure of classical compensated compactness results in the $\Lrm^1$ setting.

\subsection{Young measures generated by full-rank elliptic operators} We show that, for $\Acal$ a full-rank elliptic operator, we can give a simple characterization of the $\Acal$-constrained Young measures. Let us define first define ellipticity:
\begin{definition}We say that a homogeneous linear operator $\Acal$ of order $k$, from $W$ to $X$, is elliptic if there exists 
	$c > 0$ such that
	\[
	|\Abb(\xi)[w]| \ge c |\xi|^k |w| \qquad \text{for all $\xi \in \R^d \setminus \{0\}$ and all $w \in W$.} 
	\]
If moreover $\dim(W) \ge \dim(X)$, then we say that $\Acal$ is a full-rank elliptic operator.
\end{definition}
	The following result says that the sets of $\Acal$-free and $\Acal$-gradient generalized Young measures are trivial for (full-rank) elliptic operators:
\begin{lemma}\label{lem:over} Assume that $\Acal$ is an elliptic operator from $W$ to $X$. Then,
	\begin{align*}
		\Ybf_\Acal(\Omega) \cap \Ybf_0(\Omega;W) &= \set{(\delta_w,0,\frarg)}{w \in \Lrm^1(\Omega;W), \,\Acal w = 0}.
	\end{align*}
	If moreover $\Acal$ is a full-rank elliptic operator, then also
	\begin{align*}
		\Acal\!\Y(\Omega) \cap  \Ybf_0(\Omega;X) &= \Ybf_0(\Omega;X).
	\end{align*}
\end{lemma}
\begin{proof}
			Let us prove the first statement. Let us fix $\bm \nu \in \Y_\Acal(\Omega) \cap \Y_0(\Omega;W)$. The ellipticity of $\Acal$ implies that the only mean-value zero $\Acal$-free smooth $[0,1]^d$-periodic map is the zero function. Indeed, if $w \in \Crm^\infty_\per([0,1]^d;W)$ is $\Acal$-free, then applying the Fourier transform (on the torus) to the equation gives
			\[
			0 = |\Abb(\xi)[\widehat{w}(\xi)]| \ge c |\xi|^k|\widehat w(\xi)| \quad \text{for all $\xi \in \Zbb^d \setminus \{0\}$}.
	\]
	If moreover $w$ has mean-value zero, this shows that $\widehat w(\xi) = 0$ for all $\xi \in \Zbb^d$. Or equivalently, $w = 0$. Therefore, by definition, every integrand $f \in \E(W)$ is $\Acal$-quasiconvex. Since $\Crm(\cl \Omega) \times \E(W)$ separates $\Y(\Omega;W)$ (see Lemma~\ref{lem:separation}), then properties (i)-(iii) in Theorem~\ref{thm:char} imply that $\bm \nu$ must be an elementary Young measure, i.e., 
	\[
	\bm \nu = \bm\delta_{\mu} \coloneqq (\delta_{\ac \mu},|\mu^s|,\delta_{\frac{\mu}{|\mu|}}) \quad \text{for some $\Acal$-free $\mu \in \Mcal(\Omega;W)$.}
	\] 
	The ellipticity of $\Acal$ implies that $\Lambda_\Acal = \{0\}$ and hence~\cite[Theorem~1.1]{de-philippis2016on-the-structur} implies that $\mu = \ac\mu \Leb^d$. This proves that $\bm \nu = (\delta_w,0,\frarg)$ for some $\Acal$-free integrable map $w$.
	
	We now show the statement for the potential perspective when $\Acal$ is a full-rank elliptic operator. If $d = 1$, then $\Acal$ is equivalent to the operator $D^k u$ acting on real-valued functions of one-variable. Therefore, in the case $d= 1$, the second statement follows directly from the compactness properties of $\BV(\R)$-functions and existence of primitives on open intervals of the real-line (quasiconvexity is equivalent to convexity in this case). We shall focus on the case $d \ge 2$. Since $\Acal$ is a full-rank elliptic operator, the algebraic equation $\Abb(\xi)[a] = b$,
	is soluble for all $\xi \in \Zbb^d \setminus \{0\}$. Therefore, if $w \in \Crm^\infty_\per([0,1]^d;X)$ with $\int_{[0,1]^d} w = 0$, then 
	\[
	u(x) = \sum_{m \in \Zbb^d \setminus \{0\}} \Abb(m)^{-1} [\widehat w(m)] \, \mathrm{e}^{2\pi \mathrm{i} \xi \cdot x},
	\]	
	belongs to $\Crm^\infty_\per([0,1]^d;W)$ and satisfies $\Acal u = w$. Here $\widehat \frarg$ denotes the Fourier transform on the $d$-dimensional torus. This observation and a density argument convey that a function $f : {X} \to \R$ is $\Acal$-gradient quasiconvex if and only if $f$ is constant. This, in turn, conveys that (ii)-(iii)  Theorem~\ref{thm:char2} hold trivially for all $\bm \nu \in \Ybf(\Omega;X)$. Now, we show that property (i) is also holds trivially. Since $\Abb(\xi)$ is onto for all non-zero frequencies, then $\Tbb(\xi) \coloneqq \Abb(\xi)^{-1}$ exists and is homogeneous of degree $(-k)$ on $\R^d \setminus \{0\}$.
	By~\cite[Thms.~3.2.3 and 3.2.4]{HormanderBook}, $\Tbb$ extends to a distribution $\Tbb^{\frarg} \in \Dcal'(\R^d;\mathrm{Lin}(X,W))$ satisfying $p\Tbb^{\frarg} = (p \Tbb)^{\frarg}$ for all homogeneous polynomials of degree $\ell > k - d$. Moreover $\Fcal \Tbb^{\frarg}$ is smooth on $\R^d \setminus \{0\}$ (here $\Fcal$ is the Fourier transform on $\R^d$). Setting $K_\Acal \coloneqq \mathrm i^k \Fcal \Tbb^{\frarg}$ we find that if $\eta \in \Ecal'(\R^d;X)$, then
	$u \coloneqq K_\Acal \star \eta \in \Scal'(\R^d;W)$ satisfies $\Fcal[\Acal u] = \Abb \circ \Abb^{-1} \Fcal \eta = \Fcal \eta$ (here we are using that $\Abb$ is a tensor-valued  homogeneous polynomial of degree $k > k-d$).  Thus, inverting the Fourier transform, we find that
	\[
	\Acal u = \eta \quad \text{in the sense of distributions on $\R^d$}.
	\]
	Moreover, since $k - 1 > k - d$, then up to a complex constant,
	\[
	\Fcal(\partial^\alpha u)(\xi) =  [\xi^\alpha|\xi|\Abb(\xi)^{-1}]^{\frarg} (|\xi|^{-1}\Fcal \eta),
	\]
	for all multi-indexes $\alpha \in \Nbb^d_0$ such that $|\alpha| = k-1$.
	The multiplier $m(\xi) = \xi^\alpha |\xi|\Abb(\xi)$ is homogeneous of degree zero and smooth on $\Sbb^{d-1}$. Therefore, by an application of the Mihlin multiplier theorem, we deduce the bound
	\begin{align*}
		\|\partial^\alpha u\|_{\Lrm^q}   \le C_q \bigg\|\Fcal^{-1}\bigg(\frac{\Fcal \eta}{|\xi|}\bigg)\bigg\|_{\Lrm^q(\R^d)} = C_q \|\eta \|_{\Wrm^{-1,q}(\R^d)}, \quad q \in (1,\infty).
	\end{align*}
Here, in passing to the last equality we have used that $d \ge 2$ so that the Riesz potential norm $\|\Fcal^{-1}(|\xi|^{-1}\Fcal(\frarg))\|_{\Lrm^q}$ is an equivalent norm for $\Wrm^{-1,q}(\R^d)$. 

	Now, let $\mu \in \Mcal(\Omega;X)$ be an arbitrary bounded measure and let $\eta \in (\Ecal' \cap \Mcal) (\Omega;X)$ be its trivial extension by zero on $\R^d$. Define $u \coloneqq (K_\Acal \star \eta) \mres \Omega$ and observe that the bound above and Lemma~\ref{cor:compact_embedding} imply that $u \in \Wrm^{k-1,q}(\Omega)$ for all $1 \le q < d/(d-1)$. Moreover, by construction, 
	\[
	\Acal u = \mu \quad \text{in the sense of distributions on $\Omega$}.
	\] 
	We conclude that if $\bm \nu \in \Ybf_0(\Omega;X)$, then there exists $u \in \Wrm^{k-1,1}(\Omega;W)$ such that (i) in Theorem~\ref{thm:char2} holds. Since (ii)-(iii) are trivially satisfied, Theorem~\ref{thm:char2} implies that $\bm \nu \in \Acal\!\Y(\Omega)$. This proves that indeed
	\[
	\Acal\!\Y(\Omega) \cap  \Ybf_0(\Omega;X) = \Ybf_0(\Omega;X).
	\]
	This finishes the proof.
\end{proof}

	\subsection{Failure of $\Lrm^1$-compactness for elliptic systems}	In this section we collect some results and examples that showcase the lack of rigidity occurring along sequences of $\Acal$-free functions due to concentration effects. To account for this, let us recall that sequence of functions $\{u_j\} \subset \Lrm^1(\Omega)$ is said to converge weakly in $\Lrm^1(\Omega)$ if and only if there exists $u \in \Lrm^1(\Omega)$ such that

	\[
		\int_\Omega u_jg \to \int_\Omega u g \qquad \text{for all $g \in \Lrm^\infty(\Omega)$.}
	\]
	We write $u_j \toweak u$ in $\Lrm^1(\Omega)$.
	Notice that in general $u_j \Leb^d \toweakstar u\Leb^d$ does not imply weak convergence in $\Lrm^1$. This owes to concentrations that diffuse into an absolutely continuous part. In the generalized Young measure context, this corresponds to the analysis of $\nu_x^\infty$ on points $x$ where $\ac\lambda$ is non-zero and
	\[
		u_j \toY  (\nu,\lambda,\nu^\infty)
	\]
	The Dunford-Pettis theorem gives the following criterion to rule out the appearance of \emph{diffuse concentrations}: a sequence $\{u_j\}$ is sequentially weak pre-compact in $\Lrm^1(\Omega)$ if and only if $\{u_j\}$ is \emph{equi-integrable}, i.e.,  for every $\eps>0$ there exists some $\delta>0$ such that for any Borel set $U \subset \Omega$ with $\Leb^d(U) \le \delta$ it holds
	\[
		\sup_{j \in \Nbb} \int_U |u_j| \le \eps.
	\] 	
	The examples given below are intended to exhibit how classical compensated compactness assumptions fail to prevent the lack of equi-integrability of PDE-constrained sequences.  
	We begin with the following general result, which exploits the \emph{unconstrained} behavior of the singular part of $\Acal$-free Young measures:

\begin{lemma}\label{lem:counter} Let  $\lambda \in \Mcal^+(\Omega)$ be an arbitrary finite measure.
	For any fixed vector $A \in W$ and {any} probability measure 
	$p \in \mathrm{Prob}(S_{W})$ satisfying
	\[
	\supp (p) \subset W_\Acal \quad \text{and} \quad \int_{S_{W}} z \dd p(z) \; =  \; 0,
	\] 
	there exists a sequence of functions $\{w_j\} \subset \Crm^\infty(\Omega;W)$ such that
	\begin{align*}
	\Acal u_j & = 0 \;\; \text{on $\Omega$, and}\\
	 w_j \, \Leb^d &\toY (\delta_A,\lambda,p) \;\; \text{on $\Omega$}.
	\end{align*}
	In particular, 
	\begin{align*}
		w_j \, \Leb^d&\;\toweakstar \; A \Leb^d \quad \text{in $\Mcal(\Omega;W)$},\\
		w_j &\;\not\toweak \; A \Leb^d \quad \text{in $\Lrm^1(\Omega;W)$, }\\
	\end{align*}
and
\[		\text{$\{|w_j|\}$ is not equi-integrable on open neighborhoods of $\supp (\lambda) \subset \Omega$.}\]
\end{lemma}
\begin{proof} Let $R>0$ be sufficiently large so that $\Omega \Subset B_R$. Let $\tilde \lambda \in \Mcal(B_R)$ be the trivial extension by zero on $B_R$ of the measure $\lambda$. Define also $\tilde \nu_x = \delta_A$ if $x\in\Omega$ and $\nu_x = \delta_0$ if $x \in B_R \setminus \Omega$. Then, since $\Omega$ is an open set, the  triple $\tilde {\bm \nu} = (\nu,\tilde \lambda,p)$ satisfies the weak-$*$ measurability requirements to be  Young measure in $\Y_0(B_R;W)$. We claim that $\tilde {\bm \nu}$ is an $\Acal$-free measure.
	According to Theorem~\ref{thm:char} and Remark~\ref{rem:ya} it suffices to show that
	(a) $\dpr{\id,\delta_A} + \Leb^d + \dpr{\id,p} \lambda = A$, which follows by the assumption on $p$;
	and (b) that $\dpr{h,\delta_A} + \dpr{h^\#,p} \ac\lambda(x) \ge h(0)$ for all $\Acal$-quasiconvex integrands $h:W\to \R$ with linear growth at infinity. The latter follows from Remark~\ref{rem:kk} and the fact that that the restriction of $h^\#$ to $W_\Acal$ is convex at zero, i.e.,  $\dpr{h^\#,\tau} \ge h^\#(0) = 0$ for all probability measures $\tau \in \mathrm{Prob}(W_\Acal)$ satisfying $\dpr{\id_W,\tau} = 0$ (see~\cite{kirchheim2016on-rank-one-con}). Therefore, $\dpr{h,\delta_A} + \dpr{h^\#,p} \ac\lambda(x) \ge h(A) + h^\#(0) \ac\lambda(x) = h(A)$ for $\Leb^d$-a.e. $x \in \Omega$. Then, from Corollary~\ref{cor:pure} we deduce the existence of a sequence $\{\mu_j\} \subset \Mcal(B_R;W)$ of $\Acal$-free measures that generates $\tilde {\bm \nu}$. By the theory discussed in Section~\ref{sec: Young} and a standard mollification argument we conclude that there exists a sequence $\delta_j \to 0^+$ (where $\delta_j \to 0$ faster than $j \to \infty$) such that
	\[
		\mu_j \star \rho_{\delta_j} \; \toY \; \tilde{\bm\nu} \mres \Omega \equiv \bm \nu \quad \text{on $\Omega$}.
	\]
	The first two statements then follow by setting $w_j \coloneqq (\mu_j \star \rho_{\delta_j})|_{\Omega}$. That $\{|w_j|\}$ is not equi-integrable on open neighborhoods of $\supp(\lambda)$ follows follows from~\cite[Theorem~2.9]{alibert1997non-uniform-int}.
\end{proof}

\begin{remark}
	The previous result also holds if
	\[
		\int_{S_{W}} z \dd p(z)  \; \in \;  \Lambda_\Acal
	\]
	and $\lambda = c \Leb^d$ for some $c \in \R$ (see~\cite{kirchheim2016on-rank-one-con}).
\end{remark}

A direct consequence of this lemma is the following failure of the $\Lrm^1$-rigidity for elliptic systems (cf.~\cite[Section~2.6]{MullerBook} and~\cite{new}) for $\Acal$-free measures.
\begin{corollary}
	Let $L$ be a non-trivial subspace of $W_\Acal$  and assume that $L$ has no non-trivial $\Lambda_\Acal$-connections, i.e., 
	\[
	L \cap \ker\Abb(\xi) = \{0\}  \qquad \text{for all $\xi \in \R^d \setminus \{0\}$}.
	\]
	Then, there exists a sequence $\{w_j\} \subset \Crm^\infty(\Omega;W)$ of $\Acal$-free measures satisfying
	\begin{align*}
	w_j  \, \Leb^d&\; \toweakstar \; 0 \quad \text{in $\Mcal(\Omega;W)$,}\\
	\dist(w_j,L) & \; \to  \; 0 \quad \text{in $\Lrm^1(\Omega)$},
	\end{align*}
	but
	\begin{gather*}
		w_j \not\toweak 0 \in \Lrm^1(\Omega;W), \text{ and}  \\
	\{|w_j|\} \; \text{is not locally equi-integrable on any sub-domain of $\Omega$}.
	\end{gather*}
\end{corollary}
\begin{proof}
	The assumption on $L$ is equivalent to requiring that $L \cap \Lambda_\Acal = \{0\}$. Since the class of probability measures satisfying $p \in \mathrm{Prob}(S_{W})$, $\supp (p) \subset L \cap W_\Acal$ and $\dpr{\id,p} = 0$ is non-empty, we may chose at least one $p$ with such properties. The previous lemma implies that the triple $\bm\nu = (\delta_0,\Leb^d \mres \Omega,p)$ is a generalized $\Acal$-free Young measure, generated by a sequence of $\Acal$-free measures $w_j \in \Crm^\infty(\Omega;W)$. The first  and the last two statements of the corollary follow from this observation. To prove that $\dist(w_j,L) \to 0$ in $\Lrm^1$, let us consider the positively $1$-homogeneous growth integrand $f = |\id_W- \pi_L[\frarg]| = \dist(\frarg,L)$, where $\pi_L : W\to W$ is the linear orthogonal projection onto $L$. Clearly $f = f^\infty$ and the fact that $w_j \Leb^d$ generates $\bm \nu$ implies
	\[
	\int_\Omega \dist(w_j,L) \dd \Leb^d \to \int_\Omega \dprb{f,\delta_0} \dd \Leb^d + \int_{\Omega} \dprb{f^\infty,p} \dd \Leb^d = 0.
	\]
	Here, we used that $\supp(p) \subset L$ and $f^\infty|_L \equiv 0$.
\end{proof}

The version of this result for constant rank potentials is the following:

\begin{corollary}Let $L \le \spn\{\Irm_\Bcal\} \le W$  be a non-trivial space satisfying
	\[
		L \cap  \im \Bbb(\xi) = \{0\} \qquad \text{for all $\xi \in \R^d$}.
	\]
Then there exists a sequence $\{u_j\} \subset \Crm^\infty(\Omega;V)$ such that 
	\begin{align*}
	\Bcal u_j \Leb^d &\; \toweakstar \; 0 \quad \text{in $\Mcal(\Omega;W)$,}\\
	\dist(\Bcal u_j,L) & \; \to  \; 0 \quad \text{in $\Lrm^1(\Omega)$},
\end{align*}
but
\begin{gather*}
	\Bcal u_j \not\toweak 0 \; \text{in $\Lrm^1(\Omega)$, and}\\
\{|\Bcal u_j|\} \; \text{is not locally equi-integrable on any sub-domain of $\Omega$}.
\end{gather*}
\end{corollary}
The proof of this corollary follows from a version of Lemma~\ref{lem:counter} for $\Bcal$-gradient Young measures, which we shall not prove, but that follows by similar arguments (to ones in the proof of Lemma~\ref{lem:counter}) by using Theorem~\ref{thm:char2} instead.
\begin{remark}In order to add some perspective to these results, let us recall a well-known result of
	\textsc{M\"uller} (see~\cite[Lemma 2.7]{MullerBook}) that states the following: if $L$ is a 
	space of matrices containing no rank-one connections and  
	\begin{align*}
		Du_j & \toweakstar 0  \quad \text{in $\BV(\Omega;\R^m)$},\\
	\dist(D u_j,L) &\to 0 \quad \text{in measure},
	\end{align*}
then $Du_j \to Du$ in measure. This may be understood as an $\Lrm^{1,\mathrm w}$-rigidity for gradients. The previous corollary shows that even under $\Lrm^1$-perturbations of elliptic systems, one cannot hope for sequential weak $\Lrm^1$-compactness for gradients. (Here, one should not confuse weak $\Lrm^1$-convergence with convergence in $\Lrm^{1,\mathrm w}$.)
\end{remark}
%
%


\subsection{Failure of $\Lrm^1$-compactness for the $n$-state problem} In the context of the rigidity properties for gradients, \textsc{\v{S}ver\'ak}~\cite{sverak} showed that if $K = \{A_1,A_2,A_3\}$ is set of matrices such that $\rank(A_i - A_j) \neq 1$, then every sequence $\{u_j\}$ with uniformly bounded Lipschitz constant satisfies the following compensated compactness property:
\[
	\dist(Du_j,K) \to 0 \; \text{in measure} \quad \Longrightarrow \quad Du_j \to const. \; \text{in measure}.
\]  
In particular, the restriction on $K$ prevents the formation of any non-trivial microstructures.
\textsc{\v{S}ver\'ak's} proof also implies that if $\{Du_j\}$ is $\Lrm^1$-uniformly bounded and
\[
		\dist(Du_j,K) \to 0 \; \text{in $\Lrm^1(\Omega)$},
\]
then, up to taking a subsequence,  
\[
	 Du_j \to const. \; \text{in $\Lrm^1(\Omega)$}.
\]
Notice however that neither of these compensated compactness results allows for concentrations. The first one assumes a uniform Lipschitz bound and the latter (implicitly) assumes equi-integrability since $\dist(Du_j,K) \to 0$ in $\Lrm^1(\Omega)$. 

For the  \emph{two-state problem}, \textsc{Garroni \& Nesi}~\cite{garroni} have shown a similar result for divergence-free fields. More recently, \textsc{De Philippis, Palmieri and Rindler}~\cite{palmieri} have extended this to general operators $\Acal$. The precise statement is the following: if $A_1 - A_2 \notin \Lambda_\Acal$ and $\{v_j\} \subset \Lrm^1(\Omega;W)$ is a sequence of $\Acal$-free functions satisfying
\begin{equation}\label{eq:2state}
	\dist(v_j,\{A_1,A_2\}) \to 0 \; \text{in $\Lrm^1(\Omega)$,}
\end{equation}
then, up to extracting a subsequence,
\[
	v_j \to const. \; \text{in $\Lrm^1(\Omega)$.}
\]
An interesting question to ask is what happens if we allow for concentrations, while still requiring the PDE constraint and requiring that~\eqref{eq:2state} holds. The next two examples show that one cannot expect $\Lrm^1$-compensated compactness if concentrations are allowed, even if the concentrations  occur in the directions  of $\{A_1,A_2\}$.

\begin{example} If $\Acal$ is a non-trivial operator, then
	there exist vectors $A_1,A_2 \in S_W \cap W_\Acal$ with $A_1 - A_2 \notin \Lambda_\Acal$ and a sequence $\{v_j\} \subset \Crm^\infty(\Omega;W)$ such that
	\[
		\Acal v_j = 0 \; \text{in the sese of distributions on $\Omega$.}
	\] 
	Moreover, the sequence satisfies 
\begin{align*}
		\dist(v_j,\{A_1,A_2\})  &\to 0 \; \text{in measure},\\
		\dist(v_j,\{\R^+A_1\} \cup \{\R^+ A_2\}) &\to 0 \; \text{in $\Lrm^1(\Omega)$}.
\end{align*}
However, $\{|v_j|\}$ is not equi-integrable and, for any subsequence $\{v_{j_h}\}$, 
\[
	v_{j_h} \not\to const. \; \text{weakly in $\Lrm^1(\Omega)$}.
\]
\end{example}
Similarly, we also have the following explicit example:
\begin{example} Assume that $d  = 2$ and 
	consider the $(2 \times 2)$ matrices 
	\[
	A_1 = \frac{1}{\sqrt{2}}\begin{pmatrix}
		1 & 0 \\ 0 & 1
	\end{pmatrix}, \quad A_2 =\begin{pmatrix}
		-1 & 0 \\ 0 & 0
	\end{pmatrix},\quad A_3 = \begin{pmatrix}
		0 & 0 \\ 0 & -1
	\end{pmatrix},
	\]
	Then, there exists a uniformly bounded sequence $\{u_j\} \subset \BV(\Omega;\R^2)$ satisfying
\begin{align*}
	\dist(D u_j,\{A_1,A_2,A_3\})  &\to 0 \; \text{in measure},\\
	\dist(D u_j,\{\R^+A_1\} \cup \{\R^+ A_2\}\cup \{\R^+ A_3\}) &\to 0 \; \text{in $\Lrm^1(\Omega)$}.
\end{align*}
And, for any subsequence $\{v_{j_h}\}$, 
\[
Du_{j_h} \not\to const. \; \text{in $\Lrm^1(\Omega)$}.
\]
Both examples follow directly from the following result (and its corollary below):
\end{example}

\begin{proposition}
	Let $A_1,\dots,A_n \in (W_\Acal \cap S_{W})$ be vectors satisfying 
	\[
	0_W\in \mathrm{conv}\{A_1,\dots,A_n\}.
	\] 
	Then, there exists a 
	sequence of $\Acal$-free functions $\{w_j\} \in \Crm^\infty(\Omega;W)$ satisfying
	\[
	w_j \toY (\delta_{A_1},\Leb^d \mres \Omega,p),
	\]
	where 
	\[
	p = c_1 \delta_{A_1} + \dots + c_n\delta_{A_n},
	\]
	\[
		c_1,\dots,c_n \in [0,1] \quad \text{and}\quad c_1A_1 + \dots c_nA_n = 0.
	\]
	An analogous statement holds for $\Bcal$-potentials.
\end{proposition}
\begin{proof} 
	By assumption we may find constants $c_1,\dots,c_n \in [0,1]$ as in the statement. The assertion then follows from Lemma~\ref{lem:counter}.
\end{proof}
\begin{remark}
	The result of the corollary remains valid even if the vectors $A_1,\dots,A_n$  are mutually $\Lambda_\Acal$-disconnected, i.e.,
	\[
	(A_i - A_j) \notin \Lambda_\Acal \qquad \text{for all distinct indexes $i,j = 1,\dots,n$.}
	\]
\end{remark}
\begin{corollary}Let $A$ be a direction in $S_{W} \cap W_\Acal$. There exists a sequence of $\Lrm^1$-uniformly bounded $\Acal$-free measures $w_j \subset \Crm^\infty(\Omega;W)$ such that
	\[
		w_j \toY \left(\delta_A,\Leb^d \mres \Omega,\frac 12 (\delta_A + \delta_{-A})\right).
	\] 
	In particular, $\{|w_j|\}$ is not equi-integrable on $\Omega$.
	
	Notice that, if $A \notin \Lambda_{\Acal}$, then 
	\[
		A_1 = A,\; A_2 = -A \quad \Longrightarrow \quad A_1 - A_2 \notin \Lambda_\Acal.
	\]
\end{corollary}

\subsection{Flexibility of divergence-free Young measures} So far we have seen how the lack of strong constraints for the concentration part of $\Acal$-free measures is responsible for the lack of rigidity in a number of interesting scenarios. 

Now, we review the case of the \emph{scalar divergence} operator
\[
\Div(\mu) = \sum_{j = 1}^d \partial_j \mu^j, \qquad u:\R^d \to \R^d,
\]
in the $\Acal$-free framework.
As we have already seen, $\Lambda_{\Div} = \R^d$ and hence 
\begin{equation}\label{eq:bad_div}
	\text{$f$ is div-quasiconvex $\quad \Longleftrightarrow \quad f$ is convex}.
\end{equation}
Since condition (ii) in Theorem~\ref{thm:char} holds for all convex functions, and (iii) holds trivially in this case, divergence-free generalized Young measures are only  constrained by the barycenter property (i). The following example exhibits how~\eqref{eq:bad_div} yields the existence of rather ill-behaved weak-$*$ convergent sequences of divergence-free fields:
\begin{proposition}
	Let $\lambda \in \BV(\R^d)$ be an \emph{arbitrary} compactly supported function of bounded variation and let $p \in \mathrm{Prob}(\Sbb^{d-1})$ be an \emph{arbitrary} probability measure. Define
	\[
	w = - \left(\dprb{\id_{\R^d},p} \cdot D\lambda \right) \star \Phi,	
	\]
	where 
	\[
	\Phi(x) =  \,\frac x{|x|^{d}}, \qquad x \in \R^d,
	\] 
	is the fundamental solution of the scalar divergence operator. Then, for every open and bounded Lipschitz set $\Omega \subset \R^d$ with $|D\lambda|(\partial \Omega) = 0$, there exists a sequence $u_j \in \Crm^\infty(\Omega;\R^d)$ satisfying
	\[
	\Div u_j = 0 \quad \text{and} \quad u_j \toY \bm \nu = (\delta_w,\lambda \mres \Omega,p) \; \text{on $\Omega$}.
	\]
\end{proposition}
\begin{proof} Since $\lambda$ is a compactly supported and $D\lambda$ is a Radon measure, Young's inequality implies that $w \in \Lrm^1_\loc(\R^d;\R^d)$.
	Moreover, by construction we find that the barycenter of  $\bm \nu$ is the Radon measure $\mu = \dpr{\id,\delta_w} \, \Leb^d + \dprb{\id,p} \lambda = w\, \Leb^d + \dpr{\id,p} \lambda$ and hence
	\[
	\Div(\mu) = \Div(w) + \dpr{\id,p}\cdot D\lambda = 0. 
	\]
	Thus, $\bm \nu$ satisfies (i)--(iii) and hence it is an $\Acal$-free measure on $\Omega$. That $\bm \nu$ can be generated on $\Omega$ by a sequence of smooth divergence-free fields follows from the theory discussed in Section~\ref{sec: Young}.
\end{proof}

\section{Preliminaries} The $d$-dimensional torus is denoted by $\Tbb^d$, and by $Q$ we denote the closed $d$-dimensional unit cube $[-1/2,1/2]^d$. We denote by $Q_r(x)$ the open cube with radius $r > 0$ and centered at  $x \in \R^d$.


\subsection{Geometric measure theory}
Let $X$ be a locally convex space. We denote by $\Crm_c(X)$ the space of compactly supported and continuous functions on $X$, and by $\Crm_0(X)$ we denote its completion with respect to the $\|\frarg\|_\infty$ norm. 
Here, $\Crm_c(X)$ is the inductive limit of Banach spaces $\Crm_0(K_m)$ where $K_m \subset X$ are compact and $K_m \nearrow X$. 
By the Riesz representation theorem, the space $\M_b(X)$ of bounded signed Radon measures on $X$ is the dual of $\Crm_0(X)$; a local argument of the same theorem states that the space $\M(X)$ of signed Radon measures on $X$ is the dual of $\Crm_c(X)$.
We denote by $\M^+(X)$ the subset of non-negative measures. 
Since $\Crm_0(X)$ is a Banach space, the Banach--Alaoglu theorem and its characterizations hold.
In particular:
\begin{enumerate}[1.]
	\item there exists a complete  and separable metric $d_\star : \M(X) \times \M(X) \to \R$. Moreover, convergence with respect to this metric coincides with the {weak-$*$} convergence of Radon measures (see Remark~14.15 in~\cite{mattila1995geometry-of-set}), that is,
		\[
	d_\star(\mu_j,\mu) \to 0 \quad  \Longleftrightarrow \quad \mu_j \toweakstar \mu \; \text{in $\M(X)$.}
	\]
	\item bounded sets of $\M(X)$ are $d_\star$-metrizable in the sense that $d_\star$ induces the (relative)
	weak-$*$ topology on the unit open ball of $\M(X)$.
\end{enumerate}
%
%
%
In a similar manner, for a finite dimensional inner-product euclidean space $W$, $\M_b(X;W)$ and $\M(X;W)$ will denote the spaces of {$W$-valued} bounded Radon measures and {$W$-valued} Radon measures respectively. The space $\Mcal_b(X;W)$ is a normed space endowed with the \emph{total variation} norm
\[
|\mu|(X;W) \coloneqq \sup\setBB{\int_{\Xbb} \varphi \dd \mu}{\phi \in \Crm_0(X;W), \|\phi\|_\infty \le 1}.
\]
The set of all positive Radon measures on $X$ with total variation equal to one is denoted by 
\[
\mathrm{Prob}(X) \coloneqq \setB{\nu \in \M^+(X)}{\nu(X) = 1}; 
\]
the set of \emph{probability measures} on $X$. 
Here and all the follows we write 
\begin{gather*}
	B_W\coloneqq \set{w \in W}{|w|^2 < 1},\\
	S_{W} \coloneqq \set{w \in W}{|w|^2 = 1},
\end{gather*} to denote the open unit ball and the unit sphere on $W$ respectively. Riesz' representation theorem states that every vector-valued measure $\mu \in \M(\Omega;W)$ can be written as 
\[
\mu \; = \; g_\mu|\mu| \quad \text{for some $g_\mu \in \Lrm^\infty_{\loc}(\Omega,|\mu|;S_{W})$}. 
\]
This decomposition is commonly referred as the \emph{polar decomposition} of $\mu$. The set $L_\mu$ of points $x_0 \in \Omega$ where 
\[
\lim_{r \todown 0} \aveint{Q_r(x_0)}{} |g_\mu(x) - g_\mu(x_0)| \dd |\mu|(x) = 0
\]
is satisfied, 
is called the set of $\mu$-Lebesgue points. This set conforms a full $|\mu|$-measure set of $\Omega$, i.e, $|\mu|(\Omega \setminus L_\mu) = 0$. In what follows, we shall always work with good representatives of $\mu$-integrable maps. If $g \in \Lrm^1_\loc(\Omega,\mu;W)$, then $g$ satisfies
\[
	g(x) = \lim_{r \to 0^+} \aveint{Q_r(x)}{} g(y) \dd \mu(y) \quad \text{for all $x \in L_\mu \subset \Omega$}. 
\] 
If $\mu,\lambda$ are Radon measures over $\Omega$, and $\lambda \ge 0$, then the Besicovitch differentiation theorem states that there exists a set $E \subset \Omega$ of zero $\lambda$-measure such that
\[
	\lim_{r \todown 0}\frac{\mu(Q_r(x))}{\lambda(Q_r(x))} = \frac{\dd \mu}{\dd \lambda}(x) \quad \text{for any $x \in \supp(\lambda) \setminus E$,}
\]
where $\frac{\dd \mu}{\dd \lambda} \in \Lrm^1_{\loc}(\Omega,\lambda;W)$ is the Radon--Nykod\'ym derivative of $\mu$ with respect to $\lambda$. 
Another resourceful representation of a measure is given by the \emph{Radon--Nykod\'ym--Lebesgue decomposition} which we shall frequently denote as
\[
\mu \; = \; \ac\mu \, \Leb^d \; + \; g_\mu |\mu^s|,
\] 
where as usual $\ac{\mu} \coloneqq \frac{\dd \mu}{\dd \Leb^d} \in \Lrm^1_\loc(\Omega;W)$, $|\mu^s|  \perp \Leb^d$.



\subsubsection{Push-forward measures}If $T : \Omega \to \Omega'$ is Borel measurable, the \emph{image} or \emph{push-forward} of $\mu$ under $T$ is defined by the formula
\[
	T [\mu](E) = \mu\big(T^{-1}(E)\big) \quad \text{for every Borel set $E \subset \Omega'$}.
\]
If $g : \Omega' \to [-\infty,\infty]$ is a Borel map, then
\[
	\int_E g(y) \dd T[\mu](y) \, = \, \int_{T^{-1}(E)} g(T(x)) \dd \mu(x).
\]
whenever the integrals above exist or if $g$ is integrable.
%
\subsubsection{Tangent measures}

In this section we recall the notion of tangent measure as introduced by \textsc{Preiss}~\cite{preiss1987geometry-of-mea}.
Let $\mu \in \Mcal(\Omega;W)$ and consider the map $\Trm_{x,r}(y) := (y - x)/r$, 
which blows up $B_r(x_0)$, the open ball around $x_0 \in \Omega$ with radius $r > 0$, into the open unit ball $B_1$. The {push-forward} of $\mu$ under $\Trm_{x,r}$ is given by the measure
\[
\Trm_{x,r}[\mu](E) := \mu(x_0+rE), \qquad \text{for all Borel $E \subset r^{-1}(\Omega-x)$.}
\]
A non-zero measure $\tau \in \M(\R^d;W)$ is said to be a \emph{tangent measure} of $\mu$ at $x_0 \in \R^d$, if there exist sequences $r_m \todown 0$ and $c_m > 0$ such that 
\[
c_m \, \Trm_{x,r_m} [\mu]  \toweakstar \tau \quad \text{in $\Mcal(\R^d;W)$};
\]
in this case the sequence $c_m \, \Trm_{x,r_m}[\mu]$ is called a {\it blow-up sequence}. We write $\Tan(\mu,x_0)$ to denote the set of all such tangent measures. 

Using the canonical zero extension that maps the space $\Mcal(\Omega;W)$ into the space $\Mcal(\R^d;W)$ we may use most of the results contained in the general theory for tangent measures when dealing with tangent measures defined on smaller domains. 
The following theorem,   due to \textsc{Preiss}, states that one may \emph{always}  find tangent measures.
\begin{theorem}[{Theorem~2.5 in~\cite{preiss1987geometry-of-mea}}]\label{thm:preiss} If $\mu$ is a Radon measure over $\R^d$, then $\Tan(\mu,x) \neq \varnothing$ for $\mu$-almost every $x \in \R^d$. 
\end{theorem}
This property of Radon measures measures will play a silent, but fundamental role, in our results. We shall use it to \enquote{amend} the current lack of a Poincar\'e inequality for general domains; this, because~\eqref{eq:goodbu} acts as an artificial \emph{extension operator} for tangent measures restricted to the unit cube $Q \subset \R^d$. 
Returning to the properties of tangent measures, one can show (see Remark 14.4 in~\cite{mattila1995geometry-of-set}) that, for a tangent measure $\tau \in \Tan(\mu,x_0)$, it is always possible to choose the scaling constants $c_m > 0$ in the blow-up sequence to be
\[
c_m := c \mu(x_0 + r_m \cl{U})^{-1},
\]
for any open and bounded set $U \subset \R^d$ containing the origin and with the property that $\sigma(U) > 0$, for some positive constant $c = c(U)$; this process may involve passing to a subsequence. Then, from~\cite[Thm~2.6(1)]{preiss1987geometry-of-mea} it follows that at $\mu$-almost every $x \in \Omega$ we can find 
$\tau \in \Tan(\mu,x)$ as the weak-$*$ limit a blow-up sequence of the form
\begin{equation}\label{eq:goodbu}
\frac{1}{|\mu|({Q_{r_m}(x)})} \, \Trm_{x,r_m}[\mu] \toweakstar \tau \;\; \text{in $\M(\R^d;W)$},  \qquad  |\tau|(Q) = |\tau|(\cl Q) = 1.
\end{equation}
Yet another special property of tangent measures is that at, $|\mu|$-almost every $x \in \R^d$, it holds that  
\begin{gather*}
\text{$\tau \in \Tan(\mu,x)$ \; if and only if \; $|\tau| \in \Tan(|\mu|,x)$},\\
	\tau = g_\mu (x) |\tau|;
\end{gather*}
which in particular conveys that tangent measures are generated by strictly-converging blow-up sequences. 
%
If $\mu, \lambda \in \M^+(\R^d)$ are two Radon measures with  $\mu \ll \lambda$, i.e., that $\mu$ is absolutely continuous with respect to $\lambda$, then (see Lemma 14.6 of~\cite{mattila1995geometry-of-set})
\begin{equation}\label{eq: tangent absolute}
\Tan(\mu,x) = \Tan(\lambda,x) \quad \text{for $\mu$-almost every $x \in \R^d$}.
\end{equation}
Then, a consequence of~\eqref{eq: tangent absolute} and Lebesgue's differentiation theorem  is that
%
%
%
\begin{equation}\label{eq:ae tangent}
\Tan(\mu,x) = \setB{\alpha \, \ac\mu(x) \, \Leb^d}{\alpha \in (0,\infty) }, \quad \text{at $\Leb^d$-a.e.\ $x \in \R^d$}.
\end{equation}
In fact, if $f \in \Lrm^1(\Omega,W)$, then it is an simple consequence from the Lebesgue Differentiation Theorem that 
\begin{equation}\label{eq:l1t}
	\frac{1}{r^d} \cdot \Trm_{x_0,r}[f \, \Leb^d] \, \to \, f(x_0) \, \Leb^d \quad \text{strongly in $\Lrm^1_\loc(\R^d,W)$.}
\end{equation}


\subsection{Integrands and Young measures }\label{sec: Young} Bounded generalized Young measures conform a set of dual objects to the integrands in~$\E(U;W)$. We recall briefly some aspects of this theory, which was introduced by \textsc{DiPerna} and \textsc{Majda} in~\cite{diperna1987oscillations-an} and later extended in~\cite{alibert1997non-uniform-int,kristensen2010characterizatio}.\\

\noindent\textbf{Notation.} We remind the reader that $U \subset \R^d$ denotes an open set and $\Omega \subset \R^d$  denotes an open an bounded set with $\Leb^d(\partial \Omega) = 0$.\\

For $f \in \Crm(U \times W$) we define the transformation 
\[
(Sf)(x,\widehat{z}) := (1 - |\widehat{z}|)f\left(x,\frac{\widehat{z}}{1 - |\widehat{z}|}\right), \qquad (x,\widehat z) \in {U}\times  B_W,
\]
where $B_W$ denotes the open  unit ball in $W$. Then, $Sf \in \Crm(U \times B_W)$. We set 
\begin{align*}
\E(U;W)&\coloneqq \setb{ f \in \Crm(U \times W) }{ \text{$Sf$ extends to $\Crm(\overline{U \times  B_W})$} },\\
\E(W) &\coloneqq \setb{f \in \Crm(W)}{\text{$Sf$ extends to $\Crm(\overline{B_W})$}}.
\end{align*}
Heuristically, $\E(W)$ is isomorphic to the continuous functions on the compactification of $W$ that adheres to it each direction at infinity.
In particular, all $f \in \e(U;W)$ have linear growth at infinity, i.e., there exists a positive constant $M$ such that
$|f(x,z)| \leq M(1 + |z|)$ for all $x \in U$ and all $z \in W$. With the norm 
\[
\| f \|_{ \e(\Omega;W)} := \| Sf \|_\infty, \qquad f \in  \e(U;W),
\]
the space $\e(\Omega;W)$ turns out to be a Banach space and $S$ is an isometry with inverse
\[
	(Tg)(x,{z}) := (1 + |{z}|)g\left(x,\frac{{z}}{1 + |{z}|}\right), \qquad (x, z) \in {U}\times W, 
\]
Also, by definition, for each $f \in \e(U;W)$ the limit
\[
f^\infty(x,z) := \lim_{\substack {x' \to x \\ z' \to z \\ t \to \infty}} \frac{f(x',tz')}{t}, \qquad (x,z) \in \cl{U} \times W,
\]
exists and defines a positively $1$-homogeneous function called the 
{\it strong recession function} of $f$. Moreover every $f \in \E(U;W)$ satisfies 
\begin{equation}\label{eq:mamo}
	f(x,z) = (1 + |z|) Sf\bigg(x,\frac{z}{1 + |z|}\bigg)  \quad \text{for all $x \in U$ and $z \in W$.}
\end{equation}
In particular, there exists a modulus of continuity $\omega : [0,\infty) \to [0,\infty)$, depending solely on the uniform continuity of $Sf$, such that
\[
	|f(x,z) - f(y,z)| \le \omega(|x-y|)(1 + |z|) \qquad \text{for all $x,y \in \Omega$ and $z\in W$}. 
\]
%
%
%
%
For an integrand $f \in \E(U;W)$ and a Young measure $ {\bm \nu} \in \Y(U;W)$, we define a \emph{duality paring} between $f$ and $\bm \nu$ by setting
\begin{align*}
\ddprB{f, {\bm \nu}} & \coloneqq \int_{\cl U} \dprb{f ,\nu_x}  \dd x
+ \int_{\cl U}  \dprb{f^\infty,\nu^\infty_x}    \dd \lambda(x) \\
& \coloneqq \int_{\cl U}\bigg( \int_Wf(x,z) \dd \nu_x(z) \bigg)\dd x \\
& \qquad +
 \int_{\cl U} \bigg( \int_{S_{W}}  f^\infty (x,\widehat z) \dd \nu^\infty(\widehat z) \bigg) \dd\lambda(x).
\end{align*}

The \emph{barycenter} of a Young measure $\nu \in \Y_\loc({U};\R^N)$ is defined as the measure
\[
[\bm \nu] \coloneqq \dprb{\id,\nu} \, \Leb^d \mres \Omega + \dprb{ \id, \nu^\infty  } \, \lambda \in \Mcal(\cl{U};W). 
\]
The generation of a Young measure is (cf. Definition~\ref{def:gen}) a local property in the sense that
\begin{equation}\label{eq:loc_y}
	\mu_j \toY\bm \nu \; \text{on $U$} \quad \Longrightarrow \quad \mu_j    \toY \bm \nu \; \text{on $U'$}
\end{equation} 
for all open Lipschitz sets $U' \Subset U$ with $\lambda(\partial U') = 0$.
In many cases it will be sufficient to work with functions $f \in \e(U;W)$ that 
are Lipschitz continuous and compactly supported on the $x$-variable. The following density lemma can be found in~\cite[Lemma~3]{kristensen2010characterizatio}.
\begin{lemma}\label{lem:separation}
	There exist countable families of non-negative functions $\{\varphi_p\} \subset \Crm_c(\cl{U})$ and Lipschitz integrands $\{h_p\} \subset \E(W)$ such that, for any given two Young measures $ {\bm \nu}_1,  {\bm \nu}_2 \in \Y_\loc(U;W)$,
	\[
	\text{$\ddprB{ \phi_p \otimes h_p, {\bm \nu}_1} = \ddprB{ \phi_p \otimes h_p, {\bm \nu}_2}\quad \forall \; p,q \in \N  \qquad \Longrightarrow \qquad  {\bm \nu}_1 =  {\bm \nu}_2$}.
	\]
	\begin{remark}
	Bounded sets of $\Ebf(U;W)^*$ are metrizable with respect to the  weak-$*$ topology (see~\cite[Theorem 3.28]{BrezisBook}).
	\end{remark}
\end{lemma}
Since $\Y(U;W)$ is contained in the dual space of $\E(U;W)$ via the duality pairing $\ddprb{ \frarg, \frarg }$, we say that a sequence of Young measures $( {\bm \nu}_j)  \subset \Y(U;W)$ {\it weak-$*$ converges} to $ {\bm \nu} \in \Y(U;W)$, in symbols $ {\bm \nu}_j \toweakstar {\bm \nu}$, if
\[
\ddprB{ f, {\bm \nu}_j } \to \ddprB{ f,  {\bm \nu} } \qquad \text{for all $f \in \E(U;W)$}.
\]
Fundamental for all Young measure theory is the following compactness result, see~\cite[Section 3.1]{kristensen2010characterizatio} for a proof. 
\begin{lemma}[compactness] \label{lem:YM_compact}
	Let $\{{\bm \nu}_j\} \subset \Y(U;W)$ be a sequence of Young measures satisfying 
	\begin{itemize}
		\item[(i)] the family $\setB{\dprb{ |\frarg|,{\nu}_j}}{j \in \Nbb}$ is uniformly bounded in $\Lrm^1(U)$, 
		\item[(ii)] $\sup_j \{\lambda_{j}(\cl{U})\} < \infty$.
	\end{itemize}
	Then, there exists a subsequence (not relabeled) and ${\bm \nu} \in \Y(U;W)$ such that
	${\bm \nu}_j \toweakstar {\bm \nu}$ in $\E(U;W)^*$.
\end{lemma}

The Radon--Nykod\'ym--Lebesgue decomposition induces a natural embedding 
\[
	\Mcal({U};W) \embed \Y_\loc(U;W)
\] 
via the identification $\mu \mapsto \bm\delta_\mu \coloneqq (\delta_{\ac\mu},|\mu^s|,\delta_{g_\mu})$.
Notice that a sequence of measures $\{\mu_j\} \subset \M(U;W)$ {\it generates} the Young measure $\bm\nu$ in $U$  if and only if 
\[
\text{$\bm\delta_{\mu_j} \toweakstar \bm\nu$ in $\E(U',W)^*$}
\] 
for all Lipschitz subdomains $U' \Subset U$ with $\lambda(\partial U') = 0$.


\begin{remark}For a sequence $\{\mu_j\} \subset \Mcal(U;W)$ that area-strictly converges to some limit $\mu \in \Mcal(U;W)$, it is
relatively easy to characterize the (unique) Young measure it generates. Indeed, an immediate consequence of the Separation Lemma~\ref{lem:separation} and a version of Reshetnyak's continuity theorem (see~\cite[Theorem~5]{kristensen2010characterizatio}) is that
\[
	\mu_j \to \mu \; \text{area strictly in $U$} \quad \Longleftrightarrow\quad  \bm\delta_{\mu_j} \toweakstar \bm\delta_{\mu} \; \text{in $\Ebf(U;W)^*$}.
\]

\end{remark}

Since tangent Young measures are only locally bounded,  it will also be convenient to introduce a concept of locally bounded $\Acal$-free Young measure:

\begin{definition}
	A Young measure $\bm \nu \in \Y_\loc(U;W)$ is called a locally bounded generalized $\Acal$-free  Young measure if there exists a sequence $\{\mu_j\} \subset \Mcal(U;W)$ such that
	\[
		\Acal \mu_j \to 0 \; \text{in $\Wrm^{-k,q}_\loc(U)$, \quad for some $1 < q < \frac{d}{d-1}$,}
	\]
	and
	\[
		\mu_j \toY \bm \nu \; \text{on $U'$}
	\]
	for all Lipschitz open sets $U' \Subset U$ with $\lambda(\partial U')= 0$.
\end{definition}

The proof of the following result follows the same principles used in the proof of~\cite[Lem.~2.15]{arroyo-rabasa2017lower-semiconti} with $\Acal \equiv 0$.
\begin{proposition}\label{prop:Lpclose}
	Let $\bm \nu \in \Y(U;W)$ be a Young measure generated by a sequence of the form $\{u_j \, \Leb^d\}$. If there exists another sequence $\{v_j\} \subset \Lrm^1(U;W)$ that satisfies 
	\[
		\lim_{j \to \infty} \|u_j - v_j\|_{\Lrm^1(U)} = 0,
	\]
	then
	\[
		v_j \, \Leb^d \toY \bm \nu \quad \text{on $U$}. 
	\]
\end{proposition}

The following notion of translation or shift of a Young measure will be used to deal with the fact that $W$ might be in fact larger than $W_\Acal$ in the proof of Theorem~\ref{thm:char}. 
	\begin{definition}[$\Lrm^1$-shifts]\label{def:shift} 
			We define the $v$-shift of a generalized Young measure $\bm \nu = (\nu,\lambda,\nu^\infty) \in \Ybf_\loc(U;W)$, with respect to $v \in \Lrm^1_\loc(U;W)$, as
		\[
		\Gamma_v[\bm \nu] \coloneqq (\nu \star \delta_{-v},\lambda,\nu^\infty) \in \Y_\loc(U,W).
		\]	
		Notice that if $f \in \E(U;W)$, then
	\[
	\ddprB{f,\Gamma_v[\bm \nu]} = \int_{\cl U} \dprb{f,\nu}_{x + v(x)} \dd \Leb^d(x) \, + \, \int_{\cl U} \dprb{f^\infty,\nu^\infty}_{x} \dd \lambda(x).
	\]
	For a subset $\Xcal \subset \Lrm^1_\loc(U,W)$ we write 
	\[
	\mathrm{Shift}_\Xcal[\bm \nu] \coloneqq  \setB{\Gamma_v[\bm\nu]}{v \in \Xcal}.
	\]
\end{definition}

\subsubsection{Tangent Young measures}\label{sec:tan}

	Similarly to the case of measures, we can define the push-forward of Young measures. If $T: U \to U'$ is Borel, the push-forward of $\bm \nu = (\nu,\lambda,\nu^\infty) \in \Y(U;W)$ under $T$ is the Young measure acting on $f \in \E(U',W)$ as 
	\begin{align*}
		\ddprB{f,T[\bm \nu]} & = \ddprB{f \circ (T,\id_{W}), \bm \nu} \\
							 & = \int_{U'} \dprb{f,\nu} \,\dd T[\mathrm d \Leb^d] \;  + \; \int_{\cl U'} \dprb{f^\infty,\nu^\infty}  \dd T[\lambda].
	\end{align*}
	Suppose that $x \in \Omega$. A non-zero Young measure $\bm \sigma \in \Y_\loc(\R^d;W)$ is said to be a tangent Young measure of $\bm \nu$ at $x$ if there exist sequences $r_m \searrow 0$ and $c_m > 0$ such that
	\begin{equation}\label{eq:tyoung}
		c_m \cdot \Trm_{x,r_m}[\bm \nu] \toweakstar \bm \sigma \quad \text{in $\E(U;W)^*$}
	\end{equation}
 for all $U \Subset \R^d$ with $(\lambda + \Leb^d)(\partial U) = 0$.
	The set of tangent Young measures of $\bm \nu$ at $x \in \Omega$ will be denoted as $\Tan(\bm \nu,x)$. 
	Since Young measures can be seen, via disintegration, as Radon measures over $\cl U \times W$, the property of tangent measures contained in Theorem~\ref{thm:preiss} lifts to a similar principle for tangent Young measures: 
 	\begin{proposition}\label{prop: localization regular} If $\bm \nu = (\nu,\lambda,\nu^\infty)\in \Y(\Omega;\R^d)$ is a Young measure, then  
 	\[
 		\text{$\Tan(\bm \nu,x) \neq \varnothing$ for $(\Leb^d + \lambda^s)$-almost every $x \in \Omega$.}
 	\]
 	\end{proposition}
 	Young measures also enjoy a Lebesgue-point property in the sense that a tangent Young measure $\bm \sigma \in \Tan(\bm \nu,x)$ truly represents the values of $\bm \nu$ around $x$. More precisely, we have the following localization principle for $(\Leb^d + \lambda^s)$-almost every $x_0 \in \Omega$: every tangent measure $\bm \sigma \in \Tan(\bm \nu,x_0)$ is a  homogeneous Young measure of the form
 	\begin{equation}\label{eq:tangenty}
 	\bm \sigma = (\nu_{x_0},\tau,\nu_{x_0}^\infty), \qquad \text{where} \; \tau \in \Tan(\lambda,x).
 	\end{equation}
 We state two general localization principles for Young measures, one at {\it regular} points and another one at {\it singular} points. These are well-established results, for a proof we refer the reader to~\cite{rindler2011lower-semiconti,rindler2012lower-semiconti}; see also the Appendix in~\cite{arroyo-rabasa2017lower-semiconti}. 
 
 \begin{proposition}\label{prop: localization regular}
 	Let $\bm \nu = (\nu,\lambda,\nu^\infty) \in \Y(U;W)$ be a generalized Young measure. Then for $\Leb^d$-a.e.\ $x_0 \in U$ there exists a {regular tangent Young measure} $\bm \sigma = (\sigma,\gamma,\sigma^\infty) \in  \Tan(\bm \nu,x_0)$, that is, 
 	\begin{align*}
 		[\sigma] \in \Tan([\nu],x_0), \quad &  \quad \sigma_y = \nu_{x_0} \; \text{a.e.}, \\
 		\lambda_\sigma = \frac{\di \lambda_\nu}{\di \Lcal^d}(x_0) \, \Leb^d\in \Tan(\lambda_\nu,x_0), \quad & \quad \sigma_y^\infty = \nu^\infty_{x_0} \; \text{$\lambda_\sigma$-a.e.}
 	\end{align*}
 \end{proposition}
\begin{proposition}\label{prop: localization}
	Let $\bm \nu = (\nu,\lambda,\nu^\infty) \in \Y(U;W)$ be a generalized Young measure. Then there exists a set 
	$S \subset U$ with $\lambda^s(U \setminus S) = 0$ such that for all $x_0 \in S$
	there exists a  singular tangent   Young measure $\bm \sigma = (\sigma,\gamma,\sigma^\infty) \in \Tan(\bm \nu,x_0)$, that is,
	\begin{align*}
		[\bm \sigma] \in 
		\Tan([\nu],x_0), \quad &  \quad \sigma_y = \delta_0 \; \text{a.e.}, \\
		\gamma \in \Tan(\lambda_\nu^s,x_0),\qquad \gamma(Q) = 1, \quad &
		\gamma(\partial Q) = 0, \quad \sigma_y^\infty = \nu^\infty_{x_0} \; \text{$\gamma$-a.e.}
	\end{align*}
\end{proposition}
 	This properties tell us that certain aspects of the weak-$*$ measurable maps $\nu$ and $\nu^\infty$ belonging to $\bm \nu$ can be effectively studied by looking at tangent measures of $\bm\nu$ itself. 
 	In a similar fashion to~\eqref{eq:goodbu}, at every $x_0$ where Proposition~\ref{prop: localization regular} holds, we may find a tangent Young measure $\bm \sigma\in \Tan(\bm \nu,x_0)$ as in~\eqref{eq:tangenty} with
 	\begin{equation}\label{eq:goodbu2}
	 \tau(\partial Q) = 0,
 	\end{equation}
 	and $\bm \sigma$ is generated by a blow-up sequence as in~\eqref{eq:tyoung} where 
\begin{equation}\label{eq:tangent1}
c_m  \coloneqq \begin{cases}
\Leb^d(Q_{r_m}(x))^{-1} & \text{if $x$ is a regular point of $\lambda$}\\
\lambda^s(Q_{r_m}(x))^{-1} & \text{if $x$ is a singular point of $\lambda$}
\end{cases};
\end{equation}
in any case $c_m$ can be taken to be $(\ddprb{|\frarg|,\bm \nu \mres Q_r(x)})^{-1}$. 
At singular points we may assume without loss of generality that 
\begin{equation}\label{eq:chinga}
\frac{1}{|\lambda^s|(Q_{r_m}(x))} \, \Trm_{x,r_m} [\,\ac{[\bm \nu]} \, \Leb ^d\,] \to 0 \; \text{strongly in $\Lrm^1_\loc(\R^d;W)$.}
\end{equation}

\subsection{$\A$-quasiconvexity}\label{sec:qc} We write $\Tbb^d\cong \R^d/\Zbb^d$ to denote the $d$-dimensional flat torus. 
We this convention we have $\Crm^\infty_\per([0,1]^d;W) = \Crm^\infty(\Tbb^d;W)$. For a function $u \in \Crm^\infty(\Tbb^d;W)$, we write
\[
	\int_{\Tbb^d} u \dd y \coloneqq \int_{[0,1)^d} u \dd y.
\]
In all that follows we shall write $\Crm_\sharp^\infty(\Tbb^d;W)$ to denote the subspace of smooth, $W$-valued periodic functions with mean-value zero. We recall,  from the theory discussed in~\cite[Section~2.5]{arroyo-rabasa2017lower-semiconti}, that maps
\begin{equation}\label{eq:span}
	\{\Crm_\sharp^\infty(\Tbb^d;W) \cap \ker \Acal\} \subset \Crm^\infty(\Tbb^d;W_\Acal).
\end{equation}
This set contention will be crucial for the proof of Theorem~\ref{thm:char}.

\begin{definition}[$\Acal$-quasiconvex envelope]If $h : W \to \R$ is a locally bounded Borel integrand, we define its $\Acal$-quasiconvex envelope $\Qcal_\Acal : W \to \R \cup \{-\infty\}$ as
	\[
		\Qcal_\Acal h(z) \coloneqq \inf\set{\int_{\Tbb^d} h(z + w(y)) \dd y}{w \in \Crm^\infty_\sharp(\Tbb^d;W) \cap \ker \Acal}.
	\]	
	If $f : U \times W \to \R$ is a locally bounded Borel integrand, we define its $\Acal$-quasiconvex envelope $\Qcal_\Acal f : U \times W \to \R \cup\{-\infty\}$ as
	\[
		\Qcal_\Acal f(x,z) \coloneqq \Qcal_\Acal(f(x,\frarg))(z).
	\]
\end{definition}
Below we recall some well-known convexity and Lipschitz properties of $\Acal$-quasiconvex functions.

Let $\mathcal D$ be a balanced cone of directions in $W$, that is, we assume that $tA \in \mathcal D$ for all $A \in \mathcal D$ and every $t \in \R$. A real-valued function $h \colon W\to \R$ is said to be $\mathcal D$-convex provided its restrictions to all line segments in $W$ with directions in $\mathcal D$ are convex.  
We recall the following $\Lambda_{\Acal}$-convexity property of $\A$-quasiconvex functions contained in lemma from~\cite[Proposition~3.4]{fonseca1999mathcal-a-quasi} for first-order operators and in~\cite[Lemma~2.19]{arroyo-rabasa2017lower-semiconti} for the general case:
\begin{lemma}\label{lem:cone_convex} 
	If $h:W\to \R$ is locally finite and $\Acal$-quasiconvex, then $h$ is $\Lambda_{\Acal}$-convex. 
\end{lemma}

In Section~\ref{sec:main}, we will require to work around the fact that $W_\Acal$ may not necessarily be equal to $W$. The following definition and propositions will play an important role in this regard.
\begin{definition}
	For a Borel integrand $f : \cl{U} \times W \to \R$, we define the integrand $\tilde f : \cl U \times W$ given by 
	\[
		\tilde f(x,z) \coloneqq f(x,\mathbf p z),
	\]
	where $\mathbf p :W \to W$ is the canonical linear projection onto $W_\Acal$. 
	
	If $f : W \to \R$, we also write $\tilde f : W\to \R$ to denote the integrand
	\[
		\tilde f(z) \coloneqq f(\mathbf p z).
	\]
\end{definition}

\begin{proposition}\label{prop:properties}
	Let $f : W \to \R$ be a locally bounded Borel integrand. The following holds:
	\begin{enumerate}[(a)]
		\item $\Qcal_\Acal \tilde f = (\Qcal_\Acal f)\,\tilde{}$,
		\item if $f$ is $\Acal$-quasiconvex, then $\tilde f$ is $\Acal$-quasiconvex,
		\item if $f$ is $\Lambda_\Acal$-convex, then $\tilde f$ is $(\Lambda_\Acal \cup (W_\Acal)^\perp)$-convex,
		\item if $f$ is $\Lambda_\Acal$-convex with linear growth constant $M$. Then $\tilde f$ is globally Lipschitz with 
		\[
		|\tilde f(z_1) - \tilde f(z_2)| \le C|z_1 -  z_2| \qquad \forall z_1,z_2 \in W.
		\]
		for some constant $C = C(M,\Acal)$.
	\end{enumerate}
\end{proposition}
\begin{proof}
	Property (a) is a direct consequence of~\eqref{eq:span} and the definition of $(\frarg)\,\tilde{}$\,. Property (b) follows directly from (a). Property (c) follows from Lemma~\ref{lem:cone_convex}, property (b) and the fact that $\tilde f$ is invariant on $W_\Acal$-directions. Finally, we prove (d). Up to a linear isomorphism we may assume that $\Dcal = \Lambda_\Acal \cup W_\Acal^\perp$ contains an orthonormal basis $\{w_1,\dots,w_s\}$ basis of $W$. Of course, the change of variables carries a constant in the desired Lipschitz bound, but that constant depends solely on $\Acal$. The difference between two points $z_1, z_2 \in W$ can be written as
	\[
	z_1 -  z_2 = \lambda_1 w_1 + \dots \lambda_s w_s, \quad |\lambda_1| + \dots + |\lambda_s| \lesssim |z_1 - z_2|,
	\] 
	where the constant of the last estimate depends solely on $s = s(\Acal)$. Property $(c)$ implies that $\tilde f$ is $\Dcal$-convex. Since moreover $\Dcal$ is a spanning set of directions of $W$, then~\cite[Lemma~2.5]{kirchheim2016on-rank-one-con} implies that
	\[
		\tilde f(x + y) \le (\tilde f)^\# + \tilde f(y) \quad \text{for all $x\in\Dcal$ and $y \in W$}.
	\]
	An iteration of this identity yields the upper bound 
	\[
	\tilde f(z_1) - \tilde f(z_2)\le \lambda_1(\tilde f)^\#(w_1) + \dots + \lambda_s(\tilde f)^\#(w_s).
	\]
	Since $|(\tilde f)^\#(w_i)| \le |f^\#(\mathbf p w_i)| \le M$, we may further estimate this difference by
	\[
	\tilde f(z_1) - \tilde f(z_2) \lesssim_s M \times |z_1 - z_2|.
	\]
	Reversing the roles of $z_1$ and $z_2$ gives the desired Lipschitz bounds.
\end{proof}
\begin{corollary}\label{cor:last}
	Let $\alpha \ge 0$ and let $p_1 \in \mathrm{Prob}(W)$, $p_2 \in \mathrm{Prob}(S_W)$ be two probability measures satisfying
	\[
		h(P_0) \le \dprb{h,p_1} + \alpha\dprb{h^\#,p_2}, \qquad P_0 \coloneqq \dprb{\id_W,p_1} + \alpha\dprb{\id_W,p_2} 
	\]
	for all $\Acal$-quasiconvex upper semicontinuous integrands $h : W \to \R$ with linear growth. Then,
	\[
		\supp(\delta_{-P_0} \star p_1) \subset W_\Acal
	\]
	and 
	\[
	\text{either $\alpha = 0$ \quad or \quad $\supp(p_2)\subset W_\Acal$.}
	\]
\end{corollary}
\begin{proof}
	If $W_\Acal = W$, then we have nothing to show. Else, let $X \coloneqq W_\Acal^\perp$ and let $g \in E(X)$ be an arbitrary integrand. We also define $h \coloneqq 1_{W_\Acal} \otimes g \in \E(W)$. It follows from Proposition~\ref{prop:properties}(a) that
%
	$\Qcal_\Acal h = 1_{W_\Acal} \otimes g$.
	If we choose $g \in \Crm_c(X)$, then $h^\infty \equiv 0$ and the assumption on $p_1,p_2$ yields
	\[
		g(P_0 - \mathbf p P_0) = \dprb{g,(\id_W-\mathbf p)[p_1]} \quad \forall g \in \Crm_c(X),
	\]
	where $(\id_W - \mathbf p)[p_1]$ is the push-forward of $p_1$ with respect to $(\id_W - \mathbf p)$.
	Since $\Crm_c(X)$ separates $\mathrm{Prob}(X)$, this implies that $(\id_W- \mathbf p)[p_1] = \delta_{(\id_W - \mathbf p)P_0}$. In particular 
	\[
	(\id_W- \mathbf p) \left[\left(\delta_{-(\id_W - \mathbf p)P_0}\right) \star p_1\right]  = \delta_{-(\id_W - \mathbf p)P_0}\star (\id_W - \mathbf p)[p_1] = \delta_0.
	\] This proves that $\supp(\delta_{-(\id_W - \mathbf p)P_0} \star p_1) \subset W_\Acal$, and therefore also
	\[
		  \supp(\delta_{-P_0}\star p_1) \subset W_\Acal.
	\]
	A similar argument and the previous identity further imply 
	\[
		\alpha\,\dprb{g,(\id_W - \mathbf p)[p_2]} = 0\quad \text{for all positively $1$-homogeneous $g : X \to \R$.}
	\]
	In particular, testing with $g = |\frarg|_X$, we conclude that, either $\alpha = 0$, or $(\id_W - \mathbf p)[p_2] = \delta_0$. This proves that $\alpha \{\supp (p_2)\} \subset W_\Acal$, as desired.
\end{proof}

The next two propositions will be used to address some technical details involving the proof of Theorem~\ref{thm:char} and Remark~\ref{rem:mio}:
\begin{proposition}\label{prop:dense}
	Let $f \in \E(\Omega;W)$ and assume that there exists a dense set $D \subset \Omega$ such that 
	\[
	\Qcal_\Acal \tilde f(x,0) > - \infty \quad \text{for all $x \in D$.}
	\]
	Then 
	\[
	|\Qcal_\Acal \tilde f(x,z)| \le C(1 + |z|) \quad \text{for all $(x,z) \in \Omega \times W$}
	\]
	for some constant $C$ depending on $\Acal$ and $\|g\|_{\E(\Omega;W_\Acal)}$.
\end{proposition}\begin{proof}
	Since $f \in \E(\Omega;W)$, there exists a constant $M = \|f\|_{\E(\Omega;W)}$ such that $|f(z)| \le M(1 + |z|)$. It follows from the definition of $\Acal$-quasiconvexity (testing with the field $w = 0$) that
	\[
	\Qcal_\Acal \tilde f(x,z) \le \int_Q f(x,\mathbf p z) \dd y \le M(1 + |z|) \qquad \forall (x,z) \in \Omega \times W.
	\]
	It follows from Proposition~\ref{prop:properties}(c) and a suitable version of~\cite[Lemma 2.5]{kristensen1999} that, if $x \in D$, then 
	\begin{equation}\label{eq:H}
		|\Qcal_\Acal \tilde f(x,z)| \le C(1 + |z|) \qquad \forall z \in W,
	\end{equation}
	where $C = C(M,\Acal)$. The cited result and its proof are originally stated for quasiconvex functions. However, a similar argument can be given for $\Dcal$-convex functions where $\Dcal$ is a spanning balanced cone. This shows that the restriction of $\Qcal_\Acal \tilde f$ on $(D \times W)$ has linear growth at infinity. We shall prove now that $\Qcal_\Acal \tilde f$ is finite for all $x \in \Omega$. Let us assume that there exists $x \in \Omega \setminus D$ with $\Qcal_\Acal \tilde f(x,0) = - \infty$. Let us fix $L > 0$ be a large real number. By our assumption on $x$, we may find  a  smooth field $w \in \Crm^\infty_\sharp(\Tbb^d;W) \cap \ker \Acal$ satisfying
	\[
	\int_Q \tilde f(x,w(y)) \dd y < -L.
	\]
	The density of $D$ allows us to find a sequence $\{x_h\} \subset D$ satisfying $x_h \to x$. We use once again the fact that $f \in \E(\Omega;W)$ to deduce that $f$ is uniformly continuous on $\cl{\Omega \times RB_W}$ where $R > \|w\|_\infty$. Hence, we may use a standard modulus of continuity argument to conclude that
	\[
	\limsup_{h \to \infty} \Qcal_\Acal \tilde f(x_h,0) \le \lim_{h \to \infty} \int_Q f(x_h,w(y)) \dd y = \int_Q f(x,w(y)) \dd y < -L.
	\]
	Letting $L > 2C$ we conclude that $|\Qcal_\Acal \tilde f(x_h,0)| > C$ for $h$ sufficiently large. This poses a contradiction to the bound~\eqref{eq:H}. Repeating the first step with $D = \Omega$, we find that
	\[
	|\Qcal_\Acal \tilde f(x,z)| \le C(1 + |z|) \qquad \forall (x,z) \in \Omega \times W,
	\] 
	This finishes the proof.
\end{proof}
\begin{proposition}\label{prop:coercive}
	Let $f : W \to \R$ be locally bounded and assume that $\Qcal_\Acal \tilde f$ is locally finite. 
	Fix $\eps \in (0,1)$, $\delta> 0$ and let $w\in \Crm_\sharp^\infty(\Tbb^d;W)$ be an $\Acal$-free field satisfying
	\[
	\Qcal_\Acal \tilde f^\eps(\xi) \ge \int_Q \tilde f^\eps(\xi + w(y)) \dd y - \delta,
	\]
	where $\tilde f^\eps \coloneqq \tilde f + \eps|\frarg|$.
	Then, 
	\[
	\|\xi + w\|_{\Lrm^1(Q)} \le \frac {C}{\eps}\Big(1 + |\xi| + \delta\Big),
	\]
	for some constant $C > 0$ depending on $\Acal$ and the linear growth constant of $f$.
\end{proposition}
\begin{proof}
	Since $\Qcal_\Acal \tilde f$ is finite, then it has linear growth with a constant $C$ that depends solely on $\Acal$ and the linear growth constant of $f$ (cf. Lemma~\ref{prop:dense}). The same holds for $\tilde f^\eps$ since $\Qcal_\Acal \tilde f^\eps \ge \Qcal_\Acal \tilde f$ up to taking $C + 1$ instead. Using the assumption we get
	\begin{align*}
		\eps \|\xi + w\|_{\Lrm^1(Q)} & \le \Qcal_\Acal \tilde f^\eps(\xi) - \int_Q \tilde f(\xi + w(y)) \dd y + \delta \\
		& \le \Qcal_\Acal \tilde f^\eps(\xi) - \Qcal_\Acal \tilde f(\xi) \dd y + \delta \le 2C(1 + |\xi|) + \delta
	\end{align*}
	The conclusion follows directly from this estimate.
\end{proof}
\begin{proposition}\label{prop:trans}
	Let $f \in \E(\Omega;W_\Acal)$ be such that $\Qcal_\Acal f(x,\frarg) > -\infty$ for all $x \in \Omega$. Fix $\eps \in (0,1)$. Then, there exists a modulus of continuity $\omega : [0,\infty) \to [0,\infty)$ such that
	\[
	|\Qcal_\Acal \tilde f^\eps(x,A) - \Qcal_\Acal \tilde f^\eps(y,A)| \le \omega(|x - y|)(1 + |A|)
	\]
	for all $x,y \in \Omega$ and $A \in W$.
	The modulus of continuity depends on $\eps,\Acal$, the linear growth constant of $f$ and the modulus of continuity of $f$.
\end{proposition}
\begin{proof}
	We begin with two observations. First, that $\Qcal_\Acal \tilde f(x,\frarg)$ is finite implies that it has linear growth and that is globally Lipschitz with constants that depend solely on $\Acal$ and the linear growth constant of $f$ (cf. Propositions~\ref{prop:properties} and~\ref{prop:dense}). Now, let $\delta > 0$ and let $w_x \in \Crm^\infty_\sharp(\Tbb^d;W_\Acal)\cap \ker \Acal$ be such that (denoting $H \coloneqq \Qcal_\Acal \tilde f^\eps(x,\frarg)$)
	\[
	H(\xi) \ge \int_Q \tilde f^\eps (\xi + w_x(y)) \dd y - \delta. 
	\]
	The previous proposition yields that $\|\xi + w\|_{\Lrm^1(Q)} \le \eps^{-1}C(1 + |\xi| + \delta)$. By definition, we get
	\begin{align*}
		H(\xi) & \ge \int \tilde f^\eps(x,\xi + w(y)) \dd y - \delta \\
		&  = \int \tilde f^\eps(y,\xi + w(y)) \dd y - \delta + R(x,y)\\
		& \ge \Qcal_\Acal \tilde f^\eps(y,\xi) - \delta + R(x,y),
	\end{align*}
	where 
	\[
	|R(x,y)| \le \omega(|x-y|)\|w\|_{\Lrm^1(Q)} \le \frac{\tilde\omega(|x-y|)}{\eps}\Big(1 + |\xi| + \delta\Big)
	\]
	The desired bound follows by letting $\delta \to 0^+$ and exchanging the roles of $x,y$.
\end{proof}

\subsection{Sobolev spaces}\label{sec:sob}In order to continue our discussion, we need to recall some facts of the theory of general Sobolev spaces. The following definitions and background results about function spaces and the Fourier transform can be found in the monographs of \textsc{Adams}~\cite[Section~1]{adamsbook} and \textsc{Stein}~\cite[Section~VI.5]{steinbook}, as well as the full compendium of definitions and results contained in the book of \textsc{Triebel}~\cite{triebelbook}.  

Recall that $\Tbb^d \cong \R^d / \Zbb^d$ denotes the $d$-dimensional flat torus.
Let $\ell \in \Nbb_0$ and let $1 < p < \infty$.
The Sobolev space $\Wrm^{\ell,p}(\Tbb^d)$ is the collection of $\Zbb^d$-periodic functions $f$ all of whose distributional derivatives $\partial^\alpha f$ with $0 \le |\alpha| \le \ell$ belong to $\Lrm^p_\loc(\R^d)$. The norm of $\Wrm^{\ell,p}(\Tbb^d)$ is  
\[
	\|f\|_{ \Wrm^{\ell,p}(\Tbb^d)} \coloneqq \bigg( \sum_{|\alpha|\le \ell} \int_{\Tbb^d} |D^\alpha f|\bigg)^{\frac1p}.
\]
\begin{remark}
	$\Wrm^{0,p}(\Tbb^d) = \Lrm^p(\Tbb^d)$.
\end{remark}
Since the torus is a compact manifold, we also have $\Wrm^{\ell,p}(\Tbb^d) = \Wrm^{\ell,p}_0(\Tbb^d)$. Calder\'on showed (see~\cite[Thm~1.2.3]{adamsbook}) the equivalence $\Wrm^{\ell,p}(\Tbb^d) = \Lrm^{\ell,p}(\Tbb^d)$ between the classical Sobolev spaces and  the Bessel potential spaces, which are defined as
\[
\Lrm^{s,p}(\Tbb^d) \coloneqq \set{f \in \Dcal'(\Tbb^d)}{ \|f\|_{\Lrm^{s,p}} = \|\Ffrak^{-1}\big[{(1 + |\xi|^2)^{\frac s2} \widehat f \, \big]}\|_{\Lrm^p} < \infty}, \quad s \in \R,
\]
where $\Ffrak$ and $(\widehat{\frarg})$ denote the Fourier transform on periodic maps (see the next section).
We shall henceforth make indistinguishable use of  $\|\frarg\|_{\Wrm^{-\ell,p}}$ and  $\|\frarg\|_{\Lrm^{-\ell,p}}$ as norms of $\Wrm^{-\ell,p}$. A standard Hahn-Banach argument (see for instance~\cite[Prop. 9.20]{BrezisBook} for the case $\ell = 1$) shows that $u \in \Wrm^{-\ell,p}(\Tbb^d)$ if and only if there exists a family $\{f_\alpha\}_{0 \le |\alpha| \le k} \subset \Lrm^{p'}(\Tbb^d)$ such that
\[
u[v] = \sum_{0 \le |\alpha| \le \ell} \int_{\Tbb^d} f_\alpha\partial^{\alpha} v \quad \forall  v \in \Wrm^{\ell,p}(\Tbb^d),
\]
and
\[
	\|u\|_{\Wrm^{-\ell,p}(\Tbb^d)} = \max_{0 \le |\alpha| \le \ell} \|f_\alpha\|_{\Lrm^{p'}(\Tbb^d)}.
\]
Here $p' = p/(p-1)$. 
If $\{\rho_\eps\}_{\eps>0}$ is a family of standard mollifiers at scale $\eps > 0$, then the representation above implies that
\[
\|u \star \rho_\eps - u\|_{\Wrm^{-\ell,p}(\Tbb^d)} = \max_{0 \le |\alpha| \le \ell} \|f_\alpha \star \rho_\eps - f_\alpha\|_{\Lrm^{p'}(\Tbb^d)} \to 0 \qquad \text{as $\eps \to 0^+$}.
\]
\begin{remark}\label{rem:dense} This shows that
\[
	\text{$\Crm^\infty(\Tbb^d)$ is dense in $\Wrm^{-\ell,p}(\Tbb^d)$}
\] 
under standard mollification.
\end{remark}

Crucial to our theory are the following direct consequences of Morrey's embedding theorem (see Corollary~9.14 in Sec.~9.3 and Remark~20 in Section~9.4 of~\cite{BrezisBook}):
\begin{theorem}
	Let $p > d$ and let $U \subset \R^d$ be an open set or $U = \Tbb^d$. Then
	\[
	\Wrm^{1,p}_0(U) \embed \Crm^{0,\alpha}(U)  \cap \Crm_0(\cl U), \qquad \alpha = 1 - \frac{d}{p}\,.
	\]
\end{theorem}
\begin{corollary}\label{cor:compact_embedding} Let $U \subset \R^d$ be an open and bounded set or $U = \Tbb^d$. Then
	\[
	\Mcal_b(U) \cembed \Wrm^{-1,q}(U) \quad \text{for all $1 < q  < {d}/(d-1)$.}
	\]
	%
	%
\end{corollary} 
\begin{proof} Notice that $q' > d$. Then, Morrey's embedding and  Ascoli–Arzelà’s theorem convey the compact embedding $\Wrm^{1,q'}_0(U) \cembed \Crm_0(U)$. 
	Since  these are Banach spaces, the assertion follows directly from~\cite[Theorem~6.4]{BrezisBook}. 
\end{proof} 
\begin{remark}
	Notice that it is not necessary to require that $U$ is Lipschitz or a domain with any type of regularity.
\end{remark}

\section{Analysis of constant rank operators}\label{sec:cr} Let us recall that our main assumption is that $\Acal$ is a linear operator of integer order $k$, from $W$ to $X$, that satisfies the constant rank condition
\begin{equation}\label{eq:cr2}
	\forall \xi \in \R^d - \{0\}, \qquad \rank \Abb(\xi) = const.
\end{equation}
In this section we shall assume that $\Acal$ is a non-trivial operator, i.e., $k \ge  1$ for otherwise all the results are trivially satisfied.
The aim of this section is to give a simple extension of the well-known $\Lrm^p$-multiplier projections for constant rank operators established by \textsc{Fonseca} \& \textsc{M\"uller} in~\cite{fonseca1999mathcal-a-quasi}. 
 The Fourier transform acts on periodic measures by the formula
\[
\widehat \mu(\xi) = \Ffrak u(\xi) \coloneqq \int_{\Tbb^d}  \mathrm{e}^{-2\pi \mathrm{i} x \cdot \xi} \dd \mu(x), \quad \mu \in \M(\Tbb^d;W).
\]  
Smooth periodic functions are represented by $\Ffrak^{-1}$ through the trigonometric sum
\[
u(x) =   \sum_{\xi \in \Zbb^d} \widehat u(\xi) \, \mathrm{e}^{2\pi \mathrm{i} x \cdot \xi}.
\]
The choice to primarily work with Fourier series  lies in the following characterization for constant rank operators due to \textsc{Raita}~\cite[Thm.~1]{raitua2019potentials} and its direct implication on periodic maps (see Lemma~\ref{lem:raita} below): Let $\Acal$ be an operator  from $W$ to $X$ as in~\eqref{eq:A}. Then $\Acal$ satisfies the constant rank condition if and only if then there exists a constant rank operator $\Bcal$  from  $V$ to $W$ such that 
\begin{equation}\label{eq:exact}
	\im \Bbb(\xi) = \ker \Abb(\xi) \quad \text{for all $\xi \in \R^d \setminus \{0\}$.}
\end{equation}
For the reminder of this section $\Acal$ and $\Bcal$ will be assumed to satisfy the exactness relation~\eqref{eq:exact}, we call $\Bcal$ an associated potential to $\Acal$ (we call $\Acal$ an associated annihilator of $\Bcal$).\footnote{The class of operators $\Bcal$ satisfying~\eqref{eq:exact} may have more than element.} \textsc{Raita} showed that every $\Acal$-free periodic field  is the $\Bcal$-gradient of a suitable potential. The following is a version for measures of the original statement~\cite[Lemma~5]{raitua2019potentials}:  

\begin{lemma}\label{lem:raita}
	Let $\mu \in \Mcal(\Tbb^d;W)$. Then $\mu$ satisfies  
	\[
	\Acal \mu = 0 \; \text{in $\Dcal(\Tbb^d;X)$} \quad \text{and} \quad \mu(\Tbb^d) = 0
	\]
	if and only if there exists a potential $u \in \Mcal(\Tbb^d;V)$ such that
	\[
	\mu = \Bcal u.
	\] 
\end{lemma}  


\subsection{$\Acal$-representatives} 
Denote by $\pi(\xi): W\to W$ the orthogonal projection from $W$ to $(\ker \Abb(\xi))^\perp$ for all $\xi \in \R^d - \{0\}$. A classical result of Schulenberger and Wilcox~\cite{wilcox1,wilcox2} states that if~\eqref{eq:cr} is verified, then the map $\xi \mapsto \pi(\xi)$ is an analytic map on $\R^d - \{0\}$, homogeneous of degree 0.   The Mihlin multiplier theorem implies that $\pi$ defines an $(\Lrm^p,\Lrm^p)$ multiplier on $\R^d$ for all $1 < p < \infty$, and standard \emph{multiplier transference methods} imply that if we set $\tilde \pi(\xi) = \pi(\xi)$, for $\xi \neq 0$, and $\tilde \pi(0) = 0$, then $\{\tilde\pi(\xi)\}_{\xi \in \Zbb^d}$ defines an $(\Lrm^p,\Lrm^p)$-multiplier on $\Tbb^d$ via the assignment (see Theorem~3.8, Corollary~3.16 and its remark below in~\cite{steinweiss} for further details):
\[
u_\Abb \coloneqq \Ffrak^{-1}(\tilde \pi \widehat u), \qquad u \in \Crm^\infty(\Tbb^d;W).
\]
%
By construction 
\begin{gather}
	\Abb \widehat {u_\Abb} = \Abb \widehat u \quad \Longrightarrow \quad \Acal u_\Abb = \Acal u. \label{eq:represent}\\
		\tilde \pi(0) = 0 \quad \Longrightarrow \quad \int_{\Tbb^d} u_\Abb = 0\label{eq:avg}.
\end{gather}

\subsubsection{Sobolev estimates} 
It is well-known that~\eqref{eq:cr} implies  the map $\xi \mapsto \Abb(\xi)^\dagger$ belongs to $\Crm^\infty(\R^d \setminus\{0\};\mathrm{Lin}(X^*,W))$ and is homogeneous of degree $-k$. Here, $M^\dagger$ denotes the Moore-Penrose inverse of $M$, which satisfies the fundamental algebraic identity $M^\dagger M = \mathrm{Proj}_{(\ker M)^\perp}$. 
Partying from this identity and using that $\widehat{\Acal u}(0) = \Abb(\xi)\widehat u(0) = 0$ for all $u \in \Crm^{\infty}(\Tbb^d;W)$, one finds that
\begin{equation}\label{eq:au}
	u_\Abb = \Ffrak^{-1}(\tilde \pi \widehat u) = \Ffrak^{-1}(\Abb(\frarg)^\dagger \widehat{\Acal u}),  \qquad u \in \Crm^\infty(\Tbb^d;W).
\end{equation}
The advantage of this perspective, is that it allows one to define $u_\Abb$ in terms of $\Acal u$ rather than $u$ itself.\footnote{The fact that $u_\Abb$ can be expressed as $\Fcal^{-1}(m \widehat{\Acal u})$ for some homogeneous multiplier $m$ of degree $(-k)$ goes back to the seminal work of \textsc{Fonseca \& M\"uller}~\cite{fonseca1999mathcal-a-quasi} on $\Acal$-free measures. The idea of exploiting the representation $m = \Abb^\dagger$ appeared in the work of \textsc{Gustafson}~\cite{gustafson} and more recently in the work of \textsc{Raita}~\cite{raitua2019potentials}.} 
Recalling the seminal ideas of Fonseca and M\"uller~\cite{fonseca1999mathcal-a-quasi}, we can exploit the representation in~\eqref{eq:au} to deduce Sobolev estimates on $u_\Abb$ directly from the regularity of $\Acal u$, as one would do for elliptic operators  (see also the exposition in~\cite{raictua2018mathrm}). In order to proceed with this task let us define the auxiliary spaces
\[
	\Wcal^{\ell,p}(\Tbb^d) \coloneqq \set{u \in \Wrm^{-1,p}(\Tbb^d;W)}{\Acal u \in \Wrm^{-k+\ell,p}(\Tbb^d)},
\]
where $\ell \in [0,k]$ is a positive integer and $1 < p < \infty$. These are Banach spaces of distributions when endowed with the natural norm $\|u\|_{\Wrm^{-1,p}} + \|\Acal u\|_{\Wrm^{-k+\ell,p}}$. Next, we show that the $\Acal$-representative operator can be extended to an operator with Sobolev-type properties on  $\Wcal^{\ell,p}(\Tbb^d)$:

\begin{lemma}\label{lem:3.2} Let $1 < p < \infty$ and let $\ell \in [0,k]$ be a positive integer. There exists a continuous linear map  $T : \Wcal^{\ell,p}(\Tbb^d)\to\Wrm^{\ell,p}(\Tbb^d)$ with the following properties:
	\begin{enumerate}[1.]
		\item $T[u] = u_\Abb$ for all $u \in \Crm^\infty(\Tbb^d;W)$,
		\item there exists a constant $C(p,\ell,\Acal)$ such that 
		\[
			\|T[u]\|_{\Wrm^{\ell,p}(\Tbb^d)} \le C \|\Acal u\|_{\Wrm^{-k+\ell,p}(\Tbb^d)},
		\]
		\item $\Acal (T[u]) = \Acal u$ in the sense of distributions on $\Tbb^d$, and
		\item $\int_{\Tbb^d} T [u] = 0$.
	\end{enumerate}
Moreover, $T$  is well-defined with respect to the inclusions 
\[
	\Wcal^{\ell,p}(\Tbb^d) \embed \Wcal^{\ell',p'}(\Tbb^d), \qquad 0 \le \ell' \le \ell, \quad 1 \le p' \le p < \infty,
\] 
in the sense that $T = T(\ell',p')$ is an extension of $T= T(\ell,p)$.

\end{lemma}
\begin{proof} Let us define $T$ on smooth maps as:
	\[
		T[u] \coloneqq \Ffrak^{-1}(\Abb(\frarg)^\dagger \widehat{\Acal u}),
	\]
	so that
	property (1) follows from~\eqref{eq:au}. In light of Remark~\ref{rem:dense}, in order to prove that $T$ extends to $\Wcal^{\ell,p}(\Tbb^d)$ with linear bound as in (2), it suffices to prove (2) for $u \in \Crm^\infty(\Tbb^d;W)$. Notice that once (2) has been established, it also suffices to verify that (3)-(4) hold for smooth maps; properties (3)-(4) for smooth maps follow from~\eqref{eq:represent}-\eqref{eq:avg}.

%
%
We shall therefore focus in proving (2) for smooth maps. Fix an integer $\ell \in [0,k]$ and consider the multiplier
\[
\xi \mapsto m(\xi) = (2\pi\mathrm i)^k\xi^\alpha|\xi|^{k-\ell}\Abb(\xi)^\dagger, \qquad \xi \in \R^d - \{0\},
\] 
where $\alpha \in \Nbb_0^d$ be a multi-index with $|\alpha| = \ell$. Consider the family $\{\tilde m(\xi)\}_{\xi \in \Zbb^d}$ defined by the rule  $\tilde m(\xi) = m(\xi)$ for all $\xi \in \Zbb^d - \{0\}$ and $\tilde m(0)= 0$. Partial differentiation and the properties of the Fourier transform yield
\begin{gather}
\widehat{\partial^\alpha u_\Abb} = (2\pi\mathrm i)^k(\frarg)^\alpha\Abb(\frarg)^\dagger \widehat{\Acal u} = \tilde m \,(|\frarg|^{\ell - k}\widehat {\Acal u}),
\end{gather}
for all $u \in \Crm^\infty(\Tbb^d;W)$. Hence, inverting the Fourier transform at both sides of the equation gives
\[
\partial^\alpha u_\Abb = \Ffrak^{-1}\big(\tilde m \Ffrak\,\big(\Ffrak^{-1}\big(|\xi|^{\ell - k}\widehat{\Acal u}\,\big)\big)\big).
\]
We readily verify that $m$ is homogeneous of degree zero, analytic on $\R^d \setminus \{0\}$. Then, in light of the transference of multipliers discussed above, the Mihlin multiplier theorem implies that the assignment $f \mapsto \Ffrak^{-1}( \tilde m \widehat {f})$
extends to an $(\Lrm^p,\Lrm^p)$-multiplier on $\Tbb^d$. In particular,
\begin{equation}\label{eq:bound}
	\begin{split}
	\|\partial^\alpha u_\Abb\|_{\Lrm^p(\Tbb^d)} 
	&  \le C_{\alpha,p} \left\|\Ffrak^{-1}\left(|\xi|^{\ell - k}\widehat{\Acal u}\right)\right\|_{\Lrm^p(\Tbb^d)} \\
	&  = C_{\alpha,p} \|\Acal u\|_{\Wrm^{-k + \ell}(\Tbb^d)}.
\end{split}
\end{equation}
Here, in passing to the last equality we have used that the Mihlin multiplier theorem implies that the norms 
\[
	\|\sigma\|_{\dot\Wrm^{-s,p}} \coloneqq \left\|\Ffrak^{-1}\left(|\xi|^{-s}\widehat{\sigma}\right)\right\|_{\Lrm^p(\Tbb^d)}
	\]
	and
	\[
	\|\sigma\|_{\Wrm^{-s,p}} \coloneqq \left\|\Ffrak^{-1}\left([1 + |\xi|^2]^{-\frac{s}{2}}\widehat{\sigma}\right)\right\|_{\Lrm^p(\Tbb^d)}
\] 
are equivalent on $\Crm^\infty_\sharp(\Tbb^d) \coloneqq \setn{\sigma \in \Crm^\infty(\Tbb^d)}{\widehat \sigma(0) = 0}$.
Running through all multi-indexes $|\alpha| = \ell$ and using Poincar\'e's inequality for periodic mean-zero functions yields the sought assertion.\end{proof}

Corollary~\ref{cor:compact_embedding} and Lemma~\ref{lem:3.2} allow us to extend the notion of $\Acal$-representative to certain subspaces of measures:


\begin{definition} Let $\ell \in [0,k]$ be an integer and let $1 < q < d/(d-1)$. If $\mu \in \Mcal(\Tbb^d;W)$ is a measure with $\Acal \mu \in \Wrm^{-k+\ell,q}(\Tbb^d)$, then by Corollary~\ref{cor:compact_embedding} we may define the $\Acal$-representative of $\mu$ as
	\[
		\mu_\Abb \coloneqq T[\mu],
	\]
	where $T$ is the linear map from 
	Lemma~\ref{lem:3.2}. 
	
	Notice that $\mu_\Abb$ is well-defined regardless of the choice of $q$, it has mean-value zero, it satisfies $\Acal \mu_\Abb = \Acal \mu$ in $\Dcal'(\Tbb;X)$ and 
	\[
		\|\mu_\Abb\|_{\Wrm^{\ell,q}(\Tbb^d)} \le C \|\Acal \mu\|_{\Wrm^{-k+\ell,q}(\Tbb^d)}
	\]
	for some constant $C = C(q,\ell,\Acal)$.
\end{definition}

		\subsection{Localization estimates} We close this section with a useful observation for estimates concerning the localization with cut-off functions. Let $\phi \in \Crm^\infty_c(\R^d)$. The commutator of $\Acal $ on $\phi$ is the linear partial differential operator 
		\[
			[\Abb,\phi] \coloneqq \Acal(\phi \,\frarg) - \phi \Acal ,
		\]
		where $\phi$ acts as a multiplication operator. It acts on distributions $\eta \in \Dcal'(\R^d;W)$ as $[\Acal,\phi](\eta) = \Acal(\phi\eta) - \phi \Acal \eta$. Due to the Leibniz differentiation rule, 
		$[\Abb,\phi]$ is a partial differential operator of order $(k - 1)$ from $W$ to $X$, with smooth coefficients depending solely on the coefficients of the principal symbol $\Abb$ and the first $k$ derivatives of $\phi$. 	In particular, if $\mu \in \M(\R^d;W)$ satisfies $\Acal \mu \in \Wrm_\loc^{-k,p}(\R^d)$, then by Corollary~\ref{cor:compact_embedding} we get  
		 \begin{equation}\label{eq:com1}
		 	\Acal(\phi \mu) = \phi \Acal \mu  + [\Abb,\phi](\mu) \in \Wrm_\loc^{-k,p}(\R^d),
		 \end{equation}
		 and
		 \begin{equation}\label{eq:com2}
		 	\|\Acal(\phi \mu)\|_{\Wrm^{-k,p}(K)} \lesssim \|\phi\|_{k,\infty(K)} \big(|\mu|(K) + \|\Acal \mu\|_{\Wrm^{-k,p}(K)}\big) 
		\end{equation}
		{for all $\phi \in \Crm^\infty(\R^d)$ with $\supp(\phi) \subset K \Subset \R^d$.}

\section{Proof of the approximation theorems}

\subsection{Proof of Theorem~\ref{lem:app}} Let us recall that we are given an $\Acal$-free measure $\mu\in \Mcal(\Omega;W)$, and we aim to find a sequence of $\Acal$-free measures that area-strict converges to $\mu$. We may, without loss of generality assume that $\Omega \subset Q$. 
	
	\textit{Step~1. An asymptotically $\Acal$-free converging recovery sequence.} Let $1 < q < d/(d-1)$. First, we show that there exists a sequence $\{\mu_j\} \subset \Crm^\infty(\Omega;W)$ such that 
	\[
		\Acal \mu_j \to 0 \; \text{in $\Wrm^{-k,q}(\Omega)$ \quad and \quad 
			$\mu_j$ area-strictly converges to $\mu$ on $\Omega$}.
	\]
	We give a variant of the construction given in Step~2 of~\cite[Sec.~5.1]{arroyo-rabasa2017lower-semiconti}:
	Let $\{\varphi_i\}_{i\in \mathbb \N} \subset \Crm^\infty_c(\Omega)$ be a locally finite 
	partition of unity of $\Omega$. For a measure or function $\sigma$ on $\Omega$, we set
	\[
	\sigma_{i} \coloneqq \phi_i \sigma,\; \sigma_{i,\eps} := \varphi_i(\sigma \star \rho_\eps),
	\] 
	where $\rho_\eps = \eps^{-d}\rho(\frarg/\eps)$ is a standard radial mollifier at scale $\eps > 0$.
		Let us begin with a few observations:
	\begin{enumerate}[(a)]
		\item Every $\mu_{i,\eps}$ is a compactly supported on $\supp(\phi_i) \Subset \Omega \subset Q$. Therefore we may naturally consider each $\mu_{i,\eps}$ as an element of $\Crm^\infty(\Tbb^d;W)$.
		%
		%
		\item The $\Acal$-free constraint on $\mu$ implies that $\Acal \mu_{i,\eps} = [\Abb,\phi_i](\mu_\eps)$ for all $i \in \Nbb_0$, where we recall that $[\Abb,\phi_i]$ is a linear operator of order $(k-1)$.
		\item As $\{\phi_i\}_{i \in \Nbb}$ is a locally finite partition, we can take linear operators inside and outside arbitrary sums subjected to it. In particular, 
		\[
		\sum_{i = 1}^\infty \Acal \mu_i = \Acal  
		\bigg(\sum_{i = 1}^\infty \mu_i \bigg) = \Acal \mu = 0. 
		\]
	\end{enumerate}
	By standard measure theoretic arguments we get that
	\begin{align}
	\mu_{i,\eps}\Leb^d &\toweakstar \mu_{i} &\text{in $\Mcal(\Tbb^d;W)$ \quad as $\eps \to 0^+$},\label{eq:41}\\
	(\ac\mu)_{i,\eps} &\to (\ac\mu)_{i} &\text{in $\Lrm^1(\Tbb^d;W)$ \quad as $\eps \to 0^+$},\\
		\|\phi_i \star \rho_\eps - \phi_i\|_\infty &\to 0 &\text{as $\eps \to 0^+$},\\
	\Acal \mu_{i,\eps} &\to  \Acal\mu_i &\text{in $\Wrm^{-k,q}(\Tbb^d)$ \quad as $\eps \to 0^+$},\label{eq:42}
\end{align}
where in establishing~\eqref{eq:42} we have used that~\eqref{eq:41} implies 
$\mu_{i,\eps} \to \mu_{i,0}$ in $\Wrm^{-1,q}(\Tbb^d)$ (see Corollary~\ref{cor:compact_embedding}) so that $[\Abb,\phi_{i}](\mu_\eps -\mu_0) \to 0$ in $\Wrm^{-k,q}(\Tbb^d)$. Then, In light of~\eqref{eq:41}-\eqref{eq:42}, for every $i \in \Nbb$ we may choose $0 < \eps_j(i)< \dist(\supp \phi_i,\partial\Omega)$ such that (writing  $\tilde \sigma_{i,j} \coloneqq \sigma_{i,\eps_j(i)}$) 
\begin{align*}
	d_\star(\tilde\mu_{i,j},\mu_i)_{\Tbb^d} &\le \frac{1}{2^i\times j}, \\
	\|\tilde{(\ac\mu)}_{i,j} - (\ac\mu)_i\|_{\Lrm^1(\Tbb^d)} &\le\frac{1}{2^i\times j}, \\
		\|\phi_i \star \rho_{\eps_i(j)} - \phi_i\|_\infty &\le\frac{1}{2^i\times j} \times \min\left\{1,\frac{1}{B_{\eps_i(j)}(\supp \phi_i)}\right\},\\
	\|\Acal (\tilde \mu_{i,j} - \mu_i)\|_{\Wrm^{-k,q}(\Tbb^d)} &\le \frac{1}{2^i\times j}.
\end{align*}
Now, for a measure or function $\sigma$ on $\Omega$ let us define 
	\[
		\tilde\sigma_j \coloneqq \sum_{i = 1}^\infty \tilde \sigma_{i,j} \, \Leb^d, \qquad  j \in \Nbb.
	\]
	By construction we have
	\[
		d_\star(\tilde\mu_j,\mu)_\Omega \le \sum_{i = 1}^\infty d_\star(\tilde \mu_{i,j},\mu_i)_{\Tbb^d} \le  \frac{1}{j},
	\]
	which shows the first condition, that indeed $\tilde \mu_j \toweakstar \mu$ on $\Omega$. In particular, the sequential  lower semicontinuity of the area functional gives the lower bound
	\[
		\area{\mu,\Omega} \le \liminf_{j \to \infty} \area{\tilde\mu_j,\Omega}.
	\]
	Next, we prove the upper bound: 
	We use that $\sqrt{1 + |w|} + |z| \ge \sqrt{1 + |w + z|}$ for all $w,z \in W$ to deduce that (recall that $\Omega \subset Q$)
	\begin{align*}
		\area{\tilde \mu_j,\Omega} &\le \area{\ac\mu,\Omega} + |\tilde{(\mu^s)}_j|(\Omega) + \|\tilde{(\ac\mu)}_j - \ac\mu\|_{\Lrm^1(\Omega)}\\
		& \le \area{\ac\mu,\Omega} + |\tilde{(\mu^s)}_j|(\Omega) + \frac 1j  \\
		&  \le \area{\ac\mu,\Omega} + |\mu^s|(\Omega) + \frac{2}{j},
	\end{align*}
where the last inequality follows from the intermediate step 
\begin{align*}
	|\tilde{(\mu^s)}_{i,j}|(\Omega) & = \dpr{|\mu^s\star \rho_{\eps_j(i)}|,\phi_i} \\
	& \le \dpr{|\mu^s|,\phi_i \star \rho_{\eps_j(i)}} \\
	& = \dpr{|\mu^s|,\phi_i} + \dpr{|\mu^s|,\phi_i \star \rho_{\eps_j(i)} - \phi} \\
	& \le |\phi_i\mu^s|(\Omega) + \frac{1}{2^i\times j}.
\end{align*}
Here, the passing to the second inequality follows from Jensen's inequality and the radial symmetry of $\rho_\eps$.
Hence, we conclude that
\[
	\limsup_{j \to\infty} \area{\tilde \mu_j,\Omega} \le \area{\mu,\Omega}. 
\] 
This, together with the lower bound and the convergence $\tilde \mu_j \toweakstar \mu$ imply that $\tilde \mu_j$ converges area-strictly to $\mu$ on $\Omega$. Lastly, we use the triangle inequality and the embedding $\Wrm^{k,q'}_0(\Omega) \embed \Wrm^{k,q'}(\Tbb^d)$ to find that
\[
	\|\Acal(\tilde \mu_j - \mu)\|_{\Wrm^{-k,q}(\Omega)} \le \sum_{i = 1}^\infty \|\Acal(\tilde \mu_{i,j} - \mu_i)\|_{\Wrm^{-k,q}(\Tbb^d)} \le \frac{1}{j},
\]
which shows that indeed $\Acal \tilde \mu_j \to \Acal \mu$ strongly in $\Wrm^{-k,q}(\Omega)$.


%

	\textit{Step~2. Construction of the $\A$-free sequence.}   Let us fix $i \in \Nbb$.  We write $v_{i,j} \coloneqq (\tilde \mu_{i,j})_\Abb$ and $v_i \coloneqq (\mu_i)_\Abb$.
	Then, in light of the estimates from Lemma~\ref{lem:3.2} and the triangle inequality we get
	\begin{equation}\label{eq:VV}
		\|v_{i,j} - v_{i}\|_{\Lrm^{q}(\Omega)} \le C \|\Acal (\tilde \mu_{i,j} - \mu_i)\|_{\Wrm^{-k,q}(\Omega)} \le \frac{C}{2^i\times j}.
	\end{equation}
	This proves that $v_{i,j} \to v_i$ in $\Lrm^{q}(\Omega)$. 
	Now, let us look at the translations $(\tilde \mu_{i,j} - v_{i,j})\Leb^d \in \M(\Tbb^d;W)$. These are $\A$-free by construction, which lead us to  define the following candidate for an $\A$-free recovery sequence:
	\[
	u_j \coloneqq \sum_{i = 1}^\infty  (\tilde\mu_{i,j} - v_{i,j} + v_i) \, \Leb^d, \quad j \in \Nbb.
	\]
	
	\emph{Claim~1.} Each $u_j$ is $\A$-free. Indeed,
%
	\[
	\Acal  u_j = \sum_{i = 1}^\infty  \Acal v_i = \sum_{i = 1}^\infty  \Acal(\phi_i \mu) = \Acal \mu = 0 \quad \text{in the sense of distributions on $\Omega$.}
	\]

%

	\emph{Claim~2.} The sequence $\{u_j \Leb^d\}$ area-strict converges to $\mu$. Since $\{\tilde \mu_j \, \Leb^d\}$ already area-converges to $\mu$, it suffices to show that 
	$\tilde \mu_j$ and $u_j$ are asymptotically $\Lrm^1$-close to each other (this is sufficient to ensure the asymptotic closeness of the area functional, which has a uniformly Lipschitz integrand). This is easily verified  since
	\begin{align*}
	\left\|\tilde\mu_j -  u_j\right\|_{\Lrm^1(\Omega)} & \lesssim_q 
	\sum_{i = 1}^\infty \|v_{i,j} - v_i\|_{\Lrm^q(\Omega)} \\
	& \stackrel{\eqref{eq:VV}}\lesssim  \frac{1}{j}\,.
	\end{align*}
	This proves the second claim, which finishes the proof.  
\qed

\subsection{Proof of Theorem~\ref{lem:app2}} Let us recall that we are given $u \in \Mcal(\Omega;V)$ with $\Bcal u \in \Mcal(\Omega;W)$. We want to show there exists a uniformly bounded sequence $\{u_j\}\subset \Crm^\infty(\Omega;V)$ such that
\[
\text{$\Bcal u_j \, \Leb^d$ area-strictly converges to $\Bcal u$ on $\Omega$.}
\]
\begin{proof} First we need to establish Sobolev estimates for the $\Bcal$-representatives of localizations of an arbitrary potential $w \in \Mcal(\Omega;V)$. 
	
	In all that follows we may assume that $\Omega \Subset \Tbb^d$. As in previous arguments, we shall indistinguishably identify compactly supported functions on $\Omega$ with their periodic extensions on $\Tbb^d$.
	Let $\Omega' \subset \Omega_{k_\Bbb-1} \Subset \dots \Subset \Omega_1 \Subset \Omega_0$ be a nested family of equidistant Lipschitz open sets. By standard methods, we may find cut-off functions $\psi_1,\dots,\psi_k \in \Crm_c^\infty(\Omega;[0,1])$ satisfying
	\[
		1_{\Omega_{r+1}} \le \psi_r \le 1_{\Omega_r}, \quad \|D^k \phi_r\|_\infty \lesssim \dist(\Omega',\partial\Omega) \qquad \forall\; 0 \le r \le k_\Bbb-1.
 	\]
 	Let $w \in \Mcal(\Omega;V)$ be such that $\Bcal w \in \Mcal(\Omega;W)$.
 	The main advantage of the potential framework is that we may localize inside the PDE: we define a sequence of smooth functions by setting $\sigma_{r} \coloneqq \psi_r w$. Computing the $\Bcal$-gradient we find that, for $1 \le r \le {k_\Bbb} -1 $ it holds
 	\[
 		\Bcal \sigma_{r} = \psi_r (\Bcal w) + [\Bbb,\psi_r](w) = \psi_0(\Bcal \sigma_{r-1}) + [\Bbb,\psi_r](\sigma_{r-1}).
 	\]
 	The idea now is to deduce Sobolev estimates for their $\Bcal$-representatives following a standard bootstrapping argument: Since $\Bcal$ satisfies the constant rank condition, we may define $w_{r} \coloneqq (\sigma_{r})_\Bbb$.
 	Then, by Lemma~\ref{lem:3.2}, we deduce the a priori estimates
 	\begin{align*}
 		\|w_{r}\|_{\Wrm^{r,q}(\Tbb^d)} & \le C \big(\|\psi_r\Bcal w\|_{\Wrm^{-{k_\Bbb} + r,q}(\Tbb^d)} + \|[\Bbb,\psi_r](\sigma_{r-1})\|_{\Wrm^{-{k_\Bbb} + r,q}(\Tbb^d)} \big) \\
 		& \le C \big(\|\Bcal w\|_{\Wrm^{-1,q}(\Tbb^d)}  + \|\sigma_{r-1}\|_{\Wrm^{r-1,q}(\Tbb^d)}\big),
 	\end{align*}
 where the constant $C$ may change from line to line and depends solely on $\dist(\Omega',\Omega)$, $\Bbb$ and $q$. Iteration of this bounds until the $({k_\Bbb}-1)$\textsuperscript{th} step yields 
 \[
 	\|w_{k_\Bbb-1}\|_{\Wrm^{r,q}(\Tbb^d)} \le C \Big(\|\Bcal w\|_{\Wrm^{-1,q}(\Tbb^d)}  + \|w\|_{\Wrm^{-1,q}(\Tbb^d)}\Big)
 \]
 for yet another constant $C = C(q,\Bbb,\dist(\Omega',\partial \Omega))$. With these estimates in hand, the approximation argument follows almost the same lines as the one used in the proof of Theorem~\ref{lem:app}. Therefore, we shall only give a sketch of the proof: Let $\{\phi_i\} \subset \Crm^\infty_c(\Omega)$ be a locally finite partition of unity, and assume that $\phi_i = (\psi_i)_{k_\Bbb -1}$ so that we may apply the previous estimates on localizations of the form $\phi_i w$.
\begin{enumerate}[1.]
	\item Set $w_{i,j} \coloneqq (\phi_i(u \star \rho_{\eps_j}))_\Bbb$ and $w_i \coloneqq (\phi_i u)_\Bbb$. Then, first applying the bootstrapping argument above with $w = u$, and subsequently with $w = u \star \rho_{\eps_j} - u$, conveys the estimate 
	\begin{align*}
		\|w_{i,j} - w_i\|_{\Wrm^{k_\Bbb-1,q}(\Tbb^d)} & \le C \Big(\|\Bcal (u \star \rho_{\eps_j}) - \Bcal u\|_{\Wrm^{-1,q}(\Tbb^d)} \\
		& \qquad + \|u \star \rho_{\eps_j} - u\|_{\Wrm^{-1,q}(\Tbb^d)}\Big) \to 0 \quad \text{as $\eps \to^+ 0$}.
	\end{align*}
Here $C = C(q,\phi_i,\Bbb)$. 
	\item Similarly to the previous proof, we define 
	\[
		u_j \coloneqq \sum_{i = 1}^\infty w_{i,j(i)} \in \Crm^\infty(\Omega;V).
	\]
	where, for each $i \in \Nbb$, $\eps_{j(i)} > 0$ is chosen so that
	\begin{align*}
		d_\star(\Bcal w_{i,j(i)},\Bcal w_i) & \le \frac{1}{2^i \times j}, \\
		\|\phi_i(\ac\Bcal u \star \rho_{\eps_{j(i)}} - \ac\Bcal u)\|_{\Lrm^1(\Tbb^d)} & \le \frac{1}{2^i \times j},\\
		\|w_{i,j(i)} - w_i \|_{\Wrm^{k_\Bbb-1,q}(\Tbb^d)} & \le \frac 1{C(\Bbb,\phi_i)} \times \frac{1}{2^i \times j},
	\end{align*}
where $C(\Bbb,\phi_i) \coloneqq \max\{1,\|\Bbb\|_{\infty}(\Sbb^{d-1}) + \|\phi_i\|_{\Wrm^{k_\Bbb,\infty}}^{k_\Bbb}\}$. 
	\item It follows that 
	\[
		d_\star(\Bcal u_j,\Bcal u) \le \sum_{i} d_\star (\Bcal w_{i,j(i)},\Bcal w_i) \le \frac{1}{j}.
	\]
	This proves the convergence $\Bcal w_j \toweakstar \Bcal w$ in $\Mcal(\Omega;W)$. Moreover, the convexity of the area functional implies the lower bound
    \[
    \area{\Bcal u,\Omega} \le \liminf_{j \to \infty}\area{\Bcal u_j,\Omega}.
    \]
	\item We decompose $\Bcal u_j =  \ac\Ibf_j    + \Ibf^s_j +  \Ibf\Ibf_j$, where
	\[
		\Ibf_j^\sigma \coloneqq \sum_{i = 1}^\infty \phi_i (\Bcal^\sigma u \star \rho_{\eps_{j(i)}}), \quad \Ibf\Ibf_j \coloneqq \sum_{i=1}^\infty [\Bbb,\phi_i](w_{i,j(i)}), \quad \sigma = \mathrm{ac},s.
	\]
	Young's inequality implies $|\Ibf^s_j|(\Omega) \le |\Bcal^s u|(\Omega)$, and from the estimates of Step~2 we deduce that
	\[
		  \ac\Ibf_j\to \ac\Bcal u \; \text{in $\Lrm^1(\Omega)$} \quad \text{and} \quad \Ibf\Ibf_j \to \Ibf\Ibf \coloneqq \sum_{i=1}^\infty [\Bbb,\phi_i](w_i) \; \text{in $\Lrm^q(\Omega)$}.
	\]
	Thus, the inequality $\sqrt{1 + |z + z'|} \le \sqrt{1 + |z|} + |z'|$ implies the upper bound
	\begin{align*}
		\area{\Bcal u_j,\Omega} & = \area{\ac\Ibf_j  + \Ibf^s_j + \Ibf\Ibf_j,\Omega} \\
		& \lesssim_q \area{\ac\Bcal u,\Omega} + |\Ibf^s_j|(\Omega)\\
		& \qquad + \|\ac\Ibf_j - \ac\Bcal u\|_{\Lrm^1(\Omega)} + \|\Ibf\Ibf_j - \Ibf\Ibf\|_{\Lrm^q(\Omega)} \\
		& \le \area{\Bcal u,\Omega} + \BigO\bigg(\frac{1}{j}\bigg).
	\end{align*}
This proves that $\limsup_{j \to \infty} \area{\Bcal u_j,\Omega} \le \area{\Bcal u,\Omega}$. We thus conclude that $\lim_{j \to \infty} \area{\Bcal u_j,\Omega} = \area{\Bcal u,\Omega}$ as desired. 
\end{enumerate}
This finishes the proof.
\end{proof}

\section{Helmholtz decomposition of generating sequences}\label{sec:dec}

In this section $\Acal$ and $\Bcal$ are constant rank operators from $W$ to $X$ and $V$
 to $W$, of respective orders $k$ and $k_\Bbb$. When we work in the $\Acal$-free context, $\Bcal$ will denote an associated potential of $\Acal$, which was discussed in the previous section (cf.~\eqref{eq:exact}) and for which Lemma~\ref{lem:raita} holds.  In all that follows $U \subset \R^d$ is an open set and we assume that 
\[
	1 <  q < \frac{d}{d-1}.
\]
The following lemma establishes that oscillations and concentrations generated along $\A$-free sequences are, in fact, only carried by $\Bcal$-gradients: 
%
%
%

\begin{lemma}\label{lem:Helm}
	Let $\bm \nu = (\nu,\lambda,\nu^\infty) \in \Y_\loc(U,W)$ be a locally bounded 
	$\A$-free Young measure and let $\Omega' \Subset U$ be a Lipschitz open subset 
	with $\lambda(\partial \Omega') = 0$.
	Then, on $\Omega'$, the barycenter measure $[\bm \nu]$ can be decomposed into the $\Bcal$-gradient of a potential $u \in \Wrm^{k_\Bbb-1,q}(\Omega')$ an $\Acal$-free   field $v \in \Lrm^{q}(\Omega')$, i.e.,  
	\[
		[\bm \nu] \mres \Omega' \, = \, \Bcal u \, +  \, v\,\Leb^d, \qquad \Acal v = 0 \; \text{on $\Omega'$}.
	\]
	Moreover, there exists a sequence $\{u_h\} \subset \Wrm^{k_\Bbb-1,q}(\Omega')$ satisfying   
	\begin{align*}
			u_h & \equiv u  \quad\; \text{on a neighborhood of $\partial \Omega'$},\\
			u_h &\to u \quad \, \text{in $\Wrm^{k_\Bbb -1,q}(\Omega')$, and}\\
			\Bcal u_h \, + \, v \Leb^d \, &\toY \, \bm \nu \quad \text{on $\Omega'$}. 
	\end{align*}
Lastly, if there exists $u_0 \in W^{k_\Bbb -1,q}(\Omega')$ such that $[\bm \nu] \mres \Omega' = \Bcal u_0$, then $v$ may be chosen to be the zero function.
\end{lemma}
\begin{proof}
		Let $\phi \in \Crm^\infty_c(\Omega;[0,1])$ be a cut-off function satisfying $\Omega' \subset \{\phi \equiv 1\}$. Without loss of generality we may assume that $\supp (\phi) \subset Q$. Let $\{\mu_j\}$ be a sequence of measures generating $\bm \nu$ and satisfying $\Acal \mu_j \to 0$ in $\Wrm^{-k,q}(\supp \phi)$. 
 We define a sequence of compactly supported measures on $U$ by setting
		$\sigma_j \coloneqq \phi\mu_j$ and $\sigma_0 \coloneqq \phi \mu$. Using the trivial extension by zero, we may regard each measure $\sigma_j$ as an element of $\M(\Tbb^d;W)$. For a function $\tau$ on the torus we write $\cl \tau \coloneqq \int_{\Tbb^d} \tau$ (if $\tau$ is a measure, we set $\cl \tau \coloneqq \tau(\Tbb^d)$) to denote its mean-value. Next, define the sequence $\{w_j\} \subset \Lrm^q(\Tbb^d;W)$ of mean-value zero maps as
		\[
			w_j \coloneqq (\sigma_j)_\Abb, \qquad j \in \Nbb_0. 
		\]
	Indeed, thanks to Lemma~\ref{lem:3.2} and  Corollary~\ref{cor:compact_embedding}  we obtain  
	\begin{align*}
		\Acal w_j = \phi \Acal \mu_j + [\Abb,\phi] (\mu_j) \quad & \Longrightarrow \quad f_j \coloneqq \Acal w_j \in \Wrm^{-k,q}(\Tbb^d) \\
		&  \Longrightarrow \quad w_j \in \Lrm^q(\Tbb^d). 
	\end{align*}
		Since $\mu_j \toweakstar \mu$ in $\Mcal(U;W)$, it holds   
		\[
			f_j = \phi \Acal \mu_j + [\Abb,\phi](\mu_j) \to [\Abb,\phi](\mu) = f_0 \quad \text{in $\Wrm^{-k,q}(\Tbb^d)$}.
		\]
		Indeed, $\Acal \mu_j \to 0$ in $\Wrm^{-k,q}(\supp \phi)$, while the convergence involving the commutator follows from the fact that (cf. Corollary~\ref{cor:compact_embedding})  $\mu_j \to \mu$ $\Wrm^{-1,q}(U)$  and that $[\Abb,\phi]$ is an operator of order at most $(k-1)$. 
%
	Hence, it follows from the estimates in Lemma~\ref{lem:3.2} that 
	\[
	\|w_j - w_0\|_{\Lrm^{q}(\Tbb^d)} \le C_q \|f_j - f_0 \|_{\Wrm^{-k,q}(\Tbb^d)} \to 0.
	\]
	This allows us to define an asymptotically $\Lrm^q$-close sequence to $\sigma_j$ by setting  
		\begin{align*}
			\tilde \sigma_j  & \phantom{:}= z_j + (w_0 + \cl{\sigma_j})\Leb^d\\
			& \coloneqq (\sigma_j - \cl{\sigma_j}\Leb^d - w_j\Leb^d) + (w_0 + \cl{\sigma_j})\Leb^d, \qquad j \in \Nbb_0.
		\end{align*}
		By construction $\{z_j\}_{j \in \Nbb} \subset \Mcal(\Tbb^d;W)$ is a sequence mean-value zero measures satisfying $z_j \toweakstar z_0$ in $\Mcal(\Tbb^d;W)$, and hence also $z_j \to z_0$ in $\Wrm^{-1,q}(\Tbb^d)$.
		Moreover, $\{z_j\}_{j \in \Nbb}$ is a sequence of $\Acal$-free measures since $\Acal w_j = \Acal \sigma_j$. 
		
		Next, we exploit the potential property of mean-value zero $\Acal$-free functions on the torus. Proposition~\ref{lem:raita} yields potentials $\tilde u_j  \subset \Lrm^q(\Tbb^d;V)$ satisfying $\Bcal \tilde u_j = z_j$ for all $j \in  \Nbb_0$, each of which we may assume to be given by its own $\Bcal$-representative, i.e.,  $\tilde u_j = (\tilde u_j)_\Bbb$. Applying once more the estimates of Lemma~\ref{lem:3.2} (for $\Bcal$ instead of $\Acal$), we find that
%
	\begin{equation}\label{eqeq}
		\|\tilde u_j - \tilde u_0\|_{\Wrm^{k_\Bbb - 1,q}(\Tbb^d)} \le   C \|z_j - z_0\|_{\Wrm^{-1,q}(\Tbb^d)} \to 0.
	\end{equation}
Since $\supp(f_0) \subset \Omega \setminus \cl{\Omega'}$, then $w_0$ is an $\A$-free measure on $\Omega'$. We readily check, setting 
	\[
		\tilde U_j \coloneqq \tilde u_j|_{\Omega'} \in \Wrm^{k_\Bbb-1,q}(\Omega'), \quad v \coloneqq (w_0 + \cl{\sigma_0})|_{\Omega'} \in \Lrm^q(\Omega'),
	\] 
	that 
	\[
		\Bcal \tilde U_0 + v\Leb^d  \equiv \sigma_0 \equiv \mu \quad \text{as measures in $\M(\Omega';W)$.}
	\]
	This proves the first assertion on the decomposition of the barycenter $\mu$ on $\Omega'$. 
Moreover, since $\lambda(\partial \Omega') = 0$, we obtain  
\begin{align*}
	\Bcal \tilde U_j +  (w_j  + \cl {\sigma_j}) \, \Leb^d \equiv \mu_j   \, \toY \, \bm \nu  \quad \text{on $\Omega'$.}
\end{align*}
Therefore, using that $w_j \to w_0$ strongly in $\Lrm^q(\Tbb^d)$, it follows from Proposition~\ref{prop:Lpclose}  that 
\begin{equation}\label{eq:eqY}
	\Bcal \tilde U_j +  v \, \Leb^d   \, \toY \, \bm \nu  \quad \text{on $\Omega'$.}
\end{equation}
	We are left to see that we can adjust the boundary of $\{\tilde U_j\}_{j \in \Nbb}$ to match the values of $u \coloneqq \tilde U_0$ near $\partial \Omega'$. 
	For a positive real $t > 0$ we define  $\Omega'_t = \set{x \in \Omega'}{\dist(x,\partial \Omega') > t}$. Fix $\phi_t \in \Crm_c^\infty(\Omega';[0,1])$ to be a cut-off of $\Omega'_{2t}$ with $\phi \equiv 0$ on $\Omega'_t$, and such that $\|\phi_t\|_{k_\Bbb,\infty} \lesssim t^{-k_\Bbb}$. Let $\delta_h \searrow 0$ be an infinitesimal sequence of positive reals. We define a sequence with $u$-boundary values by setting 
	\begin{equation}\label{eq:bry}
		u_{h,j} \coloneqq \phi_{\delta_h} [\tilde U_j - u] + u, \quad u_{h,j} \equiv u \; \text{on $\Omega' \setminus \Omega'_{\delta_h}$.} 
	\end{equation}
	Fix $h$. Since $\tilde u_j \to \tilde u_0$ in $\Wrm^{k_\Bbb - 1,q}(\Omega')$, there exists $j = j(h)$ such that 
	\[
		\|u_{h,n} -  u\|_{\Wrm^{k_\Bbb-1,q}(\Omega')} \le \frac{(\delta_h)^{k_\Bbb}}{h} \quad \text{for all $n \ge j(h)$}.
	\] 
	In particular, setting $u_h \coloneqq u_{h,j(h)}$, we can estimate the total variation of $\Bcal u_h$ as
	\begin{align*}
		|\Bcal u_h|(\Omega') & \lesssim_{\Bbb} \|\phi_{\delta_h}\|_{k_\Bbb,\infty} \cdot \| u_{h,j(h)} -  u\|_{\Wrm^{k_\Bbb-1,1}(\Omega')}  + |\Bcal  u|(\Omega')\\
		& \lesssim \frac 1h  + |\Bcal u|(\Omega').
	\end{align*} 
	Notice that this not only implies that $\{\Bcal u_h\}$ is uniformly bounded, but also that the
	 sequence does not concentrate mass on the boundary $\partial \Omega'$. Therefore, up to extracting a subsequence (which we will not relabel), the sequence generates a Young measure on $\Omega'$ which does not carry mass into the boundary, i.e., 
	\[
		\Bcal u_h + v \, \Leb^d \, \toY \, \bm \sigma \quad \text{on $\Omega'$}, \qquad \lambda_{{\bm \sigma}}(\partial \Omega') = 0,
	\]
%
On the other hand, our construction gives the equivalence of measures $\Bcal u_h = \Bcal \tilde U_{j(h)}$  when these are restricted to the set $\Omega'_{2\delta_h}$. Since $\delta_h \searrow 0$, we deduce from~\eqref{eq:eqY}-\eqref{eq:bry} that $\bm \sigma \equiv \bm \nu$ on $\Omega'$, and therefore
	\[
		\Bcal u_h + v \, \Leb^d \, \toY \, \bm \nu \quad \text{on $\Omega'$},
	\]
	with $u_h \equiv u$ on a neighborhood of $\partial \Omega'$. 
	
	The last statement follows by noticing that $v = \Bcal(u_0 -  u
	)$ and hence we may simply re-define the sequence of potentials as $U_h \coloneqq u_h + u_0 -  u$. This finishes the proof of the lemma.\end{proof}

	The proof of following lemma follows by verbatim from the first step of the proof of the lemma above:
	
	\begin{lemma}\label{lem:essential} 
	Let $\{\mu_j\}	\subset \Mcal(U;W)$ be a sequence of $\Acal$-free measures satisfying 
	\[
		\mu_j \toweakstar \mu \; \text{in $\Mcal(U;W)$}
	\]
	Then, for a bounded open subset $\Omega' \subset U$, there exist $u \in \Wrm^{k_\Bbb-1,q}(\Omega')$ and $v \in \Lrm^q(\Omega')$ such that
	\[
		\mu \mres \Omega' = \Bcal u + v \, \Leb^d.
	\]
	Moreover, there exist sequences $\{u_j\} \subset \Wrm^{k_\Bbb-1,q}(\Omega')$ and $\{v_j\} \subset \Lrm^q(\Omega')$ such that 
	\[
		\mu_j \mres \Omega' = \Bcal u_j + v_j \, \Leb^d 
	\]
	and
	\begin{align*}
			\Bcal u_j &\toweakstar \Bcal u &&\text{in $\Mcal(\Omega';W)$},\\ 
			u_j & \to u &&\text{in $\Wrm^{k_\Bbb-1,q}(\Omega')$},\\
			v_j & \to v &&\text{in $\Lrm^q(\Omega')$}.
		\end{align*}
	\end{lemma}
 	The following two results show that tangent $\Acal$-free Young measures and $\Bcal$-gradient Young measures differ only by a constant shift:	
 	\begin{corollary}[decomposition of blow-up sequences]\label{cor:helm} Let $\bm \nu = (\nu,\lambda,\nu^\infty) \in \Y(\Omega;W)$ be an $\A$-free measure and let $\bm \sigma \in \Tan(\bm \nu,x)$ be a tangent Young measure. Then, for every Lipschitz domain $\omega \Subset \R^d$ with $\lambda(\partial\omega) = 0$, there exist a potential $u \in \Wrm^{k_\Bbb-1,q}(\omega)$ and a vector $z \in W$ such that
 		\begin{gather*}
 		\Bcal u + z \, \Leb^d = [\bm \sigma] \; \text{as measures on $\omega$.}
 		\end{gather*} 
 		Moreover, there exists a sequence $\{u_h\} \subset \Wrm^{k_\Bbb-1,1}(\omega)$ satisfying  
 		\begin{align*}
 			u_h &\to u  &&\text{in $\Wrm^{k_\Bbb -1,q}(\omega')$},\\
 			u_h &\equiv u  &&\text{on a neighborhood of $\partial \omega'$},\\
 			\Bcal u_h + z \, \Leb^d \, &\toY \, \bm \sigma   &&\text{on $\omega$}.
 		\end{align*}
 	Furthermore, if $x \in \Omega$ is a singular point of $\bm \nu$ or if there exists $u_0 \in \Mcal(\Omega;V)$ such that $[\bm \nu] = \Bcal u_0$, then $z = 0 \in W$.  
 	\end{corollary}
\begin{proof}  The locality property~\eqref{eq:loc_y} of Young measures and the local decomposition of generating sequences given in Lemma~\ref{lem:Helm} imply that it is enough to show the assertion when
	\[
		\mu_j = \Bcal u_j + v\, \Leb^d \toY \bm \nu \; \text{on $\Omega$,} \qquad v \in \Lrm^1(\Omega;W), \; \Acal v = 0.
	\] 
	We consider two cases: when $x \in \Omega$ is a regular or a singular point of $\lambda$. 
	
	\textit{Regular points:} Every tangent Young measure $\bm \sigma \in \Tan(\bm \nu,x)$ is generated by a sequence of the form
	\begin{equation}\label{eq:o1}
		\frac{c}{r^d_j} \cdot \Trm_{x,r_j}[\Bcal u_j + v\,\Leb^d] \toY \bm \sigma \; \text{in $\Y_\loc(\R^d,W)$.}
	\end{equation}
	Recall from~\eqref{eq:l1t} that 
	\begin{equation}\label{eq:o2}
		\frac{c}{r^d_j} \cdot \Trm_{x,r_j}[v \, \Leb^d] \to v(x) \, \Leb^d \; \text{strongly in $\Lrm^1_\loc(\R^d,W)$.}
	\end{equation}
	Hence, from the linearity of the push-forward and the compactness of Young measures, it follows that (here we use that $\lambda(\partial \omega) = 0$). 
	\[
		\frac{c}{r^d_j} \cdot \Trm_{x,r_j}[\Bcal u_j + v(x) \,\Leb^d] \toY \bm \sigma \; \text{in $\Y(\omega,W)$.}
	\]
	The assertion follows by taking $z = v(x)$. If $[\bm \nu] = \Bcal u_0$, then a localization argument and~\eqref{eq:span} imply that
	\[
		z \in  \spn \left\{\bigcup_{\xi \in \R^d} \im \Bbb(\xi) \right\} = W_\Acal.
	\] 
	In particular, there exist $\xi_i,\dots,\xi_r \in \Sbb^{d-1}$ and $a_1,\dots,a_r \in V$ such that
	\[
		z = \Bbb(\xi_1)[a_1] + \cdots + \Bbb(\xi_r)[a_r].
	\]
	This allows us to construct a smooth primitive of $z$ as follows: Let $\eta(t) = t^{k_\Bbb}/{k_\Bbb}!$ and define
	\[
		u_z(x) \coloneqq a_1 \eta(x\cdot \xi_1) + \cdots + a_r \eta(x \cdot \xi_r)\in \Crm^\infty(\R^d;V).
	\]
	By construction satisfies we have 
	\[
		\Bcal (a_h \eta(x\cdot \xi_h)) = \sum_{ |\alpha| = k_{\Bbb}} \xi_h^\alpha B_\alpha[a_h] = \Bbb(\xi_h)[a_h] \quad \text{for all $h = 1,\dots,r$,}
	\]
	which implies that $\Bcal u_z = z$. Therefore, by Lemma~\ref{lem:Helm}, $z$ can be taken to be the zero constant. 
	
	\textit{Singular points:} This proof is easier since instead of~\eqref{eq:o1}-\eqref{eq:o2} we have  
	\begin{equation}\label{eq:o3}
	\frac{c}{\lambda^s(Q_{r_j}(x))} \cdot \Trm_{x,r_j}[\Bcal u_j + v\,\Leb^d] \toY \bm \sigma \; \text{in $\Y_\loc(\R^d,W)$.}
	\end{equation}
Recall however from~\eqref{eq:chinga} that 
	\begin{equation}\label{eq:o4}
	\frac{c}{\lambda^s(Q_{r_j}(x))} \cdot \Trm_{x,r_j}[v \, \Leb^d] \to 0 \;\; \text{strongly in $\Lrm^1_\loc(\R^d,W)$.}
	\end{equation}
	Therefore, using the same arguments as before (with different normalization constants) yields $z = 0$. 
	This completes the proof. \end{proof}

    \begin{corollary}\label{cor:shift} If $\bm \nu \in \Y(\Omega,W)$ is an $\A$-free Young measure, then 
	\[
		\Tan(\bm \nu,x) \subset \mathrm{Shifts}_{W}\Big\{\Bcal\Ybf_\loc(\R^d)\Big\} \quad \text{for $\Leb^d$ almost every $x \in \Omega$,}
	\]
	and
   	\[
    		\Tan(\bm \nu,x) \subset \Bcal\Ybf_\loc(\R^d) \quad \text{for $\lambda^s$ almost every $x \in \Omega$.}
    \]
    If moreover, there exists $u_0 \in \Mcal(\Omega;W)$ such that $[\bm \nu] = \Bcal u_0$, then
    \[
    	\Tan(\bm \nu,x) \subset \Bcal\Ybf_\loc(\R^d) \quad \text{for $(\Leb^d + \lambda^s)$ almost every $x \in \Omega$.}
    \]
    \end{corollary} 

	 \section{Proof of the local characterizations}

\subsection{Proof of Theorem~\ref{thm:local}}
	\textit{Necessity.} This is straightforward from the definition of $\A$-free Young measure, a blow-up, and a diagonalization argument. For further details we refer the reader to~\cite[Section~2.8]{arroyo-rabasa2017lower-semiconti}.
	
	\textit{Sufficiency.}  Let $\{\psi_p \otimes h_p\}_{p \in \Nbb} \subset \E(\Omega;W)$ be the countable family from Lemma~\ref{lem:separation} which separates $\Y(\Omega;W)$.  	 
Let $\bm \nu = (\nu,\lambda,\nu^\infty) \in \Y(\Omega;W)$ as in the assumptions of Theorem~\ref{thm:local} and let us  write $\mu = [\bm \nu]$ to denote the barycenter of $\bm \nu$. Consider also the positive measure
\[
	\Lambda \coloneqq  \Leb^d  + \lambda^s \in \M^+(\cl\Omega). 
\]
It follows from the main assumption, that there exists a full $\Lambda$-measure set $B \subset \Omega$ with the following property: at every $x \in B$ there exists a tangent Young measure $\bm \sigma = (\nu_x,\kappa,\nu^\infty_x) 
\in \Tan(\bm{\nu},x)$ satisfying~\eqref{eq:goodbu2} and (without carrying the $x$-dependence on several of the following elements)
\begin{gather}\label{eq:s1}
\ddprB{ \phi \otimes h,\bm \nu^{(r_j)}} \coloneqq {c_j}\ddprB{ (\phi \circ \Trm_{x,r_j})\otimes h,\bm \nu} \to  \ddprB{\phi \otimes h, \bm\sigma} \qquad \text{$r_j \searrow 0$,}\\
c_j^{-1}(x) = \ddprb{|\frarg|,\bm \nu \mres Q_{r_j}(x)}, \quad \lambda(\cl {Q_{r_j}(x)}) = \lambda(Q_{r_j}(x)) >  0.
\end{gather}
In what follows we shall simply write $c_{r_j} = c_{r_j}(x)$ when no possible confusion arises. Particular consequences of the convergence above are the following:  at every $x \in B$  we can find a blow-up sequence  
\begin{equation}\label{eq:preiss1}
	c_j \cdot \Trm_{x,r_j} [\lambda] \toweakstar \kappa, \qquad \kappa(\cl Q) = \kappa(Q) = 1,
\end{equation}
and (composing with the identity map $\id_W$) also
	\begin{equation}\label{eq:s2}
		\gamma_j \coloneqq c_j \cdot \Trm_{x,r_j} [\mu] \; \toweakstar \; [\bm \sigma], \quad |[\bm \sigma]|(Q) \le 1.
	\end{equation}
Applying Lemma~\ref{lem:essential} on the sequence $\gamma_j$ and the sets $U = \R^d$ and $\Omega' = Q$, which yields  (cf. Corollary~\ref{cor:shift})
\begin{align}
	\label{eq:cro1}	\gamma_j \mres Q &= \Bcal u^{(r_j)} + v^{(r_j)} \, \Leb^d\\
	[\bm \sigma] \mres Q &= \Bcal u + z \Leb^d, \quad z \in W. \nonumber
\end{align}
where $\{u^{(r_j)}\} \subset \Wrm^{k_\Bbb-1,1}(Q)$ and $\{v^{(r_j)}\} \subset \Lrm^1(Q;W)$ are such that
\begin{align}
	v^{(r_j)} & \to z && \text{in $\Lrm^1(Q;W)$}, 
	\quad \Acal v^{(r_j)} = 0 \; \text{on $Q$,}\\
	u^{(r_j)} &\to u  &&\text{in $\Wrm^{k_\Bbb -1,q}(Q)$}.
\end{align}

\textit{Step~1. Construction of a disjoint cover of $B$.} Fix $m \in \Nbb$ and let $\phi \in \Crm(\cl Q)$. At every $x \in \Omega$ we define $\rho_m(x)$ as the supremum over all radii $0 < r_{j}(x) \le \frac{1}{m}$  (where $r_j(x) \searrow 0$ is the sequence from the previous step at a given $x \in B$) such that
\begin{align}
\label{eq:p1}	\left|\ddprB{ \phi \otimes h_p,\bm \nu^{(r_j)}} -  \ddprB{\phi \otimes h_p, \bm\sigma}\right| &\le \frac{1}{m} \quad \forall\; p \le m,\\
\label{eq:p2}	|\gamma_m|(Q) &\le 2,\\
\label{eq:p3}				\|u^{(r_j)} - u\|_{\Wrm^{k_\Bbb -1,1}(Q)} &\le \frac{1}{m^{k_\Bbb+1}},\\
\label{eq:p4}				\|v^{(r_j)} - z\|_{\Lrm^q(Q)} &\le \frac{1}{m}.
\end{align}
Next, define the cover (of open cubes) with centers in $B$ given by
\[
\Qcal_m \coloneqq \set{Q_{r_j}(x) \subset \Omega}{x \in B, r_j(x) \le \rho_m(x)}.
\]
Notice that, since $\rho_m(x) > 0$ exists for all $x \in B$, then $\Qcal_m$ is a fine cover of $B$ and hence we may apply Besicovitch's Covering Theorem, with the measure $\Lambda$, to find a disjoint sub-cover $\Ocal_m = \{Q_{x,m}\}$, where each $Q_{x,m}$ is of the form $Q_{R_m}(x)$ for some 
\[
R_m = R_m(x) \coloneqq r_{j(m)}(x) \le \rho_m(x) \le \frac 1m
\]
and
\begin{equation}\label{eq:pizza3}
\Lambda(\Omega \setminus O_m) = 0,	\quad \text{where} \; O_m \coloneqq \bigcup_{Q_x \in \Ocal_m} Q_x.
\end{equation}

	\textit{Step~2. An adjusted generating sequence  of $\bm \sigma$.}  Let $x \in B$ be fixed and let 
 	$\bm \sigma= \bm \sigma(x)$ be the $\A$-free tangent Young measure from the beginning of the proof. Now, we apply Corollary~\ref{cor:helm} to find a sequence $\{w_h\} \subset \Wrm^{k_\Bbb-1,q}(Q)$ satisfying 
 	\begin{align*}
 	w_h & \to u && \text{in $\Wrm^{k_\Bbb-1,1}(Q)$},\\
 	\tilde \gamma_h \coloneqq  \Bcal w_h + z \, \Leb^d \, &\toY \, \bm \sigma && \text{in $\Y(Q;W)$}. 
 	\end{align*}
%
 	Since it will be of use later, let $H(m) \in \Nbb$ be sufficiently large so that 
 	\begin{align}
 	\label{eq:t1} |\Bcal w_{h}|(Q) &\le 2, && \text{for all $h \ge H(m)$}\\
 	\|w_h - u\|_{\Wrm^{k_\Bbb -1,q}(Q)} &\le \frac{1}{m^{k_\Bbb +1}}, && \text{for all $h \ge H(m)$}. \label{eq:t2}
 	\end{align}
 	We also consider $\eta_m,\phi_m \in \Crm^\infty(\cl{Q};[0,1])$ be two  cut-off functions (with disjoint support) that satisfy the following properties: 
	\[
		\mathbbm 1_{Q_{1 - \frac 4m}} \le \phi_m \le \mathbbm 1_{Q_{1 - \frac 3m}} \quad \text{and} \quad \|\phi_m\|_{k_\Bbb,\infty} \lesssim m^{k_\Bbb},
	\]	
	\[
		 \mathbbm 1_{Q_{1 - \frac 2m}} \le \mathbbm 1_{Q} - \eta_m \le  \mathbbm 1_{Q_{1 - \frac 1m}} \quad \text{and} \quad \|\eta_m\|_{k_\Bbb,\infty} \lesssim m^{k_\Bbb}.
	\]
	 
	\textit{Step~3. Boundary adjustment for generating sequences of $\bm \sigma(x)$.} The next step is to define an $\A$-free sequence generating $\bm \sigma = \bm \sigma(x)$ on $Q$, which also has a blow-up of $\mu$ as boundary values.  This should allow us to freely glue each of this approximations together while keeping the $\A$-free constraint 
	
	Fix $m \in \Nbb$ and let $Q_{x,m} \in \Ocal_m$. We begin by constructing a sequence on $Q$, which we shall later translate to $Q_{x,m} \in \Ocal_m$. Bearing in mind all the $x$-dependencies that we have omitted in the previous steps, define the $\A$-free sequence  
	\begin{align*}
		q_{h,m} \, & \coloneqq \,  \overbrace{\Bcal\big(\phi_m(w_{h} - u)\big) + 
			\Bcal u + z \, \Leb^d}^{\text{generating sequence of $\bm \sigma$}} \\
				 & \qquad  + \; \overbrace{\Bcal\big(\eta_m(u^{(R_m)} - u)\big) + (v^{(R_m)} - z) \, \Leb^d }^\text{boundary adjustment to match $\gamma_{j(m)}$} \\
					 & \, = \, [\Bbb, \phi_m](w_h - u) \,  + z \, \Leb^d + \, \phi_m  \Bcal w_h \, \\ 
					 & \qquad + \,  (1 - \phi_m - \eta_m) \Bcal u   \, + \, [\Bbb, \eta_m](u^{(R_m)} - u) \,  \\
					 & \qquad \qquad + \, \eta_m  \Bcal u^{(R_m)} \, + \, (v^{(R_m)} - z) \, \Leb^d.
	\end{align*}
	Here, let us recall that the commutator $[\Bbb,\chi] \coloneqq \Bcal (\chi \,\frarg) - \chi\Bcal$ is a differential operator of order at most $k_\Bbb -1$ (with coefficients involving the coefficients of $\Bbb$ and the derivatives of $\chi$ of order less or equal than $k_\Bbb$). By this token, if $h \ge H(m)$, we may estimate the total variation of $q_{h,m}$ as
	\begin{align*}
		|q_{h,m}|(Q) & \lesssim_{q,\Bbb} \|\phi_m\|_{k,\infty} \cdot \|w_h - u\|_{\Wrm^{k_\Bbb -1,q}(Q)} + |\Bcal w_h|(Q) +  \\ 
					   & \qquad + \|\eta_m\|_{k,\infty} \cdot \|u^{(R_m)} - u\|_{\Wrm^{k_\Bbb -1,q}(Q)}  + |\Bcal u + z\, \Leb^d|(Q)  \\
					   & \qquad \qquad + |\Bcal u^{(R_m)}|(Q) + \|v^{(R_m)} - z\|_{\Lrm^q(Q)} \\
					   & \!\!\!\!\!\!\!\!\!\!\!\stackrel{\eqref{eq:p3},\eqref{eq:p4},\eqref{eq:t1},\eqref{eq:t2}}\lesssim \frac 3m +  5|\gamma_m|(Q); 
	\end{align*} 
	whence it is established that $(q_{h,m})_{h \ge h(m)}$ is uniformly bounded in $\M(Q;W)$. In fact, we get that $\limsup_{m \to \infty} |q_m|(Q \setminus Q_{1 - \frac 1m}) = 0$; this follows from the property $\kappa (\partial Q) = 0$. Therefore, passing to further subsequence of the $h$'s if necessary (not relabeled), we may assume that 
	\[
		q_{h,m} \; \toY \; \bm \tilde {\bm \sigma} \; \text{in $\Y(Q;W)$} \quad \text{as $h \to \infty$}, \qquad \lambda_{\tilde{ \bm\sigma}}(\partial Q) = 0.
	\]
	On the other hand, observe that $q_{h,m} \mres (Q_{1 - \frac 4m})  = \tilde \gamma_{h(m)} + (v^{R_m} - z)$ and hence, by Lemma~\ref{prop:Lpclose} and the locality of Young measures, it must hold $\tilde{ \bm\sigma} \equiv \bm \sigma$ in $\Y(Q_{1 - \frac 4m};W)$. Since this holds for all $h \in \Nbb$ and neither $\tilde{\bm \sigma}$ or $\bm \sigma$ charge the boundary $\partial Q$, it follows  that 
	\begin{gather}
		\label{eq:u1}	q_{h,m} \; \toY \; \bm {\bm \sigma} \; \text{in $\Y(Q_{1 - \frac 1m};W)$ as $h \to\infty$},\\
		\label{eq:u2} q_{h,m} \equiv v^{(R_m)} \, \Leb^d  + \Bcal u^{(R_m)} \equiv \gamma_{j(m)}  \quad \text{as measures on $Q \setminus Q_{1 - \frac 1m}$}.
		\end{gather}
	In particular, the uniform bound above and~\eqref{eq:preiss1} ensure that we may find another subsequence $h(m) \ge H(m)$ satisfying
	\begin{equation}\label{eq:v1}
		\left|\ddprB{\mathbbm 1_{Q} \otimes h_p, \bm \delta_{q_{h(m),m}}} - \ddprB{\mathbbm 1_{Q} \otimes h_p, \bm \sigma }\right| \le \frac{1}{m} \quad \text{for all $p \le m$}.
	\end{equation}

%

	\textit{Step~4.~Gluing together and generating $\bm \nu$.} So far, we have constructed generating sequences for \emph{specific} tangent Young measures of $\bm \nu$ on every $x$ where there is a cube $Q_{x,m} \subset \Ocal_m$.  
	The rest of the proof can be summarized in the following two steps: First, we construct an $\A$-free sequences by gluing together the $Q \to Q_{x,m}$ push-forwards of each $q_{h(m),m}$. Second, we show the new global sequence is uniformly bounded.
	
	
	
	\textit{Step~5a. Gluing the generating sequences.} For $x \in \R^d$ and $r>0$ we define the map $\Grm_{x,r} (y) = (\Trm_{x,r})^{-1} = x + ry$, which is defined for all $y \in \R^d$. Fix a cube $Q_{x,m}$ in $\Ocal_m$ and define an $\A$-free measure there by setting 
	\begin{equation}\label{eq:pott}
		U_m  \coloneqq C_m^{-1} \cdot \Grm_{x,R_m}  [q_{h(m),m}] \in \M(Q_{x,m};W), \qquad C_m \coloneqq c_{j(m)}.
	\end{equation}
	Notice  that
	\begin{align*}
	U_m  &\equiv C_m^{-1} \cdot  \Grm_{x,R_m} [\gamma_m] \\ 
						&\equiv C_m^{-1} \cdot  C_m (\Trm_{x,R_m} \circ \Grm_{x,R_m}) [\mu] \\
						& \equiv \mu \quad \text{as measures on $Q_{x,m} \setminus Q_{x,m}'$,}
	\end{align*}
where $Q_{x,m}'$ is the concentric sub-cube of $Q_{x,m}$ given by $z + (1 - \frac 1m)(Q_{x,m} - x)$.
	Therefore, the measure defined as 
	\[
		\tau_m(\mathrm dy) \coloneqq 
		\begin{cases} 
			U_m(\mathrm dy) & \text{if $y \in Q_{x,m}$}  \\
			\mu(\mathrm dy) & \text{if $y \in \Omega \setminus O_m$}
		\end{cases}
	\]
	is well-defined in $\Omega$. Moreover, it is also  $\A$-free on $\Omega$ (cf.~\eqref{eq:u2}) and its total variation in $\Omega$ can be controlled as follows (recall that the push-forward is mass preserving)
	\begin{align*}
		|\tau_m|(\Omega) & \le \sum_{Q_{x,m} \in \Ocal_m} |U_m|(Q_{x,m}) + |\mu|(\Omega \setminus O_m) \\
					  & \!\!\stackrel{\eqref{eq:pizza3}}\le \sum_{Q_{x,m} 
					  	\in \Ocal_m} C_m^{-1} |q_{h(m),m}|(Q) \\
					  & \!\!\stackrel{\eqref{eq:pizza3}}\le \sum_{Q_{x,m} \in \Ocal_m} C_m^{-1} \cdot \bigg( \frac 3m + 5|\gamma_m|(Q)\bigg)  \\
					  & \le 13 \cdot \sum_{Q_{x,m} \in \Ocal_m} \ddprB{1 + |\frarg|,\bm\nu \mres Q_{x,m}}  \\
					  & \lesssim   \ddprB{1 + |\frarg|,\bm\nu} < \infty.
	\end{align*}
	
	\textit{Step~4b. The new $\A$-free sequence generates $\bm \nu$.} This last step consists of checking that $\bm \nu$ is indeed an $\A$-free Young measure in $\Omega$. In light of the previous steps, it suffices to check that $\tau_m$ generates $\bm \nu$ in $\Omega$. 
	First, we estimate how close $U_m$ is from generating $\bm \nu$ on $Q_{x,m}$. Fix $\phi \in \Crm(\cl \Omega)$. Every cube $Q_{x,m} \in \Ocal_m$ has diameter at most $\sqrt{d}m^{-1}$ and therefore there exists a modulus of continuity (depending solely on $\phi$) such that $\|\phi(x) - \phi\|_\infty(Q_{x,m}) \le \omega(m^{-1})$ for all $Q_{x,m} \in \Ocal_m$; the same bound holds for any dilation of $\phi$ on the corresponding dilation of $Q_{x,m}$.   
	
	Let $p \in \Nbb$ and let $M_{p}$ to be the linear growth constant of $h_p$. We define 
	\[
		\delta(m) \coloneqq \ddprb{ 1 + |\frarg|,\bm \nu} \Big(15 M_p \omega(m^{-1})   + 2{\|\phi\|_\infty}{m^{-1}} \Big)   \searrow 0 \qquad \text{as $m \to \infty$}.
	\]
	Let $m \ge p$. Regarding $U_m$ as an element of $\M(\Omega;W)$ through the trivial extension by zero, we obtain the estimate
	\begin{align*}
		\big|\ddprb{\phi \otimes  h_p,&\bm \delta_{U_m}}   - \ddprb{\phi  \otimes h_p, \bm \nu \mres Q_{x,m}}\big| \\
								       & = C_m^{-1} \, \left|\ddprb{(\phi \circ \Grm_{x,R_m}) \otimes h_p, \bm \delta_{q_{h(m),m}} } - C_m\ddprb{\phi  \otimes h_p, \bm \nu \mres Q_{x,m}}\right| \\
								  & \le  C_m^{-1} \, \big|\ddprb{(\phi \circ \Grm_{x,R_m})(0) \otimes h_p, \bm \delta_{q_{h(m),m}}}  - C_m\ddprb{\phi  \otimes h_p, \bm \nu\mres Q_{x,m}}\big| \\
								  & \qquad  + C_m^{-1} \omega(m^{-1}) \cdot M_p [ \Leb^d(Q) + |q_{h(m),m}|(Q)] \\
								  & \stackrel{\eqref{eq:v1}}\le C_m^{-1} \Big( \left|\dprb{(\phi \circ \Grm_{x,R_m})(0) \otimes h_p, \bm \sigma}  - C_m\ddprb{\phi  \otimes h_p, \bm \nu\mres Q_{x,m}}\right| \\
								  & \qquad  + {\|\phi\|_\infty}{m^{-1}} +  \omega(m^{-1}) \cdot M_p [ \Leb^d(Q) + |q_{h(m),m}|(Q)] \Big)\\
								  & \stackrel{\eqref{eq:p1}}\le C_m^{-1} \Big( \left|\ddprb{(\phi \circ \Grm_{x,R_m})(0) \otimes h_p, \bm \nu^{(R_m)}}  - C_m\ddprb{\phi  \otimes h_p, \bm \nu\mres Q_{x,m}}\right|  \\
								  & \qquad  + 2{\|\phi\|_\infty}{m^{-1}} +  \omega(m^{-1}) \cdot M_p [ \Leb^d(Q) + |q_{h(m),m}|(Q)] \Big)\\
								  & \le C_m^{-1} \bigg( \left|\ddprb{(\phi  \circ \Grm_{x,R_m}) \otimes h_p, \bm \nu^{(R_m)}} - C_m \ddprb{\phi  \otimes h_p, \bm \nu\mres Q_{x,m}}\right|  \\
								  &  + 2{\|\phi\|_\infty}{m^{-1}}  +  \omega(m^{-1}) \cdot  M_p\big[\Leb^d(Q) + |q_{h(m),m}|(Q) \\
								  & \qquad +  C_m\ddprb{ 1 + |\frarg|,\bm \nu \mres Q_{x,m} }\big] \Big) \\
								  & \le     2{\|\phi\|_\infty}{m^{-1}}\ddprb{|\frarg|,\bm \nu \mres Q_{x,m}} + \omega(m^{-1}) \Big( 15 M_p \ddprb{ 1 + |\frarg|,\bm \nu \mres Q_{x,m} } \Big).
	\end{align*}
	Therefore, adding up these estimates for each cube $Q_{x,m}$ on $\Qcal_m$  yields
	\begin{align*}
		|\ddprb{\phi \otimes h_p, \bm \delta_{\tau_m}}  & - \ddprb{\phi \otimes h_p, \bm \nu}  |  \\
									  & \le \sum_{Q_{x,m} \in \Qcal_m} |\ddprb{\phi \otimes h_p, \bm \delta_{U_m}}  -  \ddprb{\phi \otimes h_p, \bm \nu \mres Q_{x,m}}| + |\mu|(O_m)\\ 
				                      & \stackrel{\eqref{eq:pizza3}}\le \delta(m).
	\end{align*}
	This shows that $\ddprb{\phi \otimes h_p, \bm \delta_{\tau_m}} \to \ddprb{\phi \otimes h_p, \bm \nu}$ as $m \to \infty$, and, in particular, this holds for $\phi = \psi_p$ for any $p \in \Nbb$. 
	
	\proofstep{Conclusion.} Since the family $\{\psi_p \otimes h_p\}_{p \in \Nbb}$ separates $\E(\Omega;W)$, we conclude that the (uniformly bounded) sequence of $\Acal$-free measures $\{\tau_m\}$ generates $\bm \nu$, i.e., 
	\[
		\Acal \tau_m = 0 \quad \text{and} \quad \tau_m \, \toY \, \bm \nu \; \text{on $\Omega$}. 		
	\] 
	This finishes the proof.\qed

\subsection{Proof of Corollary~\ref{cor:pure}} 
If $\bm \nu \in \Y_\Acal(\Omega)$ is such that $\lambda(\partial \Omega)$, then from Theorem~\ref{thm:local} we may assume that there exist tangent $\Acal$-free measures of $\bm \nu$ at $(\Leb^d + \lambda)$-a.e. in $\Omega$. The proof follows from the sufficiency part of the proof above, in particular from Step~4. The \emph{recovery sequence} $\tau_m$ constructed there is $\Acal$-free and it also generates $\bm \nu$. \qed

\subsection{Proof of Theorem~\ref{thm:local2}}\label{sec:si} The necessity follows from Theorem~\ref{thm:local} and the fact that,  if $\Acal$ is an annihilator of $\Bcal$, then
\begin{enumerate}[(a)]
	\item $\Acal$-quasiconvexity is equivalent to $\Bcal$-gradient quasiconvexity for locally bounded integrands,
	\item $\Bcal\!\Y(\Omega) \subset \Y_\Acal(\Omega)$.
\end{enumerate}
\proofstep{Sufficiency.} Due to a small clash of notation, we re-write assumption (i) as 
	\[
\Bcal u_0 = \dprb{\id,\nu}   \, \Leb^d  \, + \, \dprb{\id,\nu^\infty}\, \lambda\,, \quad u_0 \in \Mcal(\Omega;V).
\] 
We will show that there exists a sequence $\{V_m\} \subset \Mcal(\Omega;V)$ with $\{\Bcal V_m\} \subset \Mcal(\Omega;W)$ that generates $\bm \nu$. This will be deduced from the constructions contained in the proof of the sufficiency of Theorem~\ref{thm:local}, but first we need to recall the following fact of blow-downs: if $w \in \Mcal(Q;V)$ and $\Bcal w \in \Mcal(Q;W)$, then the blow-down of $\Bcal w$, centered at $x$ and at scale $r$, is given by
\[
\Grm_{r,x} [\Bcal w] = \Bcal w^{r,x}, \quad \text{where \; } w^{r,x}(y) \coloneqq r^{-k_\Bbb} \, \Grm_{r,x}[w].
\]
Notice that $w^{r,x} \in \Mcal(Q_r;V)$ and $\Bcal w^{r,x} \in \Mcal(Q;W)$. Moreover, the mass preserving property of push-forwards gives
\begin{equation}\label{eq:bd}
|\Bcal w^{r,x}|(Q_{r,x}) = |\Bcal w|(Q) 
\end{equation}
We are now ready to prove the assertion. 
Let us recall from (b) that $\bm \nu$ satisfies the sufficiency assumptions of Theorem~\ref{thm:local} and hence we may apply the elements contained in its proof to $\bm \nu$. In particular, the sequence $\{\gamma_j\}$ (introduced in~\eqref{eq:cro1}) has elements of the form $\Bcal u^{(r_j)} + v^{(r_j)}$. However, since by assumption $[\bm \nu] = \Bcal u_0$, Corollary~\ref{cor:shift} says that we may assume $z = 0$ in~\eqref{eq:cro1}.
In particular, keeping the notation of the previous proof, the sequence $\{q_{h,m}\}$ (defined in Step~3) is a sequence of $\Bcal$-gradients. Indeed, since $z = 0$ it follows that $v^{(R_m)} = C_m \Trm_{x,R_m}[\Bcal u_0 \mres Q_{x,m}] - \Bcal u$ and by the (inverse) property of blow-downs we get
\begin{align*}
	q_{h,m} &  = 	\Bcal(\phi_m (w_{h} - u)) + \Bcal u + \Bcal (\eta_m(u^{(R_m)} - u)) \\
	& \qquad + \Bcal (C_m\,\Trm_{x,R_m}[(R_m)^{k_\Bbb}u_0 \mres Q_{x,m}]) - \Bcal u^{(R_m)}.
\end{align*}
It follows that the sequence $\{\tau_m\} \subset \Mcal(\Omega;W)$ (defined in in p. 44), which generates our Young measure $\bm \nu$, has  the form
\[
\tau_m(\mathrm dy) \coloneqq 
\begin{cases} 
	U_m^x(\mathrm dy) & \text{if $y \in Q_{x,m}$}  \\
	\Bcal u_0(\mathrm dy) & \text{if $y \in \Omega \setminus O_m$}
\end{cases}
\]
where, according to~\eqref{eq:pott} and ignoring the $x$-dependence, the $U_m$ are defined as
	\[
	U_m  \coloneqq C_m^{-1}q_{h(m),m}.
	\]
In particular, the identity for blow-downs above implies that $U_m = \Bcal W_m$, where $W_m$ is the Radon measure given by
\begin{align*}\label{eq:derivation}
	 W_m & \coloneqq C_m^{-1}(R_m)^{-k_\Bbb}\,\Grm_{R_m} [\phi_m (w_{h(m)} - u) + u + \eta_m(u^{(R_m)} - u) - u^{(R_m)}]\\
	& \qquad + [u_0 \mres Q_{m}] .
\end{align*}
By construction we have $W_m \equiv  u_0$ (as measures) on a neighborhood of $\partial Q_{x,m}$. This compatibility across the partition ensures that 
\[
	V_m \coloneqq \begin{cases} 
	W_{m}(\mathrm dy) & \text{if $y \in Q_{x,m}$}  \\
		u_0(\mathrm dy) & \text{if $y \in \Omega \setminus O_m$}
	\end{cases}, \quad m \in \Nbb,
\] 
defines a sequence of Radon measures in $\Mcal(\Omega;V)$, with $\Bcal V_m \in \Mcal(\Omega;W)$ and such that
\[
	\Bcal V_m = \tau_m \toY \bm \nu \quad \text{on $\Omega$}.
\]
This finishes the proof.\qed

%
%

	\section{Proof of the dual characterizations}\label{sec:main}

\subsection{The convexity of $\Y_{\Acal,0}(\mu,\Omega)$}
Let $\mu \in \M(\Omega;W)$ be an $\A$-free measure. We define the set 
\[
	\Y_{\Acal,0}(\mu,\Omega) \coloneqq \setB{\bm \nu \in \Y_\Acal(\Omega)}{\lambda(\partial \Omega) = 0, [\bm \nu] = \mu}.
\]

The proof of the following proposition is contained in Lemma~5.3 of~\cite{baia2013}. There the authors state their main results under additional assumptions. However, the proof of this specific proposition makes not use of such assumptions and can be worked out by verbatim in our context. 

\begin{proposition}\label{prop:closed}
	The set $\Y_{\Acal,0}(\mu,\Omega)$ is weak-$*$ closed in $\E(\Omega;W)^*$.
\end{proposition}

The main in step towards the proof of the characterization result Theorem~\ref{thm:char} rests in showing the following convexity property. Once this is established the proof of Theorem~\ref{thm:char} follows by relaxation argument and the geometric version of Hahn--Banach's Theorem argument. 

\begin{theorem}\label{thm:convex} The set $\Y_{\Acal,0}(\mu,\Omega)$ is a convex set.
\end{theorem}
\begin{proof}
Fix $0 < \theta < 1$ and let 
\[
\bm \nu_1 = (\nu_1,\lambda_1,\nu^\infty_1), \; \bm \nu_2 =(\nu_2,\lambda_2,\nu^\infty_2) \in \Y_{\Acal,0}(\mu,\Omega).
\]
We also write $\bm \nu_\theta \coloneqq \theta \bm \nu_1 + (1 - \theta)\bm \nu_2 \in \E(\Omega;W)^*$. 
Our goal is to show that $\bm \nu_\theta$ is an $\A$-free Young measure on $\Omega$. To show this we will construct a sequence of $\A$-free measures om $\Omega$ which generate the functional  $\bm \nu_\theta$.  

Since this will be a fairly long and technical proof we will begin by describing a brief program of the proof. The foundation of our proof lies in a careful inspection of the infinitesimal qualitative behavior of points $x \in \Omega$ with respect to our Young measures $\bm \nu_1, \bm \nu_2$. The qualitative understanding of the set of tangent Young measures of $\bm \nu_i$ ($i = 1,2$) at a given $x \in \Omega$ will be decisive in the choice of construction of an $\A$-free recovery sequence for $\bm \nu_\theta$ about that point. Once \emph{every} point and their local constructions are established, the idea is to use Besicovitch's covering theorem to build a partition of $\Omega$ into disjoint tiles, each of which retrieves the infinitesimal properties of $\bm \nu_i$ and hence the recovery sequences of $\bm \nu_\theta$ about their center points. The one but last step is to glue the aforementioned $\A$-free recovery sequences from each tile into a globally $\A$-free sequence which generates an arbitrarily close a piece-wise constant approximation of $\bm \nu_\theta$. The conclusion of the argument then follows from a diagonalization argument between the larger scale of piece-wise constant approximations of $\bm \nu_{\theta}$ where we glue the recovery sequences,  and the smaller scale where the corresponding recovery sequences are effectively constructed.

\textit{\textit{Step~1. Qualitative analysis of points.}} 

Since we are trying to capture the fine properties of $\bm \nu_1$ and $\bm \nu_2$ simultaneously, it will be convenient to define the measure $\Lambda \coloneqq \lambda_1^s + \lambda^s_2$, which is a suitable substitute candidate to keep track of the interactions between singular points of $\lambda_1$ and $\lambda_2$.     
We start by distinguishing regular points and singular points. It follows from the Radon--Nykod\'ym theorem that at $(\Leb^d + \Lambda)$-almost every $x \in \Omega$ one of the following properties hold: either 
\begin{equation}
x \in \reg(\Omega) \coloneqq \setBB{x \in \Omega}{	\frac{\dd \Lambda}{\dd \Leb^d}(x) = \lim_{r \todown 0} \frac{\Lambda(Q_r(x))}{(2r)^d}= 0}
\end{equation}
is a regular point, or
\begin{equation}
x \in \sing(\Omega) \coloneqq \setBB{x \in \Omega}{		\frac{\dd \Leb^d}{\dd \Lambda}(x) = \lim_{r \todown 0} \frac{(2r)^d}{\Lambda(Q_r(x))} = 0}
\end{equation}
is a singular point. 
Throughout this proof we shall call points with the first property (which holds $\Leb^d$-almost everywhere) regular points, and points satisfying the second property (which holds $\Lambda$-almost everywhere) will be called singular points; we shall only consider points $x \in \Omega$ that are either regular or singular points. In addition, we may assume without any loss of generality that the limits
\[
	\frac{\dd \lambda_i^s}{\dd \Lambda}(x) = \lim_{r \todown 0}\frac{\lambda_i^s(Q_r(x)}{\Lambda(Q_r(x))}, \quad i = \{1,2\},
\]
exist at every singular point $x \in \Omega$. Next, we further partition $\sing(\Omega)$ into sets which render precise information about the size relation between $\lambda_1$ and $\lambda_2$. More precisely, we split $\sing(\Omega)$ into sets $G_0 \cup G_1 \cup G_\infty \cup N$, where
\begin{gather*}
G_0 \coloneqq \setBB{x \in \Omega}{\frac{\dd \lambda_1^s}{\dd \Lambda}(x) = 0}, \\
G_1 \coloneqq \setBB{x \in \Omega}{\frac{\dd \lambda_1^s}{\dd \Lambda}(x) \in (0,1) }, \\
G_\infty \coloneqq \setBB{x \in \Omega}{\frac{\dd \lambda_1^s}{\dd \Lambda}(x) = 1},
\end{gather*}
and $\Lambda(N) = 0$.  
If we set
\[
	g_1 = \mathbbm 1_{G_1 \cup G_\infty}\cdot \frac{\dd \lambda_1^s}{\dd \lambda_2^s} \quad \text{and} \quad 	g_2 = \mathbbm 1_{G_0 \cup G_1} \cdot \frac{\dd \lambda_2^s}{\dd \lambda_1^s},
\]
then, up to modifying $N$, we may assume that $g_1, g_2$ are $\Lambda$-measurably continuous and
\[
	\begin{cases}
		 x \in G_0 & \Longrightarrow \quad g_1(x) = \displaystyle \lim_{r \todown 0} \frac{\lambda_1^s(Q_r(x))}{\lambda^s_2(Q_r(x))} = 0, \\[10pt]
		x \in G_1 & \Longrightarrow \quad  g_1(x) = g_2(x)^{-1} = \displaystyle\lim_{r \todown 0} \frac{\lambda_1^s(Q_r(x))}{\lambda^s_2(Q_r(x))} \in (0,\infty),	\\[10pt]
		x \in G_\infty & \Longrightarrow \quad g_2(x) = \displaystyle \lim_{r \todown 0} \frac{\lambda_2^s(Q_r(x))}{\lambda^s_1(Q_r(x))} = 0.	
	\end{cases}
\]

\textit{\textit{Step~1a. Tangential properties of singular points.}} So far we have separated regular and singular points, and the latter by their weights with respect to  $\lambda_1$ and $\lambda_2$. The next step is to separate points in $\sing(\Omega)$ with respect to the qualitative behavior of $\Tan(\Lambda,x)$.
\begin{enumerate}
	\item If there exists a tangent measure $\tau \in \Tan(\Lambda,x)$ which does not charge points, i.e., 
	\[
		\text{$\tau(\{y\}) = 0$ for all $y \in \R^d$},
	\] then
	we write $x \in \Rcal$. 
	Every $x \in \Rcal$ has the following property (see Corollary~\ref{cor:theta}): if  $\Theta \in (0,1)$, $g$ is a $\Lambda$-measurable map, and $x$ is a $\Lambda$-Lebesgue point of $g$, then there exist (a) a sequence of infinitesimal radii $r_h \todown 0$ and (b) a sequence of open Lipschitz sets $D_h \subset Q_{r_h}$ satisfying
	\[
		\Lambda(x + \partial D_h) = 0, \qquad \lim_{h \to \infty} \frac{\Lambda(x + D_h)}{\Lambda(Q_{r_h}(x))} = \Theta,
	\]
	and
	\[
		\lim_{h \to \infty}  \aveint{x + D_h}{} |g - g(x)| \dd \Lambda = 0.
	\]
	In particular, if $x \in G_1$, then 
	\[
		\lim_{h \to \infty} \frac{\Lambda(x + D_h)}{\Lambda(Q_{r_h}(x))} = \lim_{h \to \infty}    \frac{\lambda^s_1(x + D_h)}{  \lambda^s_1(Q_{r_h}(x))} = \lim_{h \to \infty}    \frac{\lambda^s_2(x + D_h)}{\lambda^s_2(Q_{r_h}(x))} = \Theta.
	\]
	
	\item If otherwise~(1) does not hold for any tangent measure of $\Lambda$ at $x$, we write $x \in \Scal$. It follows from Lemma~\ref{lem:tantan} and the fact that \emph{blow-ups of blow-ups} are blow-ups (see Theorem~2.12 in \cite{preiss1987geometry-of-mea})  that
	\[
		x \in \Scal \quad \Longrightarrow \quad \delta_0 \in \Tan(\Lambda,x). 
	\]
\end{enumerate}

{\textit{Step~1b. Selection of points with Lebesgue-type properties.}} We now turn to the selection of points which later shall be the centers of the tile partitions. As usual let $\{f_{p,q}\}_{p,q \in \Nbb} \subset \E(\Omega;W)$ be the family from Lemma~\ref{lem:separation} which separates points in $\E(\Omega;W)^*$.

Up to removing a set of $\Leb^d$-measure zero, we may assume that every $x \in \reg(\Omega)$ is a Lebesgue point of the maps
\[
\Big\{x \mapsto \dprb{f_{p,q} , \nu_i}_x + \dprb{f_{p,q}^\infty  , \nu^\infty_i}_x\, \ac\lambda_i(x)\Big\} \qquad i = 1,2; \quad p,q \in \Nbb.
\]
About singular points $x \in \sing(\Omega)$, we shall be more careful and set $B^\infty_i \subset \sing(\Omega)$ to be the set of $\lambda_i^s$-Lebesgue points of the family of maps
\[
\Big\{ x \mapsto \dprb{f_{p,q}^\infty , \nu_i}_x\Big\}  \qquad i = 1,2; \quad p,q \in \Nbb.
\]
Each $B_i^\infty$ has full $\lambda^s_i$-full measure on $\Omega$ and hence $B_1^\infty \cup B_2^\infty$ has full $\Lambda$-measure on $\Omega$. Therefore, in what follows there will be no loss of generality in assuming that $\sing(\Omega) = B_1^\infty \cup  B_2^\infty$; this union may not be disjoint.\\

{\textit{Step~2. Building a partition of cubes with good fine properties.}} Let $m \in \Nbb$, in this step we will address the construction of a full $\Lambda$-measure partition of $\Omega$ with $\BigO(m^{-1})$-asymptotic approximation Lebesgue-type properties. To begin, let us define a fine cover of $L \coloneqq \reg(\Omega) \cup \sing(\Omega)$. 
At every $x \in L$ we define 
\[
	\rho_m(x) \coloneqq \sup\setBB{0 \le r \le \frac 1m}{\text{$r$ satisfies the $(\Pcal_m(x))$ property}}.
\]
A radius $r$ is said to satisfy $(\Pcal_m(x))$ provided the following continuity properties hold for $i = 1,2$ and all indexes $p,q \le m$:

If $x \in \reg(\Omega)$, then 
\begin{gather}
\label{eq:regular}
\frac{\Lambda(Q_r(x))}{(2r)^d} \le \frac 1m, \\
\label{eq:but1}
\aveint{Q_r(x)}{} \big|\dprb{f_{p,q}, \nu_i}_y - \dprb{f_{p,q}, \nu_i}_x\big| \dd y \le \frac 1m, \\
\label{eq:butt1}\aveint{Q_r(x)}{} \big|\dprb{f^\infty_{p,q}, \nu^\infty_i}_y \cdot \ac\lambda_i(y) - \dprb{f^\infty_{p,q}, \nu^\infty_i}_x \cdot \ac\lambda_i(x)\big|  \dd y \le \frac 1m.
\end{gather}

 If $x \in \sing(\Omega)$, then
\begin{gather}
\label{eq:singular}
\int_{Q_r(x)} \dprb{1 + |\frarg|,\nu}_y  \dd y + \int_{Q_r(x)} \ac\lambda_i(y)  \dd  y \;  \le \; \frac 1m \cdot {\Lambda(Q_r(x))},\\
\label{eq:but2}
\aveint{Q_r(x)}{} \big|\dprb{f^\infty_{p,q} , \nu^\infty_i}_y - \dprb{f^\infty_{p,q} , \nu^\infty_i}_x\big| \dd \lambda_i^s(y) \le \frac 1m, \quad  \; x \in B_i^\infty,\\
\label{eq:11} \frac {\lambda_1^s(Q_r(x))}{\Lambda(Q_r(x))} \le \frac 1m   \quad \text{if $x \in G_0$,} \\
\label{eq:13} \frac{\lambda^s_2(Q_r(x))}{\Lambda(Q_r(x)) } \le \frac 1m \quad \text{if $x \in G_\infty$.} 
\end{gather}

If $x \in \Rcal \cap G_1$, then we require
\begin{equation}
\label{eq:12} 
\left| \frac{\lambda_i(x + D_r)}{\lambda_i(Q_{r}(x))} - \Theta \right|  \le \frac 1m, \quad \Lambda(x + \partial D_r) = 0.
\end{equation}
If $x \in G_0$ or $x \in G_\infty$, then we can only find $D_r$ satisfying~\eqref{eq:12} for $\lambda_2$ and $\lambda_1$ respectively. 

Lastly, if $x \in \Scal \cap G_1$, then
\begin{equation}\label{eq:14}
 \frac{\lambda_i(A_r)}{\lambda_i(Q_r(x))} \le \frac 1m, 
\end{equation}
where 
\[
	A_r \coloneqq Q_r(x) \setminus \cl{Q_{s_r}(x)}, \qquad \Lambda(\partial A_r) = 0, \quad\frac {s_r}{r} \le \frac 1m.
\]
Moreover, $s_r$ can be chosen sufficiently small so that  
\begin{equation}\label{eq:15}
\|\Delta_{\pm s_r}   u - u\|_{\Wrm^{k_\Bbb -1,1}(Q_x)}  \le \frac{\Lambda(Q_r(x))}{m}, 
\end{equation}
where $\mu \mres Q_r(x) =  \Bcal u + v \Leb^d$ is the decomposition provided by Lemma~\ref{lem:Helm} for $\mu$ on $Q_r(x)$. Here we have used the short notation  
\[
\Delta_{\pm h } w \coloneqq w(\frarg \pm h e_1), \quad s_x \coloneqq s_{r}(x),
\]
for the translations of a function  $w$. 

Now, this is indeed a large amount of smallness conditions to keep track, but they are all fundamental if one wishes to avoid (trivial) partitions which do not reflect the behavior of $\bm \nu_1, \bm \nu_2$ appropriately.

\textbf{Claim~1.} $\rho_m(x) > 0$ for all $x \in L$. 

\proofstep{Proof of Claim~1.} Most of the properties are easy to check: Properties~\eqref{eq:regular}-\eqref{eq:butt1} and ~\eqref{eq:singular}-\eqref{eq:13} follow directly from the construction and the Lebesgue properties discussed in Step~1b. Property~\eqref{eq:12} is a consequence of Step~1a(1). We focus in showing~\eqref{eq:14}-\eqref{eq:15} which will follow from the fact that $\delta_0 \in \Tan(\Lambda,x)$. Indeed, in this case we may a  sequence of infinitesimal radii $r_j \todown 0$ such that
\[
	\gamma_j \coloneqq \frac{1}{\Lambda(Q_{r_j}(x))} \cdot \Trm_{x,r_j}[\Lambda] \toweakstar \delta_0 \quad \text{locally in $\M(\R^d)$}. 
\] 
 Then, by the strict convergence of the blow-up sequence we deduce that 
\[
	\lim_{j \to \infty} \Lambda(Q_{sr_j}) = \lim_{j \to \infty} \gamma_j(Q_{sr_j}(x)) = \lim_{j \to \infty} \frac{\Lambda(Q_{sr_j}(x))}{\Lambda(Q_{r_j}(x))} = 1 \quad \forall \; s \in (0,1).
\]
In particular, since $x \in G_1$, we conclude that
\[
	\lim_{j \to \infty} \frac{\Lambda(Q_{sr_j}(x))}{\Lambda(Q_{r_j}(x))} = \lim_{j \to \infty} \frac{\lambda_i(Q_{sr_j}(x))}{\lambda_i(Q_{r_j}(x))} = 1, \qquad i = 1,2.
\] 
Choosing $s \le \frac 1m$ in a way that $s_{r_j} \coloneqq s r_j$ satisfies the required properties for $A_{r_j}$ and $Q_{r_j}(x)$ (this can be done by slightly modifying each $r_j$ in the blow-up sequence), we exhibit an infinitesimal sequence $r_j$ (and their associated $s_{r_j}$) satisfying~\eqref{eq:14}-\eqref{eq:15}. 

This proves the claim.\qed

In particular, the cover 
\[
	\Qcal_m \coloneqq \setB{Q_r(x)}{ x \in L, \, 0 < r \le \rho_m(x) \text{ with $\Lambda(\partial Q_r(x)) = 0$}}
\] 
conforms a fine cover of $L$ to which we may apply Besicovitch's Covering Theorem: 
There exists a sub-cover $\Ocal_m \subset \Qcal_m$ of disjoint cubes satisfying
\begin{equation}\label{eq:but3}
	\Lambda( \Omega \setminus O_m) = 0 \quad \text{and} \quad \Lambda(\partial Q_x) = 0 \quad \text{for all $Q_x \in \Ocal_m$}.
\end{equation}
Here, we have set $O_m \coloneqq \cup_{Q_x \in \Ocal_m} Q_x$. 

{\textit{Step~3. Piece-wise homogeneous approximations of $\bm \nu_i$.}} The idea behind defining $\Ocal_m$ is to construct a \emph{piece-wise homogeneous} approximation  of $\bm \nu_1, \bm \nu_2$ of order $\frac 1m$ as follows: 
Fix $i \in \{1,2\}$ and define, through duality, a sequence of functionals $\{{\bm \nu^{(m)}_i}\}$ in $\E(\Omega,W)^*$ acting as 
\begin{align*}
\ddprB{f,{\bm \nu^{(m)}_i}}& \coloneqq \sum_{\substack{x \in \reg(\Omega) \\ Q_x \in \Ocal_m}} 
\bigg(\int_{Q_x} \dprb{f , \nu_i}_x \dd y \\ 
& \qquad + \int_{Q_x} \dprb{f^\infty , \nu_i^\infty}_x \, \ac\lambda_i(x)  \dd y \bigg) \\ 
& \qquad + \sum_{\substack{x \in  B_i^\infty \\ Q_x \in \Ocal_m}} \int_{Q_x} \dprb{f^\infty , \nu_i^\infty}_x  \dd \lambda_i^s(y).
\end{align*}
The fact that these functionals are in fact Young measures  follows directly from~\eqref{eq:but3}, the weak-$*$ measurability properties of $\bm \nu_1$ and $\bm \nu_2$, and the fact that simple Borel maps are measurable with respect to any Radon measure.  

\textbf{Claim~2.} As $m \to \infty$ it holds that
\[
	\bm \nu_i^{(m)} \toweakstar \bm \nu_i \; \text{in $\E(\Omega,W)^*$},\quad i = 1,2.
\]
\proofstep{Proof of Claim~2.} Let $p,q \in \Nbb$ (we shall simply write $f = f_{p,q}$). 
First, we show that 
\begin{equation}\label{eq:claim1}
\begin{split}
\lim_{m \to \infty}	\bigg|& \int_{\Omega}  \dprb{f,\nu_i}  \dd \Leb^d + \int_{\Omega} \dprb{f^\infty,\nu^\infty_i} \dd (\ac\lambda_i\Leb^d) \\ & - \sum_{\substack{Q_x \in \Ocal_m\\x \in \reg(\Omega)}} \int_{Q_x}  \dprb{f,\nu_i}_x \dd y  + \int_{Q_r(x)} \ac\lambda_i(x) \dprb{f^\infty , \nu_i}_x \dd y \bigg| = 0.
\end{split}
\end{equation}
We consider $p,q \le m \in \Nbb$. We may estimate (cf.~\eqref{eq:but3}) the difference of the integrals above by the sum of the two  non-negative quantities
\[
	I_m \coloneqq \bigg|\sum_{\substack{Q_x \in \Ocal_m\\x \in \sing(\Omega)}} \int_{Q_x} \dprb{f,\nu_i} \dd \Leb^d + \int_{\Omega} \dprb{f^\infty,\nu^\infty_i} \dd (\ac\lambda_i\Leb^d)\bigg|
\]
and
\begin{align*}
	II_m & \coloneqq \sum_{\substack{Q_x \in \Ocal_m\\x \in \reg(\Omega)}} \bigg|\int_{Q_x} [\dprb{f,\nu_i}_y - \dprb{f,\nu_i}_x] \dd y  \\
		 & \qquad  + \int_{Q_r(x)} [\dprb{f^\infty,\nu^\infty_i}_y \, \ac\lambda_i(y) - \dprb{f^\infty , \nu_i}_x \, \ac\lambda_i(x)] \dd y \bigg|.
\end{align*}
Using~\eqref{eq:singular} and the linear growth of $|f| \le M_f(1 + |\frarg|)$ we obtain 
\begin{align*}
	I_m & \le [1 + M_f]\sum_{\substack{Q_x \in \Ocal_m\\x \in \sing(\Omega)}} \int_{Q_x} \dprb{1 + |\frarg| , \nu_i} \dd \Leb^d(y) + \int_{Q_r(x)} \ac\lambda_i  \dd \Leb^d(y) \\
	  & \le \frac 1m \sum_{\substack{Q_x \in \Ocal_m\\x \in \sing(\Omega)}} \Lambda(Q_x) \le \frac 1m \Lambda(\Omega).
\end{align*}
It follows that $\lim_{m \to \infty} I_m = 0$. 

On the other hand, we use~\eqref{eq:but1}-\eqref{eq:butt1} to bound $II_m$ as 
\begin{align*}
II_m & \le\sum_{\substack{Q_x \in \Ocal_m\\x \in \reg(\Omega)}} \bigg( \int_{Q_x}  |\dprb{f , \nu_i }_y - \dprb{f,\nu_i}_x| \dd y   \\
	 & 	\qquad +  \int_{Q_x}  | \dprb{f^\infty,\nu^\infty_i}_y \, \ac\lambda_i(y) - \dprb{f^\infty,\nu^\infty_i}_x \cdot \ac\lambda_i(x)|  \dd y \bigg)\\
	 & \le  \frac 2m  \sum_{\substack{Q_x \in \Ocal_m\\x \in \reg(\Omega)}} \Leb^d(Q_x)  \le \frac 2m \Leb^d(\Omega).
\end{align*}
This shows that $\lim_{m \to \infty} II_m = 0$, whence~\eqref{eq:claim1} follows. 

To prove the claim we are left to show that
\[
\lim_{m \to \infty}	\bigg|	\sum_{\substack{Q_x \in \Ocal_m\\x \in B_i^\infty}}\int_{Q_x}   [\dprb{f^\infty , \nu^\infty_i} - \dprb{f^\infty,\nu^\infty_i}_x]  \dd \lambda_i^s  \bigg| = 0.
\]
We may estimate the integrand above, for fixed $m \in \Nbb$,  by
\begin{align*}
	\sum_{\substack{Q_x \in \Ocal_m\\x \in B_i^\infty}}\int_{Q_x} &  |\dprb{f^\infty , \nu^\infty_i} - \dprb{f^\infty,\nu^\infty_i}_x|  \dd \lambda_i^s \\
	&  \stackrel{\eqref{eq:but2}}\le \frac 1m	\sum_{\substack{Q_x \in \Ocal_m\\x \in B_i^\infty}}  \Lambda (Q_x)   
	  \le  \frac 1m \Lambda(\Omega). 
\end{align*}
Since $\{f_{p,q}\}$ separates $\E(\Omega;W)^*$, this proves Claim~2.\qed\\
%
%
%

{\textit{Step~4. Construction of a global $\A$-free recovery sequence.}} Let us fix $m \in \Nbb$. 
Next, we define  candidate recovery sequences for $\bm{\nu}_\theta$ on $Q_x \in \Ocal_m$. This will be done depending on whether $x$ belongs to $\Rcal$ or $\Scal$ where these sets are the ones defined in Step~1a.\\

{\textit{Step~4a. Cubes $Q_x \in \Ocal_m$ centered at $x \in \Rcal \cup \reg(\Omega)$.}}  We recall from step~1a and~\eqref{eq:12} that, if $x \in \Rcal$, then there are 
 open Lipschitz sets $D_x \subset Q_x \Subset \Omega$ satisfying 
\begin{equation}\label{eq:4aR}
	\Lambda(\partial D_x) = 0,
\end{equation}
\begin{equation}\label{eq:4aR2}
	\bigg|\frac{\lambda_i(D_x)}{\lambda_i(Q_r(x))} - \theta \bigg| \le \frac 1m  \quad i = 1,2, \quad \text{whenever $x \in G_1$,}
\end{equation}
\begin{equation}\label{eq:4aR3}
\bigg|\frac{\lambda_2(D_x)}{\lambda_2(Q_r(x))} - \theta \bigg| \le \frac 1m  \quad \text{whenever $x \in G_0$,}
\end{equation}
and
\begin{equation}\label{eq:4aR4}
\bigg|\frac{\lambda_1(D_x)}{\lambda_1(Q_r(x))} - \theta \bigg| \le \frac 1m  \quad \text{whenever $x \in G_\infty$.}
\end{equation}
On the other hand, since $\bm \nu_1, \bm \nu_2$ are $\A$-free Young measures on $\Omega$, we may apply Lemma~\ref{lem:Helm} to find sequences (to avoid adding unnecessary notation, we will omit the $x$-dependence of these sequences) of $\A$-free measures $\{u_j\} \subset \M(D_x;W)$ and $\{v_j\} \subset \M(Q_x \setminus \cl{D_x};W)$ satisfying
\begin{gather}
 \label{eq:4a1}u_j \equiv \mu \; \text{on a neighborhood of $\partial D_x$} \quad \text{and} \quad
u_j \toY \bm \nu_1 \, \text{on $D_x$}, \\
  \label{eq:4a2}	v_j \equiv \mu \; \text{on a neighborhood of $\partial (Q_x \setminus \cl{D_x})$} \quad \text{and} \quad
v_j \toY \bm \nu_2 \, \text{on  $(Q_x \setminus \cl{D_x})$}.
\end{gather}

The same construction applies with for $x \in \reg(\Omega)$ with the exception that we require $D_x \subset Q_x$ to satisfy
\begin{equation}\label{eq:4ar}
	\frac{\Leb^d(D_x)}{\Leb^d(Q_r(x))} = \theta  \quad \text{and} \quad \Leb^d(\partial D_x) = 0.
\end{equation}
It follows from the uniformity of the Lebesgue measure that this can always be achieved for some open Lipschitz $D_x \subset Q_x$; in this case the set $D_x$ can be chosen to be a strip of width $\theta$ or an open concentric cube of $Q_x$ of side $\theta^\frac{1}{d}$.

In what follows we shall write
\[
Q_x^1 = D_x \quad \text{and} \quad Q_x^2 =  Q_x \setminus \cl{D_x}.
\] 
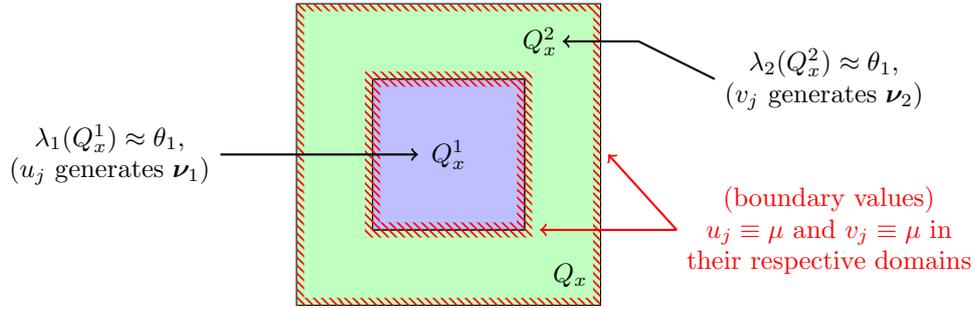
\begin{figure}[h]
	\begin{tikzpicture}
	\draw[fill=green!50,fill opacity=0.5] (0,0) rectangle (4,4); 
	\draw[fill=white] (1,1) rectangle (3,3);
	\draw[fill=blue!50,fill opacity=0.5] (1,1) rectangle (3,3);
	\fill [even odd rule,pattern=custom north west lines,hatchspread=3pt,hatchthickness=0.6pt,hatchcolor=red] (1.1,1.1) rectangle (2.9,2.9) (0.9,0.9) rectangle (3.1,3.1);
	\fill [even odd rule,pattern=custom north west lines,hatchspread=3pt,hatchthickness=0.6pt,hatchcolor=red] (0.1,0.1) rectangle (3.9,3.9) (0,0) rectangle (4,4);
	\draw[<-,thick] (3.5,3.5) -- (4.5,3.5) -- (5.5,3) node[right,align=center]	{$\lambda_2(Q_x^2) \approx \theta_1$, \\ ($v_j$ generates $\bm \nu_2$)};
	\draw[<-,thick] (1.6,2) -- (-1,2) node[left,align=center] {$\lambda_1(Q_x^1) \approx \theta_1$, \\ ($u_j$ generates $\bm \nu_1$)};
	\draw[<-,thick,red] (3.2,1) -- (5,1) node[right,align=center] { (boundary values) \\$u_j \equiv \mu$ and $v_j \equiv \mu$ in \\ their respective domains};
	\draw[<-,thick,red] (4.1,2) -- (5,1);
	\draw (3.2,3.5) node {$Q_x^2$};
	\draw (2,2)  node {$Q_x^1$};
	\draw (3.6,.4)  node {$Q_x$};
	\end{tikzpicture}
	\caption{Qualitative sketch of the construction when there exists a tangent measure $\tau \in \Tan(\Lambda,x)$ which does not charge points (cf. Step~1a).}
\end{figure}

Notice that by construction the measures 
\[
	w_j = w_j^x \coloneqq \mathbbm 1_{Q_x^1} \, u_j + \mathbbm 1_{Q_x^2} \, v_j
\]
are $\A$-free on $Q_x$ for all $j \in \Nbb$. Moreover, the $w_j$'s can be extended by $\mu$ outside $Q_x$ and particular this extension preserves the $\A$-free constraint. Moreover, in virtue of~\eqref{eq:4a1}-\eqref{eq:4a2} and the locality of the weak-$*$ convergence of Young measures it holds that
\begin{equation}\label{eq:w}
	w_j  \, \toY \, \bm \nu_1 \mres \cl{Q_x^1} \,  + \, \bm \nu_2 \mres \cl{Q_x^1} \, \stackrel{\eqref{eq:4aR},\eqref{eq:4ar}}= \, \bm \nu_1 \mres {Q_x^1} \, + \, \bm \nu_2 \mres {Q_x^2} \quad \text{in $\Y_\Acal(Q_x)$.}
\end{equation}
Therefore, upon re-adjusting the sequence $\{w_j\}$ we may assume that 
\begin{equation}\label{eq:boundw}
	\sup_{j \in \Nbb} |w_j|(Q_x) \le 2\Big( \ddprB{|\frarg|, \bm \nu_1 \mres Q_x} + \ddprB{|\frarg|, \bm \nu_2 \mres Q_x} \Big).  
\end{equation}

\textit{{Step~4b. Cubes $Q_x \in \Ocal_m$ centered at $x \in \Scal$.}} 

The constructions in these cubes will be completely different and it will consist of separating the generating sequences of $\bm \nu_1, \bm \nu_2$ locally. 
Once again, by Lemma~\ref{lem:Helm}, we may find sequences of potentials $\{u_j\}, \{w_j \}\subset \Wrm^{k_\Bbb-1,q}(Q_x)$ such that 
\begin{gather*}
\text{$u_j,w_j \equiv u$ on a neighborhood of $\partial Q_x$},\\
u_j,w_j \to u \quad \text{in $\Wrm^{k_\Bbb -1,1}(Q_x)$},
\end{gather*}
where $\mu \mres Q_x = \Bcal u + v \, \Leb^d$ for some $v \in \Lrm^1(Q_x;W)$. Moreover,
\[
\Bcal u_j \, + \, v \Leb^d \, \toY \, \bm \nu_1 \quad \text{in $\Y(Q_x;W)$}, 
\]
and
\[
\Bcal w_j \, + \, v \Leb^d \, \toY \, \bm \nu_2 \quad \text{in $\Y(Q_x;W)$},
\]
Now, let $\phi$ be a cut-off function satisfying (here $Q_x = Q_r(x)$)
\[
	\mathbbm 1_{Q_{r/2}(x)}  \le \phi  \le \mathbbm1_{Q_{3r/4}(x)}, \qquad \|\phi\|_{k,\infty} \lesssim r^{-dk}.
\] 

Due to the $\Lrm^p$-continuity of the translations, we may choose $n_1 = n_1(m) \in \Nbb$ to be sufficiently large so that 
\begin{equation}\label{eq:trans}
\begin{split}
	  \|u_j - u\|_{\Wrm^{k-1,1}(Q_x)}    + \|w_j - u\|_{\Wrm^{k-1,1}(Q_x)}  
	\end{split}
	\le \frac {r^{dk}}{m} \cdot \frac {\Lambda(Q_x)}{m}
\end{equation}
for all $j \ge n_1$.

We are now in position to define our recovery sequence  candidate for $\bm \nu_\theta$ on $Q_x$ by setting
\[
	q_j = q_j^x \coloneqq \Bbb(\phi[\theta_1 \Delta_{-s_x} u_j + \theta_2 \Delta_{s_x} w_j - u]) + \Bcal u + v, \quad j \in \Nbb,
\]

The purpose of this sequence is to shift $u_j$ and $w_j$ apart from each other, while preserving the $\mu$-boundary conditions near $\partial Q_x$ (see Figure~\ref{fig:delta} below). Clearly, $\{q_j\}$ is a sequence of $\A$-free measures on $Q_x$ with $q_j \equiv \mu$ on a neighborhood of $\partial Q_x$ and $q_j \approx \theta_1 u_j + \theta_2 v_j$ on $Q_{r/2}(x)$. Notice that this construction differs from the previous one (when $x \in \Rcal$) in the sense that the \emph{$\theta_i$-weights} are incorporated by simple multiplication. In general, this construction is too naive to work. However, in this case, it  works because we have $\Lambda \approx \delta_0$ in $Q_x$.   
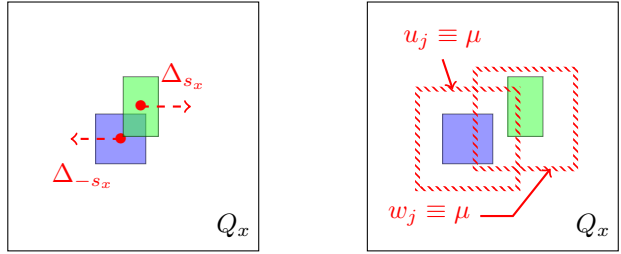
\begin{figure}[h]
	\begin{center}
		\begin{tikzpicture}[scale=0.33] 
		\draw (0,0) rectangle (10,10);
		\draw[fill=blue!80,opacity=0.5] (3.5,3.5) rectangle (5.5,5.5);
		\draw[->,color=red,dashed,thick] (4.5,4.5) -- (2.5,4.5);
		\draw (4.5, 4.5) node[color=red] {\textbullet};
		\draw (3,3) node[color=red] {$\Delta_{-s_x}$};
		\draw[fill=green!80,opacity=0.5] (4.6,4.6) rectangle (6,7);
		\draw[->,color=red,dashed,thick] (5.3,5.8) -- (7.3,5.8);
		\draw (5.3,5.8) node[color=red] {\textbullet};
		\draw (7,7) node[color=red] {$\Delta_{s_x}$};
		\draw (9,1) node {$Q_x$};
		\end{tikzpicture}
		\qquad \qquad 
		\begin{tikzpicture}[scale=0.33] 
		\draw (0,0) rectangle (10,10);
		\draw[fill=blue!80,opacity=0.5] (3,3.5) rectangle (5,5.5);
		\draw[fill=green!80,opacity=0.5] (5.6,4.6) rectangle (7,7);
		\fill [even odd rule,pattern=custom north west lines,hatchspread=3pt,hatchthickness=0.6pt,hatchcolor=red] (2.1,2.6) rectangle (5.9,6.4) (1.9,2.4) rectangle (6.1,6.6);
		\draw (9,1) node {$Q_x$};
		\fill [even odd rule,pattern=custom north west lines,hatchspread=3pt,hatchthickness=0.6pt,hatchcolor=red] (4.4,3.4) rectangle (8.2,7.2) (4.2,3.2) rectangle (8.4,7.4);
		
		\draw[<-,thick,red] (3.4,6.5) -- (3,7.7) node[above,align=center] {$u_j \equiv \mu$};
		\draw[<-,thick,red] (7.3,3.3) -- (5.8,1.4) -- (4.5,1.4)
		node[left,align=center]{$w_j \equiv \mu$};
		\end{tikzpicture}
	\end{center}
	\caption{Qualitative representation of the construction when $\delta_0 \in \Tan(\Lambda,x)$. The blue (green) area  represents  the region  where \emph{most} of the mass of $\lambda_1$ (of $\lambda_2$) is concentrated. }
	\label{fig:delta}
\end{figure}

 Let us fix $j \in \Nbb$. Writing $u = \theta_1 u + \theta_2 u$ and adding a zero, we may express $q_j$ as 
\begin{align*}
	q_j & = \theta_1 \, \phi \cdot \Delta_{-s_x} \Bbb (u_j - u) \\
		& \quad + \theta_2 \, \phi \cdot \Delta_{s_x}  \Bbb (w_j - u)  \\
		& \qquad + \theta_1 \, \Delta_{-s_x} [\Bbb,\phi](u_j - u)  \\
		& \quad\qquad +  \theta_2 \, \Delta_{s_x} [\Bbb,\phi] (w_j - u)  \\ 
		& \qquad \qquad  \theta_1\, [\Bbb,\phi](\Delta_{-s_x}u - u) + \theta_2\,[\Bbb,\phi](\Delta_{s_x}u - u) \\ 
		& \quad \qquad \qquad + \phi\cdot \Bbb\big(\theta_1 \Delta_{-s_x}  u + \theta_2 \Delta_{s_x}  u - u \big) + \mu,
\end{align*}
where as usual the commutator $[\Bbb,\chi] = \Bcal(\phi \frarg) - \phi\Bcal$ is a linear  operator of order at most $(k_\Bbb -1)$ and whose coefficients depend solely on $\|\phi\|_{k,\infty}$ and the principal symbol $\Bbb$.  
We obtain the following estimate for the total variation  of $q_j$:  
\begin{align*}
	|q_j|(Q_x)  & \lesssim_{\Bbb} |\Bcal u_j|(Q_x) + |\Bcal w_j|(Q_x) + |\mu|(Q_x)   \\
				& \quad + \|\phi\|_{k,\infty} \Big(\|u_j - u\|_{\Wrm^{k-1,1}(Q_x)} + \|w_j - u\|_{\Wrm^{k-1,1}(Q_x)}   \\ 
				& \quad +  \|\Delta_{-s_x} u - u\|_{\Wrm^{k-1,1}(Q_x)} + \|\Delta_{s_x} u - u\|_{\Wrm^{k-1,1}(Q_x)}\Big) \\
				& \stackrel{\eqref{eq:15},\eqref{eq:trans}}\lesssim |\Bcal u_j|(Q_x) + |\Bcal w_j|(Q_x) + |\mu|(Q_x)  + \frac {\Lambda(Q_x)}{m} \quad \forall \; j \ge n_1.
\end{align*}
In particular, upon re-adjusting the sequence $\{q_j\}$ we may assume that 
\begin{equation}\label{eq:boundq}
	\sup_{j \in \Nbb} |q_j|(Q_x) \lesssim_{\Bbb}  |\mu|(Q_x)  + \frac {\Lambda(Q_x)}{m},
\end{equation}
and   $q_j \toY\bm \sigma^x \in \Y_\Acal(Q_x)$. 

Observe that if $f = f_{p,q}$ with $p,q \le m$, then
\begin{align*}
	\Big|\theta_1 \cdot \ddprB{f,\bm \nu_1 \mres Q_x} &  - \theta_1 \cdot  \ddprB{f,\bm \nu_1 \mres Q_{s_x}(x)} \Big| \le M_f \ddprB{1 + |\frarg|,\bm \nu_1 \mres A_x} \stackrel{\eqref{eq:singular},\eqref{eq:14}}\le \frac 2m \Lambda(Q_x). \\
\end{align*}
Hence, there exists $n_1 \le n_2 = n_2(m) \in \Nbb$ such that (with $\gamma_j^1 = v \, \Leb^d + \Bcal u_j$) 
\begin{align*}
	\ddprB{f,\bm \delta_{\theta_1 \gamma_j^1} \mres Q_x}  & = \int_{Q_x} \dprb{f,\theta_1 \ac{(\gamma_j^1)}} \, \Leb^d + \int_{Q_x} \theta_1 \dprb{f^\infty,(\gamma_j^1)^s} \dd \lambda \\
											    & \stackrel{\eqref{eq:singular}}=   \int_{Q_{s_x}(x)} \theta_1 \dprb{f^\infty,(\gamma_j^1)^s} \dd \lambda + M_f \cdot  \BigO(m^{-1}) \cdot \Lambda(Q_x) \\
											    &  \stackrel{\eqref{eq:singular}}= \theta_1 \cdot \ddprB{f,\bm \delta_{\theta_1 \gamma_j^1} \mres Q_{s_x}(x)} + 2 M_f \cdot \BigO(m^{-1}) \cdot \Lambda(Q_x) 
\end{align*}
for all   $j \ge n_2$. An analogous estimate holds for $\bm \nu_2$, $w_j$, and $\theta_2$. Let us set $R_x \coloneqq Q_x \setminus \big((Q_{s_x}(x) - s_xe_1) \cup  (Q_{s_x}(x) + s_xe_1) \big)$. Then, by the definition of $q_j$, a similar argument combined with the right translations $\pm s_x e_1$ yield (for $j \ge n_2$)
\begin{gather*}
	 \ddprB{f,\bm \delta_{q_j} \mres R_x} =   \BigO(m^{-1}) [M_f \cdot  \Lambda(Q_x)], \\
	 \ddprB{f,\bm \delta_{q_j} \mres Q_x \setminus \cl{R_x}} = \theta_1 \cdot\ddprB{f,\bm \delta_{\theta_1 \gamma_j^1}} + \theta_2 \cdot\ddprB{f,\bm \delta_{\theta_1 \gamma_j^2}} +  \BigO(m^{-1}) [M_f \cdot  \Lambda(Q_x)].
\end{gather*}
%
Combining these estimates we obtain upon re-adjusting the sequence of $j$'s (recall that we had written $f = f_{p,q}$)
\begin{equation}\label{eq:S_1}
\begin{split}
	 \ddprB{f_{p,q},\bm \sigma^x \mres Q_x} & = \lim_{j \to\infty} \ddprB{f_{p,q},\bm \delta_{q_{j}} \mres Q_x} \\  
	 &  = \ddprB{f_{p,q},\bm \nu_\theta \mres Q_x} +  \BigO(m^{-1}) [M_f \cdot  \Lambda(Q_x)]
	 \end{split}
\end{equation}
whenever $p,q \le m$.\\ 

%

\textit{{Step~4c. Gluing the local recovery sequences.}}

Every cube  $Q_x \in \Ocal_m$ is centered at some $x \in L$ and since 
\[
	  \reg(\Omega) \cup \Rcal \cup \Scal = \reg(\Omega) \cup \sing(\Omega) = L,
\]
the constructions in Steps~4a and~4b indeed cover all possible scenarios which can present. The next task is to glue the recovery sequences together to obtain an $\A$-free global recovery sequence of the $\BigO(m^{-1})$-approximation of $\bm \nu_\theta$. For each $m \in \Nbb$, let us define the sequence  
 
\[
w^{(m)}_j(\mathrm dy)  \coloneqq \begin{cases}
w_j^x(\mathrm dy) & \text{if $x \in \Rcal \cup \reg(\Omega)$} \\
q_{j}^x(\mathrm dy) & \text{if $x \in \Scal$} \\
\mu(\mathrm dy) & \text{elsewhere}
\end{cases}, \qquad y \in \Omega.
\]
Notice that by construction each $w_j^{(m)}$ is $\A$-free since each $w_j^x$ and $q_j^x$ is $\A$-free on $Q_x$ and has $\mu$-boundary values in an open neighborhood of $\partial Q_x$. \\

%
%

\textit{{Step~4d. Generation of the $(m^{-1})$-approximations of $\bm \nu_\theta$.}} Appealing to the locality of weak-$*$ convergence of Young measures, we  show next that if $p,q \le m$, then (as $j \to \infty$)
\begin{equation}\label{eq:fmt}
\begin{split}
	\lim_{j \to \infty} \ddprB{f_{p,q},\bm \delta_{w_j^{(m)}}} \, & =  \, \ddprB{f_{p,q},\bm \nu^{(m)}_\theta} \\
	& \quad  + \BigO(m^{-1})  \cdot M_{f_{p,q}}   \cdot  \Lambda(\Omega), 
\end{split} 
\end{equation}
where $\bm \nu^{(m)}_\theta$ is the Young measure which acts on $f \in \E(\Omega,W)$ by the representation formula
\begin{align*}
\ddprB{f,&\bm \nu^{(m)}_\theta}  = \ddprB{f,\bm\delta_\mu}_{(O_m)^c} \\
& +  \sum_{i = 1,2} \Bigg(\sum_{\substack{Q_x \in \Ocal_m\\x \in \Rcal \cup \reg(\Omega)}} \ddprB{f,\bm \nu_i \mres Q_x^i}   + \sum_{\substack{Q_x \in \Ocal_m\\x \in \Scal  }} \theta_i \cdot \ddprB{f,\bm \nu_i \mres Q_x}\Bigg).
\end{align*}
Later, in the next step, we will show these Young measures are indeed $\BigO(m^{-1})$-approximations of $\bm \nu_\theta$. This, together with a diagonalization argument with~\eqref{eq:fmt} will imply  that $\bm \nu_{\theta} \in \Y_\Acal^{0,\mu}(\Omega)$. 

First, we show that the sequence $\{w^{(m)}_j\}_{j,m} \subset \M(\Omega,W)$ has uniformly bounded total variation on $\Omega$. There is no loss of generality in assuming that $f_{1,1} = \mathbbm 1_{\Omega} \otimes |\frarg|$, and therefore  
\begin{align*}
	|w^{(m)}_j|(\Omega) & \stackrel{\eqref{eq:boundw},\eqref{eq:boundq}}\le 
	\sum_{\substack{Q_x \in \Ocal_m\\x \in \Rcal \cup \reg(\Omega)}} 2\Big( \ddprB{|\frarg|, \bm \nu_1 \mres Q_x} + \ddprB{|\frarg|, \bm \nu_2 \mres Q_x} \Big) \\ 
	& 	\qquad + \sum_{\substack{Q_x \in \Ocal_m\\x \in \Scal}} \Big( |\mu|(Q_x)  + \frac {\Lambda(Q_x)}{m} \Big) \\ 
	& \le 3 \Big(\ddprB{\mathbbm 1_{\Omega} \otimes|\frarg|,\bm \nu_1} +\ddprB{\mathbbm 1_{\Omega} \otimes|\frarg|,\bm \nu_2} \Big) + \frac 1m \Lambda(\Omega) + |\mu|(\Omega).
\end{align*}
	This shows 
	\begin{equation}\label{eq:tvb}
		\sup_{m \in \Nbb} \bigg(\sup_{j \in \Nbb}|w^{(m)}_j|(\Omega) \bigg)  \le \sup_{m \in \Nbb} C(m) < \infty,
	\end{equation}
	as desired. 
	
	Since $(w_j^m)_j$ is uniformly bounded on $\Omega$, the desired limit approximation in~\eqref{eq:fmt} follows from 1) the locality of weak-$*$ convergence of Young measures, 2) the generation properties~\eqref{eq:4a1}-\eqref{eq:4a2} for points in $\reg(\Omega) \cup \Rcal$, 3) the generation property at singular points in $\Scal$ the ~\eqref{eq:S_1}, and 4) the fact that $\Ocal_m$ is a full $(\Leb^d + \Lambda)$-partition of $\Omega$.\\ 
%
%
%
%
%

\textit{{Step~5. The sequence $\bm \nu^{(m)}_\theta$ approximates $\bm \nu_\theta$.}} Next we show that
\[
	\lim_{m \to \infty} \ddprB{f_{p,q},\bm \nu_\theta^{(m)}} =  \ddprB{f_{p,q},\bm \nu_\theta} \quad \text{for all $p,q \in \Nbb$.}
\]
Accordingly, fix $p,q \in \Nbb$ and choose $m \ge p,q$. Let us, for the sake of simplicity, write $f = f_{p,q}$ and $f^\infty = f_{p,q}^\infty$. Due to the high amount of terms and estimates involving this argument, let us write  
\begin{align*}
	\ddprB{f,\bm \nu^{(m)}_\theta } & = I_1 + I_2 + II + III_1 + III_2 + IV + V_1 + V_2, 
\end{align*}
where each term contains partial sums subjected to a decomposition of the mesh $\Ocal_m$ in the following way: 
\begin{enumerate}[(a)]
	\item the cubes $Q_x$ around regular points $x \in \reg(\Omega)$. For $i \in \{1,2\}$, the corresponding partial sum is given by 
\begin{align*}
I_i & \coloneqq \sum_{\substack{Q_x \in \Ocal_m\\x \in \reg(\Omega)}} \bigg( \int_{Q_x^i} \Big( \dprb{f,\nu_i} + \dprb{f^\infty , \nu_i} \cdot \ac \lambda_i \Big) \dd \Leb^d  + \int_{Q_x^i} \dprb{f^\infty,\nu^\infty_i} \dd \lambda^s_i \bigg) \\
  & \stackrel{\eqref{eq:regular}-\eqref{eq:butt1}}= \sum_{\substack{x \in \reg(\Omega)\\Q_x \in \Ocal_m}} \Big( \dprb{f,\nu_i}_x  + \dprb{f^\infty,\nu^\infty_i}_x \ac \lambda_i(x) \Big) \cdot \Leb^d(Q_x^i)  
  \\ & \qquad + \sum_{\substack{x \in\reg(\Omega)\\Q_x \in \Ocal_m}} \BigO(m^{-1}) \cdot (\mathrm{Lip}(f)  +1) \cdot \Leb^d(Q_x^i) \\
   &\stackrel{\eqref{eq:4ar}}= \theta_i \cdot \sum_{\substack{x \in \reg(\Omega)\\Q_x \in \Ocal_m}}
   \int_{Q_x} \Big( \dprb{f,\nu_i}_x + \dprb{f^\infty , \nu_1}_x \ac \lambda_i(x) \Big) \dd \Leb^d \\
  & \qquad +  \BigO(m^{-1}) \cdot (\mathrm{Lip}(f)  +1)  \cdot  \Leb^d(\Omega) \\
     &\stackrel{\eqref{eq:but1}-\eqref{eq:butt1}}= \theta_i   \sum_{\substack{Q_x \in \Ocal_m\\x \in \reg(\Omega)}}
    \ddprB{f,\bm \nu_i \mres Q_x} \\
    & \qquad   +  \BigO(m^{-1}) \cdot (\mathrm{Lip}(f)  +2)  \cdot    \Leb^d(\Omega):
\end{align*}
\item and now we cover the singular set $\sing(\Omega)$, starting with the cubes around singular points $x \in \Rcal \cap G_0$ 
\begin{align*}
II & \coloneqq \sum_{\substack{Q_x \in \Ocal_m\\ x \in \Rcal \cap G_0}} \int_{Q_x^1} \dprb{f,\nu_1}  \dd \Leb^d + \int_{Q_x^1} \dprb{f^\infty , \nu^\infty_1} \dd  \lambda_1     \\
& \quad + \sum_{\substack{Q_x \in \Ocal_m\\ x \in \Rcal \cap G_0}} \int_{Q_x^1} \dprb{f,\nu_2}  \dd \Leb^d +  \int_{Q_x^2} \dprb{f^\infty , \nu^\infty_2} \dd  \lambda_2  \\
& \stackrel{\eqref{eq:singular}-\eqref{eq:11}}= \sum_{\substack{Q_x \in \Ocal_m\\ x \in \Rcal \cap G_0 \cap B_1^\infty}} \bigg(  \dprb{f^\infty , \nu^\infty_1}_x \cdot  \lambda_1^s(Q_x^1) + M_f \cdot \BigO(m^{-1}) \cdot \Lambda(Q_x^2) \bigg) \\
 & \qquad + \sum_{\substack{Q_x \in \Ocal_m\\ x \in \Rcal \cap G_0 \cap B_2^\infty}} \bigg(  \dprb{f^\infty , \nu^\infty_2}_x \cdot  \lambda_2^s(Q_x^2) + M_f \cdot \BigO(m^{-1}) \cdot \Lambda(Q_x^1) \bigg) \\  
& \stackrel{\eqref{eq:13},\eqref{eq:4aR3}} = \theta_2 \cdot \sum_{\substack{Q_x \in \Ocal_m\\ x \in \Rcal \cap G_0 \cap B_2^\infty}}
\int_{Q_x}   \dprb{f^\infty , \nu^\infty_2}_x \dd \lambda_2^s
\\ & \qquad    +  \BigO(m^{-1}) \cdot   (M_f + \mathrm{Lip}(f)) \cdot \Lambda(\Omega).
\end{align*}
In the first equality we have strongly used the $\Lambda$-Lebesgue property for the sets $Q_x^i$ which is justified in Step~1a(1); the precise statement is contained in Corollary~\ref{cor:theta}. The same argument will be used in the estimates (c) and (d) below;
\item passing to points   $x \in \Rcal \cap G_1$ (in this case $x \in B_1^\infty \cap B_2^\infty$). For $i = 1,2$ the partial sum reads
\begin{align*}
III_i & \coloneqq \sum_{\substack{Q_x \in \Ocal_m\\x \in \Rcal \cap G_1}}  \int_{Q_x^i} \dprb{f,\nu_i}  \dd \Leb^d + \int_{Q_x^i} \dprb{f^\infty,\nu^\infty_i} \dd  \lambda_i^s \\
& \stackrel{\eqref{eq:singular}-\eqref{eq:but2}}= \sum_{\substack{Q_x \in \Ocal_m\\x \in \Rcal \cap G_1 \cap B_i^\infty}} \bigg( \dprb{f^\infty,\nu^\infty_i}_x \cdot  \lambda_i^s(Q_x^i)  
 + M_f \cdot \BigO(m^{-1}) \cdot \Lambda(Q_x) \bigg) \\
& \qquad + \sum_{\substack{Q_x \in \Ocal_m\\x \in \Rcal \cap G_1 \cap B_i^\infty}} \BigO(m^{-1}) \cdot \lambda_i(Q_x^i) \\
& \stackrel{\eqref{eq:4aR2}}= \theta_i \cdot \sum_{\substack{Q_x \in \Ocal_m\\x \in \Rcal \cap G_1 \cap B_i^\infty}}
\int_{Q_x}  \dprb{f^\infty,\nu^\infty_i}_x \dd \lambda_i^s\\
& \qquad +  \BigO(m^{-1})  \cdot  ( M_f + \mathrm{Lip}(f) + 1)  \Lambda(Q_x);
\end{align*}
\item and to finally cover $\Rcal$, the singular points   $x \in \Rcal \cap G_\infty$: an analogous estimate to the one derived in (b) gives
\begin{align*}
IV & \coloneqq \sum_{\substack{Q_x \in \Ocal_m\\ x \in \Rcal \cap G_\infty}} \int_{Q_x^1} \dprb{f,\nu_1}  \dd \Leb^d + \int_{Q_x^1} \dprb{f^\infty , \nu^\infty_1} \dd  \lambda_1     \\
& \quad + \sum_{\substack{Q_x \in \Ocal_m\\ x \in \Rcal \cap G_\infty}} \int_{Q_x^1} \dprb{f,\nu_2}  \dd \Leb^d +  \int_{Q_x^2}  \dprb{f^\infty , \nu^\infty_2} \dd  \lambda_2  \\
& \stackrel{\eqref{eq:4aR4}} = \theta_1 \cdot \sum_{\substack{Q_x \in \Ocal_m\\ x \in \Rcal \cap G_\infty \cap B_2^\infty}}
\int_{Q_x}   \dprb{f^\infty , \nu^\infty_1}_x \dd \lambda_1^s
\\ & \qquad    +  \BigO(m^{-1}) \cdot   (M_f + \mathrm{Lip}(f)) \cdot \Lambda(\Omega).
\end{align*}
\item Lastly, the cubes with centers $x \in \Scal$ which by definition are simply given by
\begin{align*}
	V_i & \coloneqq \sum_{\substack{Q_x \in \Ocal_m\\x \in \Scal}}  \ddprB{f,\bm \nu_\theta \mres Q_x}.  \\
\end{align*}
\end{enumerate}

Since the singular set can be split into the disjoint union  
\begin{align*}
	\sing(\Omega) & = \Rcal \cup \Scal \\
	& = ( \Rcal \cap G_0 \cap B_2^\infty) \cup  ( \Rcal \cap G_1 \cap B_1^\infty \cap B_2^\infty) \cup( \Rcal \cap G_\infty \cap B_1^\infty) \cup \Scal,
\end{align*}
and since every possible cube $Q_x \in \Ocal_m$ is centered at one (and only one) of the previous four sets, we deduce from inspecting the terms $I_i,II,III_i,IV, V_i$  that 
\begin{align*}
	\ddprB{f,\bm \nu^{(m)}_\theta} & = \sum_{i=1,2}  \theta_i \bigg(  \sum_{\substack{Q_x \in \Ocal_m\\x \in \reg(\Omega)}} 
	\Big(\int_{Q_x} \dprb{f,\nu_i}_x \dd \,\Leb^d(y)  + \int_{Q_x} \dprb{f^\infty,\nu^\infty_i}_x \, \ac\lambda_i(x)  \dd \,\Leb^d(y)  \\ 
	& \qquad +  \sum_{\substack{Q_x \in \Ocal_m\\x \in ((\sing(\Omega) \cap B_i^\infty) \setminus \Scal)}} \int_{Q_x} \dprb{f^\infty,\nu^\infty_i}_x  \dd \lambda_i^s(y)  \bigg) + \sum_{\substack{Q_x \in \Ocal_m\\x \in \Scal}}
	\ddprB{f,\bm \nu_\theta \mres Q_x}   \\
	&  \qquad + \BigO(m^{-1}) \cdot \Big(M_f + \mathrm{Lip}(f) + 2\Big) \cdot \Big(\Lambda(\Omega) + \Leb^d(\Omega)\Big) \\
	& = \theta_1 \cdot \ddprB{f,\bm \nu_1^{(m)}} + \theta_2 \cdot \ddprB{f,\bm \nu_2^{(m)}} \\
	& \qquad + \BigO(m^{-1}) \cdot \Big(M_f + \mathrm{Lip}(f) + 2\Big) \cdot \Big(\Lambda(\Omega) + \Leb^d(\Omega)\Big).
\end{align*}

\textit{{Conclusion.}}
Let us recall  that 
\[
\sup_{m \in \Nbb} \bigg(\sup_{j \in \Nbb}|w^{(m)}_j|(\Omega) \bigg) \le \sup_{m \in \Nbb} C(m) < \infty, 
\]
where $C(m)$ is the constant from~\eqref{eq:tvb}. Returning to the estimate~\eqref{eq:fmt}, we may then, by a diagonalization argument, define a sequence of $\A$-free measures 
\[
w^{(m)} \coloneqq w_{j(m)}^{(m)} \in \M(\Omega;W),
\] 
satisfying (cf. Claim~2 and Step~5)
\begin{align*}
\lim_{m \to \infty} \ddprB{f,\bm \delta_{w^{(m)}}}  & =  \lim_{m \to \infty}	\ddprB{f,\bm \nu_\theta^{(m)}}\\
& =  \theta_1 	\lim_{m \to \infty}	 \ddprB{f,\bm \nu_1^{(m)}} + \theta_2 	\lim_{m \to \infty}	 \ddprB{f,\bm \nu_2^{(m)}} \\
& =  \theta_1 	\ddprB{f,\bm \nu_1} + \theta_2 	 \ddprB{f,\bm \nu_2} \\ 
& =  \ddprB{f,\bm \nu_\theta} \quad \text{for all $f \in \{f_{p,q}\}_{p,q \in \Nbb}$.}
\end{align*}
Moreover, it follows from the compactness of Young measures and the separation Lemma~\ref{lem:separation} that the convergence above implies  (this may involve passing to a subsequence)
\[
w^{(m)} \, \toY \, \bm \nu_\theta \quad \text{on $\Omega$}.
\]
%
%
This finishes the proof.
 \end{proof}

\subsection{Characterization of singular Young measures} In this section we establish a criterion for a family $\Y^\sing(Q) \subset \Y_0(Q;W)$ to belong to $\Y_\Acal(Q)$. This family mimics the properties of singular tangent Young measures, and therefore this criterion will serve as a preparation for the proof of Theorem~\ref{thm:char}. 
\begin{remark}It is worthwhile to mention that our construction departs from the approach presented  in~\cite{kristensen2010characterizatio,de-philippis2016on-the-structur}. There, the authors are able to work with a family $\Y^\sing(P_0)$ of Young measures that are \emph{one-directional}.  This, however, relies on the rigidity that  gradients and symmetric gradients possess. In turn, this simplifies enormously the proof of the convexity of $\Acal$-free young measures at the level of tangent Young measures; it also allows for the creation of \emph{artificial concentrations} (cf.~\cite[Lemma 3.5]{de-philippis2017characterizatio}), which is crucial for the separation argument. It is precisely for this reason that the convexity of $\Y_{\Acal,0}(\mu,\Omega)$ had to be conceived globally rather than at the level of tangent $\Acal$-free Young measures. 
\end{remark}

Let us turn to the heart of the matter. Let $P_0 \in W$ and let 
	\[
		\Y^\sing(P_0)  \coloneqq \set{(\delta_0,\lambda,\sigma^\infty) \in \Y_0(Q;W)}{\sigma_y^\infty = \sigma_0^\infty \, \text{$\lambda$-a.e.}, \dprb{\id_W,\sigma_0} = P_0}. 
	\]
We shall prove the following:
\begin{proposition}\label{prop:singular}
	Let $\bm \nu = (\delta_0,\lambda,\nu_0^\infty) \in \Y^\sing(P_0)$ for some $P_0 \in \Lambda_\Acal$. Assume that
	\[
		\Acal[\bm \nu] = 0 \quad \text{in the sense of distributions on $Q$}
	\]
	and further assume that
	\[
		\supp(\nu_0^\infty) \subset W_\Acal,
	\]
	then $\bm \nu \in \Y_\Acal(Q)$.
\end{proposition}
\begin{proof} Let us write $\mu \coloneqq [\bm \nu] = P_0\lambda$. The idea is to find a suitable convex subset $\Y_\Acal^\sing(\mu) \subset \Y_\Acal(Q) \subset \E(Q;W)^*$ that contains $\bm \nu$. The choice of $\Y_\Acal^\sing(\mu)$ is of course not unique, and is part of the problem in turn. 
 We shall work with the following set:
\begin{definition} Let $\Y^\sing_\Acal(\mu)$ be the set of $\Acal$-free Young measures $\bm \sigma \in \Y_0(Q;W)$ satisfying the following properties:
	\begin{enumerate}[(a)]
		\item $[\bm \sigma] = \mu$,
		\item $\supp(\sigma_y) \subset W_\Acal$ for $\Leb^d$-almost every $y \in Q$, 
		\item $\supp(\sigma_y^\infty) \subset W_\Acal$ for $\lambda_{\bm \sigma}$-almost every $y \in Q$.
	\end{enumerate} 
Notice that $\Y_\Acal^\sing(\mu)$ is non-empty since it contains $\bm \delta_{\mu}$.
\end{definition}

%

%
%
\begin{remark}
Since properties (a)-(c) are convex properties, Proposition~\ref{prop:closed} and Theorem~\ref{thm:convex} imply that $\Y_\Acal^\sing(\mu)$ is a non-empty weak-$*$ closed convex subset of $\E(Q;W)^*$. 
\end{remark}

\proofstep{Step~1. The separation property.} Let us recall that, for every weak-$*$ closed affine half-space $H \subset \E(Q;W)^*$, there exists an integrand $f_H \in \E(Q;W)$ such that 
\[
H = \set{\ell \in \E(Q;W_\Acal)^*}{\ell(f_H) \ge s_H > - \infty}.
\]
From the geometric version of Hahn--Banach's it follows that, either $\bm \nu \in \Y_\Acal^\sing(\mu)$, or there exists $f_H \in \E(Q;W_\Acal)$ such that
\begin{equation}\label{eq:separation}
	\ddprB{f_H,\bm \nu} < s_H \le \ddprB{f_H,\bm \sigma} \qquad \text{for all $\bm \sigma \in \Y_\Acal^\sing(\mu)$.}
\end{equation}
Since the former case is precisely what we want to prove, let us henceforth assume that the $f_H$ strictly separates $\Y_\Acal^\sing(\mu)$ and $\bm \nu$. By assuming this, we shall reach a contradiction.
\proofstep{Step~2. Separation with $\tilde f_H$.} Let $\tilde f_H \in \E(Q;W)$ be the integrand defined as $\tilde f(x,z) = f(x,\mathbf p z)$, where $\mathbf p: W \to W$ is the canonical linear projection onto $W_\Acal$. Notice that properties (b)-(c) and the properties of $\Y^\sing(Q)$ we get
\begin{equation}\label{eq:integrandos}
		\ddprB{\tilde f_H,\bm{\tilde \sigma}} = \ddprB{f_H,\bm{\tilde \sigma}} \quad \text{for all $\bm{\tilde \sigma} \in \Y_\Acal^\sing(\mu) \cup \{\bm \nu\}$.}
\end{equation}

\proofstep{Step~2. Boundedness of $\Qcal_{\Acal}\tilde f_H$.} The following version of~\cite[Lemma~5.5]{baia2013} states that the separation property with $\Y_\Acal^\sing(\mu)$ conveys the finiteness of the $\Acal$-quasiconvex envelope of $\tilde f_H$: 

\begin{lemma}\label{lem:below}
	Let $f \in \E(Q; W)$ be such that $\ddprb{\tilde f,\bm \sigma} \ge s > -\infty$ for all $\bm \sigma \in \Y_\Acal^\sing(\mu)$. Then, there exists a dense set $D \subset Q$ such that, for every $y \in D$, it holds $\Qcal_\Acal \tilde f(y,0) > -\infty$.
\end{lemma}
The proof follows by verbatim from the proof of~\cite[Lemma 5.5]{baia2013}. The only difference is that, there, it is assumed that $W = W_\Acal$. However, this is of little importance in our setting due to properties (b) and (c), the properties of $\tilde f$ (cf. Section~\ref{sec:qc}) and property~\eqref{eq:span}.

\proofstep{Step~3. The contradiction: $\bm \nu \in H$.} 
For an integrand $f$ and $\eps > 0$, we define $f^\eps \coloneqq f + \eps |\frarg|$. 
By construction, property (A) from Remark~\ref{rem:mio} is trivially satisfied for the integrand $f = (\tilde f_H)^\eps$ and the domain $\Omega = Q$.  The finiteness of $\Qcal_\Acal \tilde f$ of $D$ and Propositions~\ref{prop:properties}, \ref{prop:dense} and~\ref{prop:trans} imply that $(\tilde f_H)^\eps$ also satisfies property (B) from Remark~\ref{rem:mio}. Lastly, we recall that $|\mu|(\partial Q) = P_0|\lambda|(\partial Q) = 0$, which implies that (C) in Remark~\ref{rem:mio} is also satisfied. With these considerations in mind, we may apply Theorem~\ref{thm:mio} to find a recovery sequence $\{u_j\} \subset  \Lrm^1(Q;W)$ satisfying 
\[
u_j \, \Leb^d \toweakstar \mu \; \text{in $\M(Q;W)$}, \quad \Acal u_j   \to 0 \; \text{in $\Wrm^{-k,q}(Q)$},
\] 
and attaining the so-called upper-bound property  
\begin{equation}\label{eq:mame2}
	\begin{split}
\limsup_{j \to\infty} \int_Q   (\tilde f_H)^\eps(u_j) \dd y &  \le \int_Q Q_{\A}  (\tilde f_H)^\eps(y,\ac \mu(y)) \dd y \\
& \qquad + \int_Q (Q_{\A}  (\tilde f_H)^\eps)^\#(y,g_\mu(y)) \dd |\mu^s|(y).
\end{split}
\end{equation}
Passing to a further subsequence if necessary, we may assume  that
\[
u_j \, \Leb^d \toY \bm \theta \qquad \text{for some $\bm \theta \in \Y_{\Acal}(Q) \cap \Y_0(Q;W)$.}
\] 

\proofstep{Claim. $\bm \sigma \in \Y_\Acal^\sing(\mu)$}.

 By construction, $[\bm \theta] = \wslim u_j =\mu$,  and hence property (a) is satisfied. The theory developed in Section~\ref{sec:dec} implies that, on strictly contained sub-cubes $Q_t \subset Q$, we may assume it holds
 \[
 	P_0\lambda \mres Q_t = z_0 \mres Q_t + v \Leb^d \mres Q_t, \quad 
 u_j \mres Q_t = z_j \mres Q_t + v\, \Leb^d \mres Q_t,
 \]
 where the elements of the sequence $\{z_j\}_{\Nbb_0}$ are mean-value zero $\Acal$-free  measures in $\Mcal(\Tbb^d;W)$ and $v \in \Lrm^q(\Tbb^d;W)$. Property~\eqref{eq:span} and a standard mollification argument gives $\{z_j\}_{\Nbb_0} \in \Mcal(\Tbb^d;W_\Acal)$. In particular, the expression for $(P_0 \lambda)\mres Q_t$ and the assumption $P_0 \in \Lambda_\Acal$ gives $v \Leb^d \mres Q_t \in \Lrm^q(Q_t;W_\Acal)$, whence
 \[
 	u_j \mres Q_t \in \Mcal(Q_t;W_\Acal) \quad \forall j \in \Nbb.
 \] 
 Since $Q_t \subset Q$ was arbitrarily chosen, this proves that $\bm \theta \in \Y_0(Q;W_\Acal)$ and therefore properties (b)-(c) hold. This proves that $\bm \theta \in \Y_\Acal^\sing(\mu)$ and the claim is proved.



For the ease of notation, let us write $f \coloneqq (\tilde f_H)^\eps$ for the next calculation. We use the main assumptions $P_0 \in \Lambda_\Acal$ and $\supp (\nu_{0}^\infty) \subset W_\Acal$, together with Proposition~\ref{prop:properties} (recall that $\Qcal_\Acal f(x,\frarg)$ is $(\Lambda_\Acal \cup W_\Acal^\perp)$-convex) and Remark~\ref{rem:kk} to find that
\begin{equation}\label{eq:mame1}
\begin{split}
\ddprB{f,\bm \nu}  
& \ge \int_{Q} \Big(\dprb{\Qcal_{\Acal}   f(y,\frarg),\delta_0} + \dprb{\Qcal_{\Acal} f(y,\frarg)^\#,\nu_{0}^\infty} \ac\lambda(y)\Big) \dd y \\ 
& \qquad +  \int_Q \dprb{\Qcal_\Acal  f(y,\frarg)^\#,\nu_{0}^\infty} \dd \lambda^s(y)\\
& \ge \int \Qcal_\Acal f(y,\ac\mu(u))) \dd y  + \int_Q (\Qcal_\Acal f)^\#(y,P_0) \dd\lambda^s(y)\\ 
& \ge  \int_{Q} \Qcal_{\Acal}  f(y,\ac\mu(y)) \dd y + \int_{Q} (\Qcal_{\Acal}  f)^\#(y,g_\mu)\dd |\mu^s|(y).
\end{split}
\end{equation}
Here, in the one but last inequality we have dealt with the absolutely continuous part with the aid of the following property: if $\Dcal \subset W$ is a spanning cone of directions, then every $\Dcal$-convex function $g : W \to \R$ with linear growth at infinity satisfies (see for instance Lemma~2.5 in~\cite{kirchheim2016on-rank-one-con})
\[
 	\text{$g(z + P) \le g(z) + g^\infty(P)$ \quad for all $z \in W$ and $P \in \Dcal$.}
\] 
 Hence, by combining~\eqref{eq:mame2}-\eqref{eq:mame1}   we conclude that
\begin{align*} 
\ddprB{(\tilde f_H)^\eps,\bm \nu}  & \ge \lim_{j \to\infty} \int_Q   (\tilde f_H)^\eps(u_j) \dd y = \ddprB{(\tilde f_H)^\eps,\bm \theta} \ge \ddprB{\tilde f_H,\bm \theta} \ge s_H.
\end{align*}
Letting $\eps \to 0^+$, we conclude from~\eqref{eq:integrandos} that $\ddprb{f_H,\bm \nu} = \ddprb{\tilde f_H,\bm \nu} \ge s_H$, which contradicts~\eqref{eq:separation}. This proves that $\bm \nu$ must indeed belong to $\Y_\Acal^\sing(\mu)$. This completes the proof of the proposition.\end{proof}

\subsection{Characterization of regular Young measures} Now, we prove the analog of Proposition~\ref{prop:singular} for regular tangent Young measures. The proof follows the ideas of the original proof in~\cite{fonseca1999mathcal-a-quasi}, but it requires some minor changes to deal with the fact that $W_\Acal$ and $W$ may not coincide.

Let $P_0 \in W$ and consider the family of \emph{homogeneous Young measures} defined as
	\[
		\Y^\reg(P_0) \coloneqq \set{(\sigma_0,\alpha\Leb^d,\sigma_0^\infty)}{\alpha \ge 0, P_0 = \dpr{\id,\sigma_0} + \alpha \dpr{\id_W,\sigma_0^\infty}}.
	\]
Here, the sub-index \enquote{$0$} in $\sigma_0$ and $\sigma_0^\infty$ denotes that $(\sigma_0)_y = \sigma_0$ and $(\sigma_0^\infty)_y = \sigma_0$ for all $y \in Q$.
This family mimics the properties of regular tangent measures (cf. Proposition~\ref{prop: localization regular}). As one would expect, for regular points it is property (ii) from Theorem~\ref{thm:char} that plays a fundamental role in the characterization of regular blow-ups:
\begin{proposition}\label{prop:regularr}
	Let $\bm \nu = (\nu_0,\alpha,\nu_0^\infty)\in \Y^\reg(P_0)$ and assume that 
\begin{equation}\label{eq:please}
	h(P_0) \le \dprb{h,\eta_0} + \alpha \dprb{h^\#,\eta_0^\infty} 
\end{equation}
for all upper-semicontinuous $\Acal$-quasiconvex integrands $h:W \to \R$ with linear growth at infinity. Then
\[
	\bm \nu \in \Y_\Acal(Q).
\]
\end{proposition}
\begin{proof}The proof is considerably simpler than the one at singular points since, for the separation argument, it will suffice to consider the following family of homogeneous Young measures:
	\[
		\Y^\reg_\Acal(P_0) \coloneqq \Y^\reg(P_0) \cap \Y_\Acal(Q).
	\]
The first step is to verify that $\Y^\reg_\Acal(P_0)$ is non-empty. Indeed, $\Acal(P_0 \Leb^d) = 0$ and therefore the elementary (purely oscillatory) homogeneous Young measure $(\delta_{P_0},0,\frarg)$ belongs to $\Y^\reg_\Acal(P_0)$. 
	\begin{remark}
		The properties that define $\Y^\reg_\Acal(P_0)$ are convex properties. Therefore, Proposition~\ref{prop:closed} and Theorem~\ref{thm:convex} imply that $\Y^\reg_\Acal(P_0)$ is a non-empty weak-$*$ closed convex subset of $\E(Q;W)^*$. 
	\end{remark}

\proofstep{Step 1. The separation property.}  By the geometric version of Hahn-Banach's theorem it holds that, either $\bm \nu \in \Y^\reg_\Acal(P_0)$, or  there exists an affine weak-$*$ closed half-plane $H$ that strictly separates $\bm \nu$ from $\Y^\reg_\Acal(P_0)$, i.e., there exists $f_H \in \E(Q;W)$ such that 
\begin{equation}\label{eq:separation3}
	\ddprB{f_H,\bm \nu} < s_H \le \ddprB{f_H,\bm \sigma} \qquad \text{for all $\bm \sigma \in \Y^\reg_\Acal(P_0)$.}
\end{equation}
We shall henceforth assume that $f_H$ strictly separates $\bm \nu$ and $\Y^\sing_\Acal(P_0)$.

\proofstep{Step 2. Separation with elements of $\E^\reg(W)$.} Consider the class
\[
\E^\reg(W) \coloneqq \setB{\mathbbm 1_Q \otimes h}{ h \in \E(W)} \subset \E(Q;W).
\]
For an integrand $f \in \E(Q;W)$, let us consider the homogenized integrand $f_{\hom} \in \E^\reg(W)$ defined as
\[
	f_{\hom}(y,z) \coloneqq \int_Q f(\xi,z) \dd \xi, \qquad \forall (y,z) \in Q\times W.
\]
It is not hard to see that that $(\frarg)^\infty$ commutes with $(\frarg)_{\hom}$ for such integrands. Indeed, for every $(\xi,z) \in Q \times W$,
\[
	f_{t,z}(\xi)\coloneqq \frac{f(\xi,tz)}{t} \to f^\infty(\xi,z) \quad \text{uniformly on $Q$, as $t \to \infty$}. 
\]
In particular, 
\[
	(f_{\hom})^\infty(z) = \lim_{t \to \infty}  \int_Q f_{t,z}(\xi) \dd \xi  = \int_Q f^\infty (\xi,z) \dd x = (f^\infty)_{\hom}(z).
\]
Fubini's theorem gives
\begin{equation*}
	\begin{split}
	\ddprB{f,\bm{\tilde\sigma}} & = \dpr{f,\tilde \sigma_0} + \dprb{(f^\infty)_{\hom},\tilde\sigma_0^\infty}  \\
	& = \dpr{f,\tilde \sigma_0} + \dprb{(f_{\hom})^{\infty},\tilde\sigma_0^\infty}  = \ddprB{f_{\hom},\bm{\tilde\sigma}}
\end{split}
\qquad \text{for all $\bm{\tilde \sigma} \in \Y^\reg_\Acal(P_0) \cup \{\bm\nu\}$.}
\end{equation*}
This tells us there is no loss of generality in assuming that $f_H \in \E^\reg(W)$.  
	
%
%
%
\proofstep{Step~3. Reduction to the case $P_0 = 0$ and separation with $\tilde f_H$.} At regular points, we may no longer use the property $P_0  \in \Lambda_{\Acal}$. This prevents us to argue in the same way as for singular points, that it indeed suffices to work with $\tilde f_H$. There is turnaround to this issue by using the shifts introduced in Definition~\ref{def:shift}: 
	\begin{remark}
	By Corollary~\ref{cor:shift}, the shifted Young measure $\Gamma_{P_0}[\bm \sigma]$ is an homogeneous $\A$-free Young measure with  barycenter zero if and only if $\bm \sigma$ is an homogeneous $\A$-free Young measure  with barycenter $P_0$.
	\end{remark} This remark implies that
	\[
		\mathrm{Shifts}_{P_0}\{\Y^\reg_\Acal(P_0)\} = \Y^\reg_\Acal(0).
	\]
	A Young measure $(\sigma_0,\beta\Leb^d,\sigma_0^\infty)$ in the latter set satisfies the crucial property that $\supp(\sigma_0)$ (and $\supp( \sigma_0^\infty)$) are contained in $W_\Acal$ for $\Leb^d$-a.e. ($\beta\Leb^d$-a.e.) $y \in Q$; this follows from the same argument used to prove the claim that $\bm \theta \in \Y(Q;W_\Acal)$ in the singular points' case (here one uses that $0 \in \Lambda_\Acal$). On the other hand, Corollary~\ref{cor:last} implies that $\Gamma_{P_0}[\bm\nu] \subset W_\Acal$ and also that, either $\alpha = 0$, or $\supp(\nu_0^\infty) \subset W_\Acal$. In particular, we get
	\[
	\ddprB{f_H,\bm{\tilde\sigma}} = \ddprB{g_H,\Gamma_{P_0}[\bm{\tilde\sigma}]} = \ddprB{\tilde g_H,\Gamma_{P_0}[\bm{\tilde\sigma}]} \quad \text{for all $\bm{\tilde \sigma} \in \Y^\reg_\Acal(P_0) \cup \{\bm \nu\}$}.
	\]	
	Therefore, there is no loss of generality in assuming that $P_0 = 0$ and $f_H = \tilde f_H$.
	

\proofstep{Step~3. The contradiction: $\bm \nu \in H$.} The argument is essentially the same as the one for singular points, which relies on the upper bound from the relaxation argument. The relaxation argument, however, is nowhere near as involved as the one contained in Theorem~\ref{thm:mio}. First, we show that $\Qcal_\Acal \tilde f_H$ is finite with lower bound $s_H$. Let $w \in \Crm^\infty_\sharp(\Tbb^d;W_\Acal) \cap \ker \Acal$. 
By a well-known homogenization argument (see for instance~\cite[Proposition~2.8]{fonseca1999mathcal-a-quasi}) the sequence $w_j(y) \coloneqq w(jy)$ generates the homogeneous purely oscillatory Young measure $\bm \theta = (\cl{\delta_w},0,\frarg)$, where
\[
	\dprb{\cl{\delta_w},h} \coloneqq \int_Q h(w(y)) \dd y \quad \text{for all $h \in C(W)$, \quad $w_j \toweakstar 0$ in $\Lrm^\infty$.}   
\] 
In particular $\bm \theta \in \Y^\reg_\Acal(0)$, and hence
\[
	\int_Q \tilde f_H(w(y)) \dd y = \ddprB{\tilde f_H,\bm \theta} \ge s_H > - \infty.
\]
Taking the infimum over all $w$, we conclude that $\Qcal_\Acal \tilde f_H (0) \ge s_H$, as desired.
%
%
%
Now, we are in position to use~\eqref{eq:please}  with $h = \Qcal_\Acal \tilde f_H$ to get (recall that $\Qcal_\Acal \tilde f_H$ is globally Lipschitz, see Propositions~\ref{prop:properties} and~\ref{prop:dense}):
\begin{equation}\label{eq:mame}
	\begin{split}
		\ddprB{\tilde f_H,\bm \nu}  
		& \ge \int_{Q} \Big(\dprb{\Qcal_{\Acal}   \tilde f_H,\nu_{0}} +  \alpha\dprb{(Q_{\A} \tilde f_H)^\#,\nu_{0}^\infty}\Big) \dd y \\
		& \ge \int_Q \Qcal_\Acal \tilde f_H(0)\dd y \ge s_H.
	\end{split}
\end{equation} 
This poses a contradiction to inequality~\eqref{eq:separation3}.
This proves that $\bm\nu \in \Y^\reg_\Acal(0)$, as desired. \end{proof}
%

\subsection{Proof of Theorem~\ref{thm:char}}

\emph{Necessity.} Property (i) is obvious since, by definition, an $\A$-free measure is generated by a uniformly bounded sequence of (asymptotically) $\A$-free measures, therefore its barycenter is an $\Acal$-free measure. Properties (ii)-(iii) have been established in~\cite{arroyo-rabasa2017lower-semiconti}: condition (ii) is contained in Proposition~3.1, and condition (iii) in Lemma~3.2; condition (iii') is contained in Proposition~3.3.

\emph{Sufficiency.} 
The theory discussed in Section~\ref{sec:tan} implies that at $(\Leb^d + \lambda^s)$-almost every $x_0 \in \Omega$ there exists a regular or a singular tangent measure $\bm \nu_0 \in \Tan(\bm \nu,x_0)$. The idea is to show that each of these tangent measures are in fact locally $\Acal$-free Young measures and argue by the local characterization from Theorem~\ref{thm:local}. By the dilation properties of tangent Young measures, it suffices to show the following:
\begin{proposition}
	Let $\bm \nu \in \Y(\Omega)$ satisfy properties (i)-(iii) in Theorem~\ref{thm:char}. Then, 
	\[
	\set{\bm \sigma \mres Q}{\bm \sigma \in \Tan(\bm \nu,x)}  \cap \Y_{\Acal}(Q)
	\]
	for $(\Leb^d + \lambda^s)$-almost every $x \in \Omega$.
\end{proposition}
\begin{proof}[Proof of the proposition] As in the statement of Theorem~\ref{thm:char}, let us write $[\bm \nu] = \mu$.
	According to Proposition~\ref{prop: localization} and~\cite[Theorem 1.1]{de-philippis2016on-the-structur} we may find a full $\lambda^s$-measure set $S \subset \Omega$ such that, if $x_0 \in S$, then there exists $\bm \nu_0 \in \Tan(\bm \nu,x_0) \cap
	\Y^\sing(g_\mu(x_0))$ satisfying the assumptions of Proposition~\ref{prop:singular} (here we are using that $\bm \nu$ satisfies (i)-(iii)). On the other hand, Proposition~\ref{prop: localization regular} yields a full $\Leb^d$-measure set $U \subset \Omega$ with the property that, if $x_0 \in U$, then there exists $\bm \nu_0 \in \Tan(\bm\nu,x_0) \cap \Y^\reg(\ac\mu(x_0))$ satisfying the assumptions of Proposition~\ref{prop:regularr} (here we are using that $\bm \nu$ satisfies (i)-(ii)). The sought assertion then follows from the conclusions of Propositions~\ref{prop:singular} and~\ref{prop:regularr}. This proves the proposition.
\end{proof}
Returning to the proof of Theorem~\ref{thm:char} we observe that since $\bm \nu \in \Y_0(\Omega;W)$ and by assumption (iii) also $\Acal \mu = 0$ in the sense of distributions on $\Omega$, then Theorem~\ref{thm:local} implies that $\bm \nu \in \Y_\Acal(\Omega)$. This proves the sufficiency.

The proof is complete.\qed

\subsection{Proof of Theorem~\ref{thm:char2}} The necessity is a direct consequence of  Theorem~\ref{thm:char} and (a)-(b) in Section~\ref{sec:si}. To prove the sufficiency of (i)-(iii) in Theorem~\ref{thm:char2}, notice that the same (a)-(b) and the sufficiency of Theorem~\ref{thm:char} imply that $\bm \nu \in \Y_\Acal(\Omega)$. Now, we make use of Theorem~\ref{thm:char2}(i) and the last statement of Corollary~\ref{cor:shift}, to deduce that every tangent measure of $\bm \nu$ is a tangent $\Bcal$-gradient Young measure. Therefore, the local characterization Theorem~\ref{thm:local2} implies that $\bm \nu \in \Bcal\!\Y(\Omega)$, as desired. \qed

\appendix

\section{An auxiliary relaxation result}

\begin{theorem}\label{thm:mio}
	Let $f : \Omega \times W\to [0,\infty)$ be a continuous integrand  that is Lipschitz in its second argument, uniformly over the $x$-variable. Assume also that $f$ has linear growth at infinity and is such that there exists a modulus of continuity $\omega$ satisfying
	\begin{equation}\label{eq:mod}
		|f(x,z) - f(y,z)| \le \omega(|x -y|)(1 + |z|) \quad \text{for all $x,y \in \Omega, \; z \in W$.}
	\end{equation}
	Further suppose that the strong recession function $f^\infty$ exists. Then, for the functional 
	\[
	I_f(u) \coloneqq \int_\Omega f(x,u(u)) \dd x, \quad u \in \Lrm^1(\Omega;W),
	\]
	the weak-$*$ (sequential) lower semicontinuous envelope defined by
	\[
	\cl{\Gcal}[\mu] \coloneqq 
	\inf\setB{\liminf_{j \to \infty}  I_f(u_j)}{\{u_j\} \subset \Lrm^1(\Omega;W), u_j \, \Leb^d \toweakstar \mu, \|\Acal u_j\|_{\Wrm^{-k,q}}\to 0},
	\]
	where $\mu \in \M(\cl\Omega;W)$ is an $\Acal$-free measure with and $1 < q < \frac{d}{d-1}$, is given by 
	\[
	\cl{\Gcal}[\mu] = \int_\Omega \Qcal_\Acal f(x,\ac \mu(x)) \dd x + \int_\Omega (\Qcal_\Acal f)^\#(x,g_\mu)\dd |\mu^s|.
	\]
\end{theorem}
\begin{remark}[on the assumptions of Theorem~\ref{thm:mio}]\label{rem:mio} The assumptions on Theorem~\ref{thm:mio} can be relaxed to the following set of assumptions:
	\begin{enumerate}[(a)]
		\item $f \in \E(\Omega \times W)$,
		\item  $\Qcal_{\Acal}f(x,\frarg)$ is globally Lipschitz, uniformly in $x \in \Omega$, and satisfies
		\[
		|\Qcal_{\Acal}f(x,z) - \Qcal_{\Acal}f(y,z)| \le \omega(|x -y|)(1 + |z|) \quad \text{for all $x,y \in \Omega, \; z \in W$.}
		\]
		\item the relaxation $\cl \Gcal$ is restricted to weak-$*$ limit measures satisfying $|\mu|(\partial \Omega) = 0$.
	\end{enumerate}
	Let us briefly comment on how this is done. The proof of~\ref{thm:mio} contained in~\cite{arroyo-rabasa2017lower-semiconti} mainly consists in proving the upper-bound inequality
	\[
	\cl \Gcal[\mu] \le \int_\Omega \Qcal_\Acal f(x,\ac \mu(x)) \dd x + \int_\Omega (\Qcal_\Acal f)^\#(x,\mu^s).
	\] 
	\begin{enumerate}[1.]
		\item First, one proves the relaxation for $\mu = u \Leb^d$ for some $u \in \Crm(\Omega) \cap \Lrm^1(\Omega)$, and for the translation $f_\eps = f + \eps |\frarg|$, where $\eps > 0$ is a small real number. One approximates $u$ by constants $a_Q = u(x_Q)$ on small cubes $Q \subset \Omega$ centered at $x_Q$. The original argument crucially uses the $z$-Lipschitz continuity of $f$ to approximate
		\[
		\int_\Omega f_\eps(x,u) \approx \sum_{Q} \int_Q f_\eps(x_Q,a_Q).
		\] 
		However, the argument also works for $f \in \E(\Omega;W)$ using that $f$ is uniformly continuous on compact subsets of $\Omega \times W$.
		\item At each cube $Q$, a standard homogenization argument is performed with the cell-problem approximation of the definition of $\Acal$-quasiconvexity:
		\[
		\int_Q f_\eps(x_Q,a_Q + w_{j,Q}) \to \int_Q \Qcal_\Acal f_\eps(x_Q,a_Q),
		\]
		where $\{w_{j,Q}\} \subset \Crm^\infty_c(Q;W)\cap \ker \Acal$ are mean-value zero functions on $Q$. In~\cite{arroyo-rabasa2017lower-semiconti}, the fact that $f \ge 0$ is crucial  to show that $\{w_{j,Q}\}$ is $\Lrm^1$-uniformly bounded, therefore also weak-$*$ pre-compact However, (b) is enough to show this (cf. Proposition~\ref{prop:coercive}).
		\item Then, one glues the sequences $w_{j,Q}$ into a global sequence $\{w_j\} \subset \Lrm^1(\Omega)$ satisfying 
		\[
		w_j\Leb^d \toweakstar u \Leb^d \; \text{in $\Mcal(\Omega;W)$} \quad \text{and} \quad \Acal w_j \to 0 \; \text{in $\Wrm^{-k,q}(\Omega)$},
		\]
		so that
		\[
		\cl{\Gcal}[u\Leb^d] \le \lim_{j \to \infty} \int_\Omega f_\eps(x,u + w_j) \approx \sum_Q \int_Q \Qcal_\Acal f_\eps(x_Q,a_Q).
		\] 
		\item The last step consists of showing that one can glue back together the piece-wise constant relaxed energies:
		\[
		\sum_Q \int_Q \Qcal_\Acal f_\eps(x_Q,a_Q) \approx \int_\Omega \Qcal_\Acal f_\eps(x,u).
		\]
		The cited proof relies on the modulus of continuity (b) ---which is proven for $f^\eps$ when $f \ge 0$--- and the fact that, since $f$ is Lipschitz, so it is $\Qcal_\Acal f$ (and their $\eps$-translations as well). With our assumptions, the argument follows by using property (b) instead.
		\item The previous steps proves that 
		\[
		\cl{\Gcal}[u\Leb^d] \le \lim_{j \to \infty} \int_\Omega f(x,u + w_j) \approx  \int_\Omega \Qcal_\Acal f_\eps(x,u).
		\]
		The upper bound inequality then follows by letting $\eps \to 0^+$. 
		\item The upper bound inequality for general $\mu$  follows from  
		Theorem~\ref{lem:app}, Remark~\ref{rem:reg} and Reshetnyak's continuity theorem (cf.~\eqref{eq:resh}).
		\item The lower bound inequality is addressed in~\cite[Theorem~1.2(i)]{arroyo-rabasa2017lower-semiconti}  under the assumption that $f \ge 0$. There, the positivity of $f$ is only used to prevent concentration of negative mass on the boundary. Assumption (c) allows us to dispense with this assumption.
	\end{enumerate}
\end{remark}

\section{Technical lemmas for Radon measures} 
This section is devoted to address some technical results which are essential to the proof of Theorem~\ref{thm:convex}. 
	\begin{figure}[h]
	\begin{tikzpicture}
	[inside/.style={fill=CadetBlue!40,fill opacity=0.85, draw opacity=0},
	outside/.style={fill=SeaGreen!85,fill opacity=0.5,draw =green!80, draw opacity=0.5,dashed,thick}]
	\draw[thick,draw =YellowOrange!50](-.6,1,-2)--(.6,1,-2);
	\draw[thick,draw =YellowOrange!50](-.6,-1,-2)--(.6,-1,-2);
	\draw[thick,draw =YellowOrange!50](-1,-.6,-2)--(-1,.6,-2);
	\draw[thick,draw =YellowOrange!50](1,-.6,-2)--(1,.6,-2);
	\draw[thick,draw =YellowOrange!50](-.6,-2,1)--(.6,-2,1);
	\draw[thick,draw =YellowOrange!50](-.6,-2,-1)--(.6,-2,-1);
	\draw[thick,draw =YellowOrange!50](-1,-2,.6)--(-1,-2,-.6);
	\draw[thick,draw =YellowOrange!50](1,-2,.6)--(1,-2,-.6);
	\draw[thick,draw =YellowOrange!50](-2,-.6,-1)--(-2,.6,-1);
	\draw[thick,draw =YellowOrange!50](-2,-.6,1)--(-2,.6,1);
	\draw[thick,draw =YellowOrange!50](-2,-1,-.6)--(-2,-1,.6);
	\draw[thick,draw =YellowOrange!50](-2,1,-.6)--(-2,1,.6);
	\draw[outside](-1,-1,-2)--(-1,1,-2)--(1,1,-2)--(1,-1,-2)--(-1,-1,-2);
	\draw[outside](-2,-1,-1)--(-2,-1,1)--(-2,1,1)--(-2,1,-1)--(-2,-1,-1);
	\draw[outside](-1,-2,-1)--(1,-2,-1)--(1,-2,1)--(-1,-2,1)--(-1,-2,-1);
	\draw[dashed,draw=black!50,thick](-2,-2,2)--(-2,-2,-2)--(2,-2,-2);
	\draw[dashed,draw=black!50,thick](-2,-2,-2)--(-2,2,-2);
	\draw[inside](2,2,2)--(2,-2,2)--(2,-2,-2)--(2,2,-2)--(2,2,2);
	\draw[inside](2,2,2)--(-2,2,2)--(-2,2,-2)--(2,2,-2)--(2,2,2);
	\draw[inside](2,2,2)--(-2,2,2)--(-2,-2,2)--(2,-2,2)--(2,2,2);
	\draw[dashed,draw=black!50,thick](2,2,2)--(2,-2,2)--(2,-2,-2)--(2,2,-2)--(2,2,2);
	\draw[dashed,draw=black!50,thick](2,2,2)--(-2,2,2)--(-2,2,-2)--(2,2,-2);
	\draw[dashed,draw=black!50,thick](2,-2,2)--(-2,-2,2)--(-2,2,2);
	\draw[outside,fill opacity=0.7](-1,-1,2)--(-1,1,2)--(1,1,2)--(1,-1,2)--(-1,-1,2);
	\draw[outside,fill opacity=0.7](2,-1,-1)--(2,-1,1)--(2,1,1)--(2,1,-1)--(2,-1,-1);
	\draw[outside,fill opacity=0.7](-1,2,-1)--(1,2,-1)--(1,2,1)--(-1,2,1)--(-1,2,-1);
	\draw[thick,draw =YellowOrange!50](-.6,1,2)--(.6,1,2);
	\draw[thick,draw =YellowOrange!50](-.6,-1,2)--(.6,-1,2);
	\draw[thick,draw =YellowOrange!50](-1,-.6,2)--(-1,.6,2);
	\draw[thick,draw =YellowOrange!50](1,-.6,2)--(1,.6,2);
	\draw[thick,draw =YellowOrange!50](-.6,2,1)--(.6,2,1);
	\draw[thick,draw =YellowOrange!50](-.6,2,-1)--(.6,2,-1);
	\draw[thick,draw =YellowOrange!50](-1,2,.6)--(-1,2,-.6);
	\draw[thick,draw =YellowOrange!50](1,2,.6)--(1,2,-.6);
	\draw[thick,draw =YellowOrange!50](2,-.6,-1)--(2,.6,-1);
	\draw[thick,draw =YellowOrange!50](2,-.6,1)--(2,.6,1);
	\draw[thick,draw =YellowOrange!50](2,-1,-.6)--(2,-1,.6);
	\draw[thick,draw =YellowOrange!50](2,1,-.6)--(2,1,.6);
	\draw(1.75,1.75,2) node  {$S_1$};
	\draw(0,0,2) node  {$S_2$};
	\draw(0,1.25,2) node  {$S_3$};
	\draw(2.7,-1.7,2.3) node  {$D$};
	\end{tikzpicture}
	\caption{Generic shape of the approximation set $D$ from Lemma~\ref{lem:cubo} ($d = 3$); composed by the first open cube approximation $S_1 \subset Q$, the second-step  $2$-dimensional relatively open caps $S_2 \subset \partial S_1$, and the last-step $1$-dimensional relatively open caps $S_3 \subset \partial S_2$.}
	\label{fig:convex}
\end{figure}
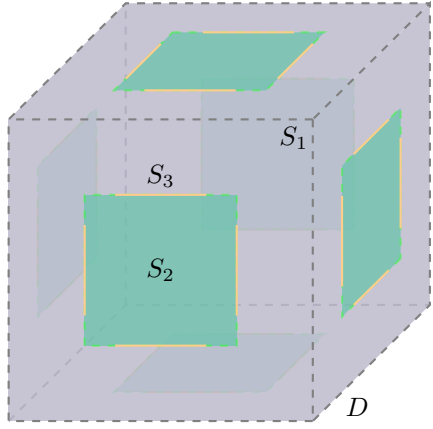

\begin{lemma}\label{lem:cubo}
	Let $\lambda$ be a probability measure on the unit open cube $Q \subset \R^d$ and assume that $\lambda$ does not charge points in $Q$, that is, $\lambda(\{x\}) = 0$ for all $x \in Q$. Let $\theta \in (0,1)$, then  there exists a convex Borel set $D \subset Q$ satisfying
	\[
		Q_r \subset D \subset \cl{Q_r} \quad \text{for some $r \in (0,1)$},
	\]
	and 
	\[
	\lambda(C) = \theta.
	\] 
\end{lemma}
\begin{proof}
	In the case when $d = 1$ we define the monotone non-decreasing map
	\[
	r \mapsto  \phi(r) \coloneqq \int_{-r}^r  \dd \lambda, \quad r \in (0,1), 
	\] 
	which in particular is a function of bounded variation. Notice that, in this case, the one dimensional $\BV$-theory and the assumption on $\lambda$ give 
	\[
		\phi'(\{r\}) = \phi(r^+) - \phi(r^-) = \lambda(\{r\}) = 0.
	\]
	This proves $\phi$ is in fact continuous in the interval $(0,1)$. Therefore, since $\phi(0) = 0$ (again by assumption) and $\phi(1) = 1$, then the Mean Value Theorem ensures there exists $\cl r \in (0,1)$ with 
	$\lambda(Q_r) = \phi(r) = \theta$.   
	
	The case $d > 1$ requires a co-area-type argument. The first step is to define the map
	\[
	\phi(r) \coloneqq \int_{Q_r} \dd \lambda, \quad r \in (0,1),
	\] 
	is monotone non-decreasing and satisfies $\lim_{r \todown 0} \phi(r) = 0$ and $\lim_{r \toup 1} \phi(r) = 1$. 
	In particular, 
	\[
	|\phi'|(\{r\}) = \phi(r^+) - \phi(r^-) = \lambda(\partial Q_r) \quad \text{for all $r \in (0,1)$.}
	\]
	Set $r \coloneqq \sup\set{s \in (0,1)}{\phi(s) \le \theta}$. Clearly, if $\phi$ is lower semicontinuous at $r$, then we can set $C_1 = Q_r$ which automatically satisfies the conclusions of the Lemma (in this case is not necessary to use the assumption that $\lambda$ does not charge points). However, in general one cannot expect this as for instance there might be mass sitting on $\lambda(\partial Q_r)$. We may then assume that $\theta_1 \coloneqq \lambda(Q_r) = \phi(r^-) \le \theta$ and
	\begin{equation}\label{eq:finish}
	0 \le \theta - \theta_1 \le \phi(r^+) - \phi(r^-).
	\end{equation}
	This will be our first approximation. The second step to carry the same approximation now on $\{q_r^1,\dots,q_r^{2d}\}$, the $(2d)$ open $(d-1)$-dimensional faces of $\partial Q_r$, which are axis-directional translated $(d-1)$-dimensional cubes (centered at the origin) in $\R^{d-1}$; the number $(2d)$ of faces will not be relevant for our construction. Define the maps 
	\[
	\phi_{i,2}(s) \coloneqq \int_{q_s^i} \dd \lambda, \quad s \in (0,1), \; i \in \{1,\dots,2d\}.
	\]
	Repeating the same procedure as in the first step  on each of the faces simultaneously, we update the error of our estimate to 
	\[
	0 \le \theta - \theta_2 \coloneqq \theta - \theta_1 - \sum_{i = 1}^{2d} \phi_{i,2}(r_1^-),
	\]
	where $r_1  \coloneqq \sup\setn{s \in (0,r)}{\sum_{i = 1}^{2d} \phi_{i,2}(s) \le \theta - \phi(r^-)}$. Notice that 
	\[
	\theta_2 = \lambda(C_2),
	\]
	where
	\[
		C_2 \coloneqq Q_r \cup q_{r_1}^1 \cup \cdots \cup q_{r_1}^{2d}.
	\]
	Moreover, by construction (since each one of the added faces are concentric to each face of $Q_r$), $C_2$ is a (semi-open/semi-closed) convex set satisfying $Q_r \subset C_2 \subset \cl{Q_r}$. There are two cases, either $\theta = \theta_2$ and then $\lambda(C_1) = \theta$, or $\theta - \theta_2 > 0$ and we must keep adding bits of $\partial Q_r$, which may require to perform a similar argument on the $(d-2)$-dimensional concentric faces of $\partial Q_r$ and the $(d-2)$-dimensional faces of each $q_s^i$, $1 \le i \le 2d$. In general, the $(j+1)$th step is to iterate this argument (when $\theta - \theta_j > 0$) on all the possible $(d-j)$-dimensional faces of $\partial Q_r$ and all the other $(d-j)$-dimensional faces resulting of adding the previous $(d-j +1)$-dimensional concentric cubical caps. The key part of the construction is that, at the end of the $(j+1)$th step, one obtains a convex set $C_{j} \subset C_{j+1}$ satisfying $Q_r \subset C_{j+1} \subset \cl{Q_r}$ and 
	\[
	\text{$0 \le \theta -  \lambda(C_{j+1}) \eqqcolon \theta - \theta_{j+1}$ \quad for some \quad $\theta_{j+1} \ge \theta_j$}.
	\]   
	We now argue why there exists $1 \le j \le d$ such that $\theta_j = \theta$.  
	If $d =2$, then by the same argument that in the one-dimensional case we get that the maps $\phi_{i,2}$ are continuous, and hence it must be that $\theta_2$ reaches $\theta$. If $d = 3$, then the maps $\phi_{i,3}$ of the third step are continuous and hence at most $\theta_3$ reaches $\theta$. In general, the description of this procedure is tedious, but it is inductively natural and always reaches and endpoint (at most after $d$-steps) where we find a convex set $C$ containing the origin and satisfying $Q_r \subset C \subset \cl{Q_r}$ and
	\[
	\lambda(C) = \theta.
	\] 
	This finishes the proof.
\end{proof}

\begin{corollary}\label{cor:inner}
		Let $\lambda$ be a probability measure on the unit open cube $Q \subset \R^d$ and assume that $\lambda$ does not charge points in $Q$. Let $\theta \in (0,1)$ and let $\eps > 0$. Then, there exists an open Lipschitz set $D \subset Q$ satisfying
		\[
			\lambda(\partial D) = 0, \quad \text{and} \quad |\lambda(D) - \theta| \le \eps.
		\]
\end{corollary}
\begin{proof}
	From the previous lemma we may find a set $D$ satisfying $\lambda(D) = \theta$. Now, from the inner and outer regularity of Radon measures we may find a compact set $K$ and an open set $O$ such that $K \subset D \subset O$ satisfying
	\[
		0 \le \lambda(D) - \lambda(K) < \frac \eps2 \quad \text{and} \quad 0 \le \lambda(O) - \lambda(D) < \frac \eps2.
	\]
	Moreover, since $\dist(K,O^\complement) > 0$, there exists a Lipschitz open set $K \subset C \subset O$ with $\rho \coloneqq \dist(C,K \cup O^\complement) > 0$. On the other hand, since ${C}$ is a Lipschitz compact set, the family of sets 
	\[
		T_\delta = \set{x \in Q}{\dist(x,C) < \delta}, \quad \delta > 0,
	\] 
	is a a family of open Lipschitz sets satisfying $K \subset T_\delta \subset O$ for some $\delta \in [0,\delta_1]$, where $0 < \delta_1 \ll \rho$. Furthermore, if $0 < s < t <\delta_1$, then
	$\partial T_s \cap \partial T_t = \varnothing$. In particular, since $\lambda$ is a Radon measure, there exists a full $\Leb^1$-measure subset of $I \subset [0,\delta_1]$ such that
	\[
		\lambda(\partial T_\delta) = 0 \quad \text{for all $\delta \in I$.}
	\]
	Let us choose $\delta_0 \in I$ and recall from our construction that $K \subset T_{\delta_0} \subset O$. Hence,
	\[
		0 \le |\lambda(T_{\delta_0}) - \lambda(D)| \le  \lambda(O) - \lambda(K) < \eps.
	\]
	This finishes the proof.
\end{proof}
	
\begin{lemma}[shrinking sequence]\label{lem:shrink} Let $\lambda$ be a positive Radon measure on $\Omega$ and assume there exists a tangent measure $\tau \in \Tan(\lambda,x)$ which does not charge points on $\R^d$. Then,
for every $\theta \in (0,1)$, there exists an  infinitesimal sequence $r_m \todown 0$ and a sequence of open Lipschitz sets $R_m \subset Q_{r_m}(x)$ satisfying
\[
	\lambda(\partial R_m) = 0 \quad \text{for all $m \in \Nbb$,}
\]
and
\[
	\lim_{m \to \infty} \frac{\lambda(R_m)}{\lambda(Q_{r_m}(x))} = \theta. 
\]
\end{lemma}
\begin{proof}
	Since $\Trm_{0,r} [\tau]$ belongs to $\Tan(\lambda,x)$ for all $r > 0$, we may assume without any loss of generality that $\tau$ is a good blow-up as in~\eqref{eq:goodbu}. That is, there exists an infinitesimal sequence $r_j \todown 0$ such that
\begin{equation*}
\frac{1}{\lambda({Q_{r_j}(x)})} \, \Trm_{x,r_j}[\lambda] \toweakstar \tau \;\; \text{in $\M(\R^d;W)$},  \qquad  |\tau|(Q) = |\tau|(\cl Q) = 1.
\end{equation*}
By assumption and Corollary~\ref{cor:inner} we may find a sequence of Lipschitz open sets $(D_m)_{m \in \Nbb}$  satisfying $\subset D_m \subset Q$ for all $m \in \Nbb$. Moreover,
\[
	\tau(\partial D_m) = 0 \quad \text{and} \quad 
	|\tau(D_m) - \theta| \le \frac 1m.
\]
Hence, for fixed $m$, we deduce from the strict-convergence of the blow-up sequence that
\[
	\lim_{j \to \infty} \frac{\lambda(x + r_jD_m)}{\lambda({Q_{r_j}(x)})}  = \tau(D_m) = \theta +  \BigO(m^{-1}).
\]
Moreover, up to a small re-scaling at each $m$ we may assume without loss of generality that $\lambda(x + r_j\partial D_m) = 0$.
A standard diagonalization argument yields a subsequence $r_m \coloneqq R_m \todown 0$ such that
\[
	\lim_{m \to \infty} \frac{\lambda(R_{m})}{\lambda(Q_{r_{m}}(x))} = \theta, \qquad R_m \coloneqq x + r_{m}D_m \subset Q_{r_{m}}(x).
\]
By construction, the sequence of sets $(R_m)_{m \in \Nbb}$ has the desired properties. 
\end{proof}

\begin{corollary}\label{cor:theta}
	 Let $\lambda$ be a positive Radon measure on $\Omega$ and assume there exists a tangent measure $\tau \in \Tan(\lambda,x)$ which does not charge points. Let $f$ be a $\lambda$-measurable map and assume furthermore that $x$ is a $\lambda$-Lebesgue point of $f$. Then, for every $\theta \in (0,1)$, there exist a sequence $r_m \todown 0$ and a sequence of Lipschitz open sets $D_m \subset Q_{r_m}$ satisfying
	\[
		\lambda(\partial D_m) = 0, \qquad \lim_{m \to \infty} \, \frac{\lambda(x + D_m)}{\lambda(x + Q_{r_m})} = \theta,
	\]
	and
	\[
		\lim_{m \to \infty} \aveint{x + D_m}{} |f(y) - f(x)| \dd \lambda(y) = 0. 
	\]  
\end{corollary}
	\begin{proof}
		The existence of the sequence of open Lipschitz sets $(D_m)$ satisfying the first two properties follows directly from the previous corollary. The third property follows from the estimate
		\[
			\aveint{x + D_m}{} |f(y) - f(x)| \dd \lambda(y) \le \frac{1}{\theta + \BigO(m^{-1})}\cdot \aveint{Q_{r_m}(x)}{} |f(y) - f(x)| \dd \lambda(y),
		\]
		and the fact that $x$ is a $\lambda$-Lebesgue point of $f$. 
	\end{proof}

\begin{lemma}\label{lem:tantan}
	Let $\lambda$ be a positive Radon measure on $\Omega$ and assume there exists $x \in \Omega$ such that $\lambda(\{x\}) > 0$. Then $\Tan(\lambda,x) = \set{c\delta_0}{c>0}$. 
\end{lemma}
\begin{proof}
	Let $\tau \in \Tan(\lambda,x)$ and set $0 < \alpha = \lambda(\{x\})$. Since $\Tan(\lambda,x)$ is a $d$-cone, it is enough to show that $\tau = \delta_0$ when $\tau$ is a probability measure on $Q$. Moreover, we may also assume the blow-up sequence converging to $\tau$ has the form
	\[
		\gamma_j = \frac{1}{\lambda({Q_{r_j}(x)})} \, \Trm_{x,r_j}[\lambda] \toweakstar \tau \;\; \text{in $\M(\R^d;W)$}.
	\] 
	It follows from the strict-convergence $\gamma \to \tau$ on $Q$ that
	\[
		\tau(Q_s) = \lim_{j \to \infty} \gamma_j(Q_s) = \frac 1\alpha \lim_{j \to \infty} \lambda(Q_{sr_j}(x)) = 1 \quad \text{for all $s > 0$.}
	\]
	Since $\tau$ is a probability measure on $Q$, this shows that $\tau \mres Q = \delta_0$ as desired.
\end{proof}

%

\subsection*{Acknowledgments} I would like to thank Bogdan Raita for fruitful conversations regarding the contents of Section~5. Special thanks go to Anna Skorobogatova, who kindly helped me to proofread the paper. Also I want to thank the reviewers for suggesting me to address relevant questions that were not contained in the original version of this paper.
This project has received funding from the European Research Council (ERC) under the European Union's Horizon 2020 research and innovation programme, grant agreement No 757254 (SINGULARITY).


\end{document}